\documentclass[a4paper,11pt]{article}
\usepackage[left=25mm,right=25mm,top=24mm,bottom=25mm,headheight=15pt,headsep=7mm]{geometry}
\usepackage{amsmath,amssymb,amsthm,mathtools,mathrsfs}
\usepackage{booktabs,array,tabularx,longtable}
\usepackage[font=small,labelfont=bf]{caption}
\usepackage{enumitem,fancyhdr,needspace,tikz-cd}
\usepackage[unicode,hidelinks]{hyperref}
\usepackage{bookmark}
\hypersetup{pdftitle={Nonvanishing of 2-primary Kolyvagin classes via simultaneous quadratic twists},pdfauthor={Zhiqi Sun},pdfsubject={Manin constants, quadratic twists, and nonvanishing of 2-primary Kolyvagin classes}}
\usepackage[T1]{fontenc}
\usepackage{lmodern}
\setlist[enumerate]{leftmargin=2.4em,itemsep=.3em,topsep=.35em}
\setlist[itemize]{leftmargin=2em,itemsep=.2em,topsep=.3em}

\allowdisplaybreaks[2]
\numberwithin{equation}{section}
\theoremstyle{plain}
\newtheorem{theorem}{Theorem}[section]
\newtheorem{proposition}[theorem]{Proposition}
\newtheorem{lemma}[theorem]{Lemma}
\newtheorem{corollary}[theorem]{Corollary}
\theoremstyle{definition}
\newtheorem{definition}[theorem]{Definition}
\newtheorem*{definition*}{Definition}
\newtheorem*{conjecture}{Conjecture}
\theoremstyle{remark}
\newtheorem{remark}[theorem]{Remark}

\newcommand{\Q}{\mathbf Q}
\newcommand{\Z}{\mathbf Z}
\newcommand{\F}{\mathbf F}
\newcommand{\R}{\mathbf R}
\newcommand{\C}{\mathbf C}

\DeclareMathOperator{\Gal}{Gal}
\DeclareMathOperator{\Sel}{Sel}
\DeclareMathOperator{\Sha}{Sha}
\DeclareMathOperator{\Cl}{Cl}
\DeclareMathOperator{\Tr}{Tr}

\DeclareMathOperator{\Norm}{Norm}
\DeclareMathOperator{\res}{res}
\DeclareMathOperator{\cor}{cor}
\DeclareMathOperator{\loc}{loc}

\DeclareMathOperator{\im}{im}
\DeclareMathOperator{\Hom}{Hom}
\DeclareMathOperator{\End}{End}
\DeclareMathOperator{\GL}{GL}
\DeclareMathOperator{\rank}{rank}
\DeclareMathOperator{\corank}{corank}
\DeclareMathOperator{\ord}{ord}
\DeclareMathOperator{\tr}{tr}
\DeclareMathOperator{\Reg}{Reg}
\DeclareMathOperator{\Frob}{Frob}
\DeclareMathOperator{\disc}{disc}

\newcommand{\tors}{\mathrm{tors}}
\newcommand{\std}{\mathrm{std}}
\newcommand{\ur}{\mathrm{ur}}
\newcommand{\fin}{\mathrm f}
\newcommand{\trans}{\mathrm{tr}}
\newcommand{\intr}{\mathfrak i}
\newcommand{\fstr}{f_{\mathrm{str}}}
\newcommand{\LC}{\textup{(LC2)}}
\newcommand{\BSD}{\textup{(BSD2)}}
\newcommand{\Heegner}{\textup{(heegner)}}
\newcommand{\Surjectivity}[1]{\ensuremath{\text{\normalfont(surjectivity)}_{#1}}}
\newcommand{\Exclusion}{\textup{(exclusion)}}

\title{Nonvanishing of $2$-primary Kolyvagin classes\\via simultaneous quadratic twists}
\author{Zhiqi Sun}
\date{October 20, 2026}
\begin{document}
\maketitle
\thispagestyle{plain}
\begin{abstract}
Let $E/\Q$ be an elliptic curve, let $K=\Q(\sqrt D)$ satisfy the Heegner hypothesis with $D<-4$, and fix a modular parametrization $\varphi$ with $\varphi(\infty)=O$. Under surjectivity modulo $2^L$, with $L\ge2$, and a cohomological exclusion condition at level $4$, we construct a finite set of primes $T$ such that, writing $d_T=\prod_{q\in T}(-q)$ and $n_T=\prod_{q\in T}q$, the $2$-Selmer dimensions of $E^{d_T}$ and $E^{d_TD}$ are $0$ and $1$. The support size is bounded in terms of the initial Selmer dimensions, independently of $L$. Inspired by the work of Tian \cite{TianCongruent} and Tian--Yuan--Zhang \cite{TYZGenus} on congruent numbers, we construct a point $Q_T$ on the terminal twist and identify its Kummer class with the standard mod-$2$ Kolyvagin class. Using the low-corank $2$-converse and the $2$-primary Birch--Swinnerton-Dyer formula, we prove that its index $m_T$ in the full Mordell--Weil lattice satisfies
\[
v_2(m_T)=v_2(c_\varphi)+\intr(E,K).
\]
The optimal Manin constant theorem \cite{CLN} supplies a $2$-indivisible parametrization $\varphi_0$, for which
\[
\kappa^{\std}_{n_T,1}(\varphi_0)\ne0
\quad\Longleftrightarrow\quad
\intr(E,K)=0.
\]
When the intrinsic index vanishes, the minimum terminal support size is $R=\min\{d_2(E),d_2(E^D)\}$, and the normalized standard classes have order $2^M$ for $1\le M\le L$. Full $2$-adic Galois image then gives strong nonvanishing.
\end{abstract}
\par\medskip
\clearpage
\begingroup
\small\linespread{1.02}\selectfont
\setlength{\parskip}{0pt}
\tableofcontents
\endgroup
\clearpage

\section{Introduction and main results}\label{sec:intro}
\subsection{Kolyvagin's conjecture and notation}
Kolyvagin constructed Galois cohomology classes from Heegner points over ring class fields and used them to study the Selmer groups of elliptic curves. His nonvanishing conjectures extend the non-torsion assertion for the basic Heegner point to the entire system of derivative classes. Conjecture A of \cite{KolyvaginPreprint} concerns primes in $B(E)$; the stronger arbitrary-prime formulation is Conjecture~1 on printed p.~11. For the construction and further formulations, see \cite[\S\S4--5]{GrossKolyvagin} and \cite[Conjecture 1.1, \S5]{Stein}.

Let $E/\Q$ be an elliptic curve of conductor $N$ and minimal discriminant $\Delta_E$. Write $w(E)\in\{\pm1\}$ for its global root number, normalized by
\[
\Lambda(E,s)=w(E)\Lambda(E,2-s),\qquad
\Lambda(E,s)=N^{s/2}(2\pi)^{-s}\Gamma(s)L(E,s).
\]
Fix a nonconstant modular parametrization over $\Q$ associated with its normalized newform $f_E$,
\[
\varphi:X_0(N)\longrightarrow E,\qquad \varphi(\infty)=O.
\]
Let $K=\Q(\sqrt D)$. We say that $(E,K)$ satisfies \Heegner\ if $D<-4$ is a fundamental discriminant and
\begin{equation}\label{eq:heegner}
(D,N)=1,\qquad \text{every }\ell\mid N\text{ splits in }K.
\end{equation}
We assume \Heegner\ throughout. In particular, $K$ is imaginary quadratic and $\mathcal O_K^\times=\{\pm1\}$. If $2$ does not split in $K$, \Heegner\ forces $E$ to have good reduction at $2$.

Let $\mathcal E/\Z$ be the N\'eron model of $E$, with identity component $\mathcal E^0$ and special fiber $\widetilde E_\ell=\mathcal E\times_{\Z}\F_\ell$ at a prime $\ell$. For $\ell\nmid N$, the fiber $\widetilde E_\ell$ is an elliptic curve over $\F_\ell$. Write $\mathcal E^{(d)}/\Z$ for the N\'eron model of the quadratic twist $E^d$.

Put $V=E[2]$ and assume that its Galois image is $\GL_2(\F_2)\simeq S_3$. For a prime $q$ of good reduction, write
\[
a_q(E)=q+1-\#\mathcal E(\F_q).
\]
\begin{definition*}
For $M\ge1$, a prime $q\nmid2ND$ is \emph{admissible at depth $M$} if it is inert in $K$ and $2^M$ divides both $q+1$ and $a_q(E)$. Define the set of admissible conductors at depth $M$ by
\begin{equation}\label{eq:NM}
\mathcal N_M(E,K)=\left\{\begin{array}{l|l}
n\ge1&\begin{array}{l}
n\text{ is squarefree; for every }q\mid n,\ q\nmid2ND,\\ 
q\text{ is inert in }K,\text{ and }2^M\mid q+1,\ a_q(E)
\end{array}
\end{array}\right\}.
\end{equation}
The empty product $n=1$ is included. For a squarefree integer $n$, its \emph{support} and \emph{support size} are
\[
\operatorname{Supp}(n)=\{q:q\mid n\},\qquad
\nu(n)=\#\operatorname{Supp}(n).
\]
For a finite set of primes $T$, put $n_T=\prod_{q\in T}q$. Its support size is $|T|=\nu(n_T)$; the empty support has size zero.
\end{definition*}

The admissibility conditions arise from the derivative construction. Let $H_n$ be the ring class field of conductor $n$, and let $y_n\in E(H_n)$ be the Heegner point obtained from $\varphi$. For $q\mid n$, the corresponding relative Galois group is cyclic of order $q+1$. Choose a generator $\sigma_q$ and put
\[
D_q=\sum_{j=0}^{q}j\sigma_q^j,\qquad
D_n=\prod_{q\mid n}D_q.
\]
The Heegner points satisfy $\Tr_{H_n/H_{n/q}}y_n=a_q(E)y_{n/q}$, while
$(\sigma_q-1)D_q=(q+1)-\sum_{j=0}^{q}\sigma_q^j$.
The two divisibility conditions therefore make $D_ny_n$ invariant modulo $2^M$ under the relative group. Summing over representatives for $\Gal(H_1/K)$, taking a Kummer class, and descending gives the standard class \cite[Proposition 4.2, \S\S5.1--5.2]{Stein}
\[
\kappa^{\std}_{n,M}(\varphi)\in H^1(K,E[2^M]).
\]
Section~\ref{sec:cm} gives the full construction and its compatibility with changes of coefficients. For $n=1$, the class is the Kummer class of the basic trace point $y_K=\Tr_{H_1/K}y_1$.

\begin{conjecture}
For the fixed triple $(E,K,\varphi)$ above, the standard system satisfies
\begin{equation}\label{eq:conjecture}
\exists M\ge1\quad\exists n\in\mathcal N_M(E,K):\quad
\kappa^{\std}_{n,M}(\varphi)\ne0.
\end{equation}
\end{conjecture}

Both $M$ and $n$ may vary. Nonvanishing modulo $2$ means that a nonzero class already occurs at $M=1$, and is stronger than the assertion above. If $y_K$ is non-torsion, its divisibility by $2$ in the finitely generated group $E(K)$ is bounded, so $n=1$ gives a nonzero class at some depth. The Gross--Zagier formula relates the analytic rank-one case to this basic trace point \cite{GZ}; more general cases require derivative classes with $n>1$.

Kolyvagin's \emph{strong nonvanishing} formulation for arbitrary primes asks for nonzero classes with a fixed support size as the admissibility congruences become arbitrarily deep \cite[printed pp.~10--11, equation (2) and Conjecture 1]{KolyvaginPreprint}. We state this condition in Section~\ref{subsec:main-strong}.

\subsection{The intrinsic index and two low-corank results}\label{subsec:inputs}
For an elliptic curve $C/\Q$, write
\[
d_2(C)=\dim_{\F_2}\Sel_2(C/\Q),\qquad
s_2(C)=\corank_{\Z_2}\Sel_{2^\infty}(C/\Q).
\]
These are, respectively, the dimension of the finite Selmer group and the corank of the divisible part of the infinite Selmer group. We normalize $v_2$ by $v_2(2)=1$. Every Galois trace is the sum of conjugates, without division by the degree of the extension.

Choose a minimal N\'eron differential $\omega_E$. Define the Manin factor of $\varphi$ by
\[
\varphi^*\omega_E=2\pi i\,c_\varphi f_E(\tau)\,d\tau,\qquad \tau\in\mathfrak H.
\]
Let $E_0(\Q_\ell)=\mathcal E^0(\Z_\ell)\subset\mathcal E(\Z_\ell)=E(\Q_\ell)$ and $c_\ell(E)=[E(\Q_\ell):E_0(\Q_\ell)]$. Put
\begin{align}
\epsilon_\infty&=\log_2\#\pi_0(E(\R)),&
T_E&=\sum_{\ell\mid N}v_2(c_\ell(E)),\label{eq:TE}\\
t_r&=\dim_{\F_2}\mathcal E(\F_r)[2]\quad(r\mid D,\ r\text{ odd}),&
h_2(E,K)&=\dim_{\F_2}\frac{E(\Q_2)}{N_{K\otimes\Q_2/\Q_2}E(K\otimes\Q_2)}.\nonumber
\end{align}
For a split quadratic algebra, the norm adds the two point-group coordinates. Since the norm image contains $2E(\Q_2)$, the quotient defining $h_2(E,K)$ is a finite-dimensional $\F_2$-vector space. Set
\begin{equation}\label{eq:index-definition}
\delta_K=\frac{\epsilon_\infty-1+h_2(E,K)+\sum_{r\mid D,\ r\text{ odd}}t_r}{2},\qquad
\intr(E,K)=T_E+\delta_K.
\end{equation}
We suppress the fixed curve $E$ in the notation $\delta_K$. Section~\ref{subsec:parity} proves that $\delta_K\in\Z_{\ge0}$ and that $1+2\delta_K$ is the sum of the local codimensions between the Kummer conditions of $E$ and $E^D$. Thus the \emph{intrinsic $2$-adic index} $\intr(E,K)$ vanishes exactly when all original Tamagawa numbers are odd and the two sets of local conditions have total discrepancy one. In this case the initial Selmer groups are nested with dimensions differing by one, which yields the sharp support bound in the descent.

Throughout the paper, we use the low-corank $2$-converse \cite[Theorem 1.1]{OAIconverse} and the two-primary Birch--Swinnerton-Dyer formula \cite[Theorem 1.1]{OAIbsd}. For an elliptic curve $C/\Q$ with $s_2(C)\in\{0,1\}$, they give:
\begin{description}[leftmargin=1.5cm,labelwidth=1.3cm,style=nextline]
\item[\LC] The analytic and algebraic ranks of $C$ both equal $s_2(C)$, and $\Sha(C/\Q)$ is finite.
\item[\BSD] Put $r=s_2(C)=\rank C(\Q)$. With the total real period and the regulator of the full Mordell--Weil lattice specified in Section~\ref{subsec:height}, the normalized leading term is a nonzero rational number and
\begin{equation}\label{eq:BSD2-input}
v_2\!\left(\frac{L^{(r)}(C,1)}{r!\,\Omega_C\Reg_C}\right)
=v_2\!\left(\frac{\#\Sha(C/\Q)\prod_\ell c_\ell(C)}{\#C(\Q)_{\tors}^{\,2}}\right).
\end{equation}
\end{description}
The definitions in \cite[\S1.1]{OAIbsd} use the total real period and the pairing with diagonal $H=2\widehat h_{[O]}$ on the full Mordell--Weil lattice, as in Section~\ref{subsec:height} below.

\paragraph{Manin constants and divisibility of the parametrization.}
To state the dependence on the fixed parametrization, define
\begin{equation}\label{eq:bphi}
b_\varphi=\max\left\{b\in\Z_{\ge0}:\ \begin{gathered}
\varphi=[2^b]\circ\psi,\quad \psi(\infty)=O,\\ 
\psi:X_0(N)\to E\text{ is defined over }\Q
\end{gathered}\right\}.
\end{equation}
The maximum exists, since any such factorization implies $4^b\mid\deg\varphi$. We call $\varphi$ \emph{$2$-indivisible} if $b_\varphi=0$. Section~\ref{sec:normalization} uses the optimal-quotient Manin constant theorem \cite{CLN} to prove that, under residual irreducibility, $b_\varphi=v_2(c_\varphi)$ and $\varphi=[2^{b_\varphi}]\circ\varphi_0$ for a parametrization $\varphi_0$ of $E$ with odd Manin factor.

\subsection{Simultaneous twists and a criterion for nonvanishing}\label{subsec:main-finite}
Fix $L\ge2$ and write $\Gamma_L=\GL_2(\Z/2^L\Z)$.
\begin{definition*}
The curve $E$ satisfies \Surjectivity{L} if
\begin{equation}\label{eq:full-image}
\rho_{E,2^L}(G_\Q)=\Gamma_L.
\end{equation}
\end{definition*}
This condition depends only on $E$ and is imposed separately from \Heegner.

Under \Surjectivity{2}, let $\xi_E$ denote the unique nonzero class in
\[
\ker\bigl(H^1(\Q,V)\longrightarrow H^1(\Q(E[4]),V)\bigr).
\]
Section~\ref{subsec:xi} proves its existence and uniqueness.

\begin{definition*}
For a pair $(E,K)$ satisfying \Heegner\ and \Surjectivity{2}, the condition \Exclusion\ is
\begin{equation}\label{eq:exclusion}
\xi_E\notin\Sel_2(E/\Q)\cap\Sel_2(E^D/\Q).
\end{equation}
\end{definition*}
For $L\ge2$, \Surjectivity{L} implies \Surjectivity{2} and residual $S_3$ image, so the same class $\xi_E$ is used at every such depth. Condition \Exclusion\ is automatic when $\Delta_E>0$, by the real local condition. By Lemma~\ref{lem:mod-eight}, \Surjectivity{3} is equivalent to full $2$-adic image, to \Surjectivity{L} for every $L\ge2$, and to \Surjectivity{L} at unbounded depths.

Put
\[
a_i=d_2(E^{D^i})\qquad(i=0,1).
\]
Fix the prime family $\mathcal P_L$ of Definition~\ref{def:PL}. For finite $T\subset\mathcal P_L$, put
\[
n_T=\prod_{q\in T}q,\qquad d_T=\prod_{q\in T}(-q).
\]
Let $\mathscr A_L$ consist of all such supports, with the restriction that $|T|$ is even when $\Delta_E>0$ and $\intr(E,K)>0$. Define the terminal supports by
\begin{equation}\label{eq:terminal-family}
\mathscr T_L=\left\{T\in\mathscr A_L:
\{d_2(E^{d_T}),d_2(E^{d_TD})\}=\{0,1\}\right\}.
\end{equation}
The parity restriction makes $d_T$ a square at every finite place dividing $2ND$. Under intrinsic index zero, Proposition~\ref{prop:actual-pair} also permits odd supports when $\Delta_E>0$.

\begin{theorem}\label{thm:main-descent}
Assume \Heegner, \Surjectivity{L}, and \Exclusion. There is a terminal support $T\in\mathscr T_L$ with
\[
|T|\le\max(a_0,a_1),\qquad n_T\in\mathcal N_L(E,K).
\]
The support may avoid any prescribed finite set of primes. There are infinitely many nonempty terminal supports of a common size at most $\max\{a_0,a_1,2\}$.

If $\intr(E,K)=0$, put
\begin{equation}\label{eq:R0}
R=\min(a_0,a_1)=\max(a_0,a_1)-1.
\end{equation}
Then
\begin{equation}\label{eq:optimal-terminal}
\min_{T\in\mathscr T_L}|T|=R.
\end{equation}
When $R>0$, infinitely many supports attain this minimum and may avoid any prescribed finite set. When $R=0$, the empty support is terminal, and infinitely many terminal supports have size two.
\end{theorem}

The general descent is proved in Proposition~\ref{prop:general-terminal}; Proposition~\ref{prop:terminal} gives the sharper statement at intrinsic index zero. These arguments use finite Selmer groups and local Kummer conditions. The results \LC\ and \BSD\ enter in the following nonvanishing criterion.

For an explicit application, assume \Heegner, \Surjectivity{L}, odd $D$, and odd Tamagawa numbers for $E$. Then $\intr(E,K)=0$ and \Exclusion\ hold in either of the following cases:
\begin{enumerate}
\item $\Delta_E>0$, and Frobenius acts as a $3$-cycle on $E[2]$ at every prime dividing $D$;
\item $\Delta_E<0$, exactly one prime $r_0\mid D$ has transposition Frobenius, $a_{r_0}(E)\equiv2\pmod4$, and Frobenius is a $3$-cycle at every other prime dividing $D$.
\end{enumerate}
Indeed, $h_2(E,K)=0$ for odd $D$, so \eqref{eq:index-definition} gives $\intr(E,K)=0$. Exclusion follows from the real calculation in Section~\ref{subsec:xi} in the first case and Proposition~\ref{prop:odd-exclusion} in the second. In the first case with $L=2$, residual surjectivity already implies \Surjectivity{2}, as shown in Section~\ref{subsec:xi}.

\begin{theorem}\label{thm:main-finite}
Assume \Heegner, \Surjectivity{L}, and \Exclusion. Let $\varphi_0$ be a $2$-indivisible modular parametrization of $E$, and let $T\in\mathscr T_L$. Then
\begin{equation}\label{eq:residual-criterion}
\kappa^{\std}_{n_T,1}(\varphi_0)\ne0
\quad\Longleftrightarrow\quad
\intr(E,K)=0.
\end{equation}
When these equivalent conditions hold,
\begin{equation}\label{eq:terminal}
\ord\kappa^{\std}_{n_T,M}(\varphi_0)=2^M
\quad(1\le M\le L).
\end{equation}
For an arbitrary parametrization $\varphi$, one has
\[
\kappa^{\std}_{n_T,1}(\varphi)\ne0
\quad\Longleftrightarrow\quad
b_\varphi=0\ \text{and}\ \intr(E,K)=0.
\]
\end{theorem}

\begin{corollary}\label{cor:main-shift}
Assume \Heegner, \Surjectivity{L}, \Exclusion, and $\intr(E,K)=0$. Let $T\in\mathscr T_L$. For any fixed parametrization $\varphi$, Proposition~\ref{prop:manin-normalization} gives $\varphi=[2^b]\circ\varphi_0$ with
$b=b_\varphi=v_2(c_\varphi)$ and $v_2(c_{\varphi_0})=0$. Then
\begin{equation}\label{eq:main-order}
\ord\kappa^{\std}_{n_T,M}(\varphi)=2^{\max\{M-b,0\}}
\quad(1\le M\le L),
\end{equation}
where the zero class has order one.
\end{corollary}

The nonvanishing and order assertions are proved in Sections~\ref{sec:index} and \ref{subsec:map-shift}.

\subsection{Strong nonvanishing and Selmer applications}\label{subsec:main-strong}
In the support notation above, strong nonvanishing means that there is a fixed $r\ge0$ such that, for every sufficiently large $k$, there exist $M\ge1$ and $n\in\mathcal N_{M+k}(E,K)$ with
\begin{equation}\label{eq:strong-intro}
\nu(n)=r,\qquad \kappa^{\std}_{n,M}(\varphi)\ne0.
\end{equation}
The least such $r$ is denoted by $\fstr$. It ranges over all admissible conductors in the standard system, whereas $R$ minimizes the support of terminal twists in $\mathcal P_L$. This is Kolyvagin's strong-system formulation for arbitrary primes \cite{KolyvaginPreprint}; see Section~\ref{sec:strong}.

\begin{theorem}\label{thm:main-bound}
Assume \Heegner, \Surjectivity{3}, \Exclusion, and $\intr(E,K)=0$. Fix a nonconstant modular parametrization $\varphi$.
Then the standard system satisfies strong nonvanishing, and
\begin{equation}\label{eq:main-strong-bound}
\fstr\le R.
\end{equation}
The coefficient depth $M=b_\varphi+1$ suffices in \eqref{eq:strong-intro}. Moreover $\fstr(\varphi)=\fstr(\varphi_0)$ under the normalization \eqref{eq:manin-normalization}.
\end{theorem}

The proof is completed in Section~\ref{sec:strong}. A separate result, Theorem~\ref{thm:Kolyvagin}, identifies the minimum support of a strongly nonvanishing standard system under \Heegner, \Surjectivity{3}, and $T_E=0$:
\begin{equation}\label{eq:main-corank}
\fstr=\max\{s_2(E),s_2(E^D)\}-1.
\end{equation}
Theorem~\ref{thm:full-selmer-structure} combines this identification with the usual abstract Selmer group structure. The proofs are in Appendix~\ref{app:kolyvagin-structure}.

\subsection{Relation to previous work}
Wei Zhang proves that some derived Heegner class is nonzero modulo $p$, for a good ordinary prime $p\ge5$ with $p\nmid\disc(K)$, under his surjectivity, generalized Heegner, and local ramification hypotheses \cite[Theorem 1.1]{Zhang}. Sweeting uses congruences modulo higher prime powers to obtain nonzero Kolyvagin classes under weaker residual and local conditions \cite[Theorem A and \S1.1]{Sweeting}. Under her standing condition $p\nmid2N\disc(K)$, the conclusion is nonvanishing at some prime-power precision.

Level raising modulo $2$ is available: Le Hung and Li prove simultaneous level raising with prescribed signs at the new primes and obtain arbitrary $\lambda$-Selmer dimensions for modular abelian varieties over $\Q$ under their respective hypotheses \cite[Theorems 1.1 and 1.4]{LHL}. The difficulty in adapting the odd-prime argument concerns rank lowering over the Heegner field. The condition $q\not\equiv\pm1\pmod p$ used there \cite[Theorem 1.7 and \S1.9]{LiRankLowering} is impossible for $p=2$; at an odd level-raising prime, residual Frobenius has repeated eigenvalue $1$, and congruences modulo $2$ do not distinguish the signs $+1$ and $-1$ \cite[\S1]{LHL}.

Li proves the following obstruction \cite[Assumption 4.1 and Theorem 7.1]{LiRankLowering}. Under his conductor and dyadic hypotheses, for an inert level-raising prime $q$ whose Frobenius has order two on $E[2]$, the corresponding level-raised abelian variety $A$, with prime $\lambda\mid2$ and residue field $k$, satisfies
\[
\dim_{\F_2}\Sel_2(E/K)=1
\quad\Longrightarrow\quad
\dim_k\Sel_\lambda(A/K)=2.
\]
Indeed, conjugation acts trivially on a one-dimensional $\F_2$-Selmer group. Its localization at $q$ therefore lies in $H^1_{\ur}(K_q,W)$, where $W=E[2]^{\Frob_q=1}$. After extension of scalars to $k$, this line is the intersection of the original and raised local conditions. The duality calculation then forces the dimension to increase by one \cite[Lemma 4.3 and Theorem 6.1]{LiRankLowering}. Thus the rank-one-to-zero step of this single-prime level-raising strategy over $K$ fails under these hypotheses.

Coates--Li--Tian--Zhai develop induction with quadratic twists and genus points, obtaining analytic nonvanishing and BSD results in specified families, including products of prescribed numbers of primes \cite[Theorems 1.1--1.4]{CLTZ}. Their hypotheses, notably those for $X_0(49)$ and for the cusp image in Theorem 1.1, differ from our residual $S_3$ setting. In particular, a rational torsion cusp image outside $2E(\Q)$ forces nontrivial rational $2$-torsion, which is incompatible with residual $S_3$ image.

We keep the elliptic curve, Heegner field, and parametrization fixed. The simultaneous descent selects a common twisting support within a prescribed prime family at each depth. For a $2$-indivisible parametrization, the genus-trace comparison and the lattice-index calculation characterize residual nonvanishing on these terminal supports by the vanishing of the intrinsic index. The corank identification in Section~\ref{sec:selmer-bound} is obtained separately by adapting Kolyvagin's arbitrary-prime descent, with uniform losses in the coefficient depth.

\paragraph{Outline of the proof.}
The two initial Selmer dimensions have opposite parity. When both groups have nonzero image modulo the exceptional class $\xi_E$, an auxiliary prime detects both and lowers both dimensions by one. If one group becomes zero or consists only of $\xi_E$, a two-prime step lowers the other dimension by two while restoring the first. This reaches dimensions $0$ and $1$. For $\Delta_E>0$, the final support can be chosen even so that the mixed local conditions are those of actual twists. When the intrinsic index vanishes, nesting gives the sharper minimum $R$. The argument uses the twisting method of Mazur--Rubin \cite{MR} and the parity theorem for mixed local conditions of Klagsbrun--Mazur--Rubin \cite{KMR}.

The twist construction can be carried out at every depth if the Galois representation on the $2$-adic Tate module is surjective:
\[
\rho_{E,2^\infty}:G_\Q\longrightarrow
\operatorname{Aut}_{\Z_2}(T_2E)\simeq\GL_2(\Z_2).
\]
Indeed, this implies \Surjectivity{L} for every $L\ge2$. Under \Heegner\ and \Exclusion, Theorem~\ref{thm:main-descent} therefore supplies terminal supports at every such depth, with a bound independent of $L$. When $\intr(E,K)=0$, their size can be fixed to be $R$. By Lemma~\ref{lem:mod-eight}, surjectivity modulo $8$, namely \Surjectivity{3}, suffices.

The norm relations for Heegner points \cite[Proposition 3.7(1)]{GrossKolyvagin} yield the genus-trace identity \eqref{eq:full-character}, proved in Lemma~\ref{lem:genus-point}. Proposition~\ref{prop:mod2-comparison} then identifies the mod-$2$ Kummer class of the resulting point on a terminal twist with the standard Kolyvagin class. Each auxiliary prime contributes the same power of $2$ to the character-trace height and to the Tamagawa product. The low Selmer dimensions also force the $2$-primary $\Sha$ groups of the comparison curves to vanish. The point's lattice index satisfies $v_2(m_T)=v_2(c_\varphi)+\intr(E,K)$, independently of $T$. The normalization in Section~\ref{sec:normalization} therefore gives the equivalence \eqref{eq:residual-criterion}. When the intrinsic index vanishes, coefficient compatibility gives the order formula; under full $2$-adic surjectivity, the uniform support size $R$ then gives strong nonvanishing.

\section{Quadratic twists and local Kummer conditions}\label{sec:local}
\subsection{Twists, Kummer images, and local norms}\label{subsec:norm}
We use Kramer's local norm comparison \cite[Proposition 7]{Kramer} and the local-condition framework of Mazur--Rubin \cite[\S2]{MR} to compare the Selmer groups of quadratic twists in a common cohomology group.
If $E:y^2=x^3+Ax+B$, its twist by $d\in\Q^\times/\Q^{\times2}$ may be written as
\[
E^d:y^2=x^3+d^2Ax+d^3B.
\]
Let $\chi_d(g)=g(\sqrt d)/\sqrt d$. Over $\Q(\sqrt d)$ there is an isomorphism
\[
\iota_d:E^d\longrightarrow E,\qquad (x,y)\longmapsto(x/d,y/(d\sqrt d)),
\qquad g\iota_d=[\chi_d(g)]\iota_d.
\]
This gives $E^d[2^M]\simeq E[2^M]\otimes\chi_d$. Since $[-1]$ acts trivially on $2$-torsion, there is a natural identification $E^d[2]=V$ \cite[Remark 2.4]{MR}, \cite[Lemma 5.2(i)]{KMR}.

For a local field $F$, write
\[
\alpha_F(E)=\im\bigl(E(F)/2E(F)\xrightarrow{\delta_{E,F,2}}H^1(F,E[2])\bigr).
\]
The subspaces $\alpha_F(E)$ and $\alpha_F(E^d)$ of $H^1(F,V)$ generally differ. After transporting points to $E$ via $\iota_d$, the twisted Kummer cocycle is
\begin{equation}\label{eq:twist-cocycle}
g\longmapsto\chi_d(g)gQ-Q,\qquad 2Q=P,\quad gP=\chi_d(g)P.
\end{equation}
Here $P$ is the image on $E$ of a rational point on the twist. The twisted action on a half-point $Q$ determines its Kummer class in the common cohomology group.

The Weil pairing and local Tate duality give a nondegenerate pairing on $H^1(F,V)$, for which each Kummer image is its own orthogonal complement \cite[Chapter I, \S3]{Milne}, \cite[Proposition 4.11]{PR}.

\begin{lemma}\label{lem:norm-intersection}
Let $F'/F$ be a quadratic extension, and let $E^\chi$ be the corresponding twist. Under the identification of $2$-torsion,
\begin{equation}\label{eq:intersection-norm}
\alpha_F(E)\cap\alpha_F(E^\chi)
=\delta_{E,F,2}\bigl(N_{F'/F}E(F')\bigr).
\end{equation}
In particular, the two Kummer images have equal dimension, and
\[
\dim\frac{\alpha_F(E)}{\alpha_F(E)\cap\alpha_F(E^\chi)}
=\dim\frac{E(F)}{N_{F'/F}E(F')}.
\]
\end{lemma}
\begin{proof}
This is Kramer's local norm comparison \cite[Proposition 7]{Kramer}, \cite[Lemma 2.9]{MR}. We recall the cocycle argument. Suppose $P\in E(F)$ and $P'\in E^\chi(F)$ define the same cohomology class. Choose $2R=P$ and $2S=\iota_\chi(P')$. The two cocycles differ by a coboundary, so, after modifying $S$ by a $2$-torsion point, we may arrange that
\[
gR-R=\chi(g)gS-S.
\]
For $g\in G_{F'}$, both sides are the same $2$-torsion point. Thus $g(R+S)-(R+S)=2(gR-R)=0$, and $Q=R+S\in E(F')$. If $\sigma$ lifts the nontrivial automorphism of $F'/F$, the displayed identity gives $\sigma(R+S)=R-S$, whence $NQ=2R=P$.

Conversely, for $Q\in E(F')$, put $P=Q+\sigma Q$ and $P^-=Q-\sigma Q$. Then $P\in E(F)$, while $\sigma P^-=-P^-$, so $P^-$ corresponds to a point of $E^\chi(F)$. Choose $2R=P$ and put $S=Q-R$; then $2S=P^-$. Checking separately on $G_{F'}$ and its other coset gives $gR-R=\chi(g)gS-S$. Hence the two points define the same Kummer class. This proves the intersection formula. The two Kummer images have equal dimension by their self-duality. Finally, $2E(F)\subseteq NE(F')$ and injectivity of the Kummer map give the stated codimension.
\end{proof}

For a nonarchimedean local field $F$ with residue field $k_F$, write $\widetilde E/k_F$ for the special fiber of the N\'eron model of $E/F$. Let $E_0(F)$ be the inverse image of $\widetilde E^0(k_F)$ under reduction, and let $E_1(F)$ be the kernel of reduction.

\begin{lemma}\label{lem:local-twists}
Suppose $E/F$ has good reduction.
If $F'/F$ is unramified quadratic, then $NE(F')=E(F)$ and $\alpha_F(E)=\alpha_F(E^\chi)$, including in residue characteristic two.
If the residue characteristic is odd and $F'/F$ is ramified, then
\[
NE(F')=2E(F),\qquad \alpha_F(E)\cap\alpha_F(E^\chi)=0.
\]
A twist by a local square leaves the Kummer condition unchanged.
\end{lemma}
\begin{proof}
In the unramified case, let $k'/k$ be the residue-field extension. Lang's theorem gives $H^1(\Gal(k'/k),\widetilde E(k'))=0$. For this finite cyclic action, the norm formula then implies surjectivity of $\widetilde E(k')\to\widetilde E(k)$. On each formal-group quotient $E_j(F')/E_{j+1}(F')$, the norm is the surjective trace $\Tr_{k'/k}$. Successive lifting and completeness yield surjectivity of $E_1(F')\to E_1(F)$. The reduction exact sequence therefore gives $NE(F')=E(F)$. The argument also applies in residue characteristic two. See \cite[Lemma 2.10(v)]{MR}; equality of Kummer images follows from Lemma~\ref{lem:norm-intersection}.

In the ramified case of odd residue characteristic, the two residue fields coincide. Given $Q\in E(F')$, choose $R\in E(F)$ with the same reduction. Then $NQ-2R\in E_1(F)$. Multiplication by $2$ is invertible on the formal group $E_1(F)$, so $NQ\in2E(F)$. The reverse inclusion follows from $NP=2P$ for $P\in E(F)$. Apply \eqref{eq:intersection-norm}; see also \cite[Lemmas 2.10--2.11]{MR}.
\end{proof}

At an odd prime $q$ of good reduction, inertia acts trivially on $V$, and
\begin{equation}\label{eq:ur-line}
\alpha_q(E)=H^1_{\ur}(\Q_q,V)\simeq V/(\Frob_q-1)V.
\end{equation}
The isomorphism in \eqref{eq:ur-line} is given by evaluating cocycles at Frobenius \cite[Lemma 2.2]{MR}. If residual Frobenius is a transposition, both its invariant space and its coinvariant quotient have dimension one. Local Tate duality then gives $\dim H^1(\Q_q,V)=2$. The Kummer line of a ramified twist intersects $\alpha_q(E)$ trivially. This is the local change of line used in simultaneous descent \cite[Lemma 2.11]{MR}.

\subsection{The real place, total local discrepancy, and index zero}\label{subsec:parity}
If $\Delta_E>0$, the defining cubic has three real roots, complex conjugation acts trivially on $V$, and $\epsilon_\infty=1$. If $\Delta_E<0$, complex conjugation is a transposition and $\epsilon_\infty=0$. In particular,
\[
H^1(\R,V)=\begin{cases}V,&\Delta_E>0,\\0,&\Delta_E<0.\end{cases}
\]
Quadratic twisting preserves the sign of the discriminant: twisting multiplies it by $d^6$, and passage to a minimal model changes it by a twelfth power.

For each place $v$ of $\Q$, define the local discrepancy between the two initial curves by
\[
h_v=\dim_{\F_2}\frac{\alpha_v(E)}{\alpha_v(E)\cap\alpha_v(E^D)}.
\]
By Lemma~\ref{lem:norm-intersection}, this equals the local norm index
\[
h_v=\dim_{\F_2}E(\Q_v)/N_{K\otimes\Q_v/\Q_v}E(K\otimes\Q_v).
\]
We compute these indices place by place. At a prime of bad reduction of $E$, the Heegner hypothesis makes $D$ a local square, so $h_v=0$. Outside $2ND$, the extension is split or unramified at a place of good reduction, again giving $h_v=0$. For odd $r\mid D$, the extension is ramified and $E$ has good reduction. Lemma~\ref{lem:local-twists} gives norm image $2E(\Q_r)$, hence $h_r=t_r$. At the real place, $N_{\C/\R}E(\C)=E(\R)^0$, so $h_\infty=\epsilon_\infty$. The contribution at $2$ is $h_2(E,K)$. See \cite[\S2]{MR} for these local computations and Appendix~\ref{app:dyadic} for the dyadic values. Thus
\begin{equation}\label{eq:total-local}
\sum_vh_v=\epsilon_\infty+h_2(E,K)+\sum_{r\mid D,\ r\text{ odd}}t_r=1+2\delta_K.
\end{equation}
For odd $r\mid D$, Frobenius acts on the three nonzero $2$-torsion points as the identity, a transposition, or a $3$-cycle; the corresponding values of $t_r$ are $2,1,0$. The sign character of the $2$-division representation corresponds to $\Q(\sqrt{\Delta_E})$, so
\[
(-1)^{t_r}=\left(\frac{\Delta_E}{r}\right)=(D,\Delta_E)_r.
\]
The last equality uses the odd-prime Hilbert-symbol formula with $v_r(D)=1$ and $v_r(\Delta_E)=0$. At $2$, Lemma~\ref{lem:local-twists} in the split and unramified cases and \cite[Propositions 4--5]{Kramer} in the ramified case give $(-1)^{h_2(E,K)}=(D,\Delta_E)_2$.
At the remaining finite places, either $D$ is a square or both $D$ and $\Delta_E$ are units at an odd prime; the Hilbert symbol is therefore one. At the real place, $D<0$, and hence
\[
(D,\Delta_E)_\infty=
\begin{cases}-1,&\Delta_E<0,\\1,&\Delta_E>0\end{cases}
=(-1)^{1-\epsilon_\infty}.
\]
The product formula $\prod_v(D,\Delta_E)_v=1$ now gives
\[
h_2(E,K)+\sum_{r\mid D,\ r\text{ odd}}t_r
\equiv1-\epsilon_\infty\pmod2.
\]
Thus $\sum_vh_v$ is a nonnegative odd integer, hence at least one. Equation~\eqref{eq:total-local} yields $\delta_K\in\Z_{\ge0}$.

On the other hand, Kramer's global parity formula \cite[Theorems 1--2]{Kramer}, in the notation of \cite[Theorem 2.7]{MR}, reads
\begin{equation}\label{eq:kramer-parity}
a_0+a_1\equiv\sum_vh_v\pmod2.
\end{equation}
Here $a_0=d_2(E)$ and $a_1=d_2(E^D)$, and the right-hand side is the sum of the local codimensions between the two Kummer conditions. The formula arises by comparing two Cassels pairings on the common Selmer subgroup and applying the Kummer exact sequence; see \cite[proof of Theorem 2.7 and Remark 2.8]{MR}. Together with the local parity just established, it gives $a_0+a_1\equiv1\pmod2$.

\begin{proposition}\label{prop:nesting}
If $\delta_K=0$, then one of $\Sel_2(E/\Q)$ and $\Sel_2(E^D/\Q)$ contains the other in $H^1(\Q,V)$, and their dimensions differ by one.
\end{proposition}
\begin{proof}
Write $S_i=\Sel_2(E^{D^i}/\Q)$ and $C=S_0\cap S_1$. For $x\in S_i$, reduce each localization modulo the intersection of the two local conditions. The kernel of the resulting map consists exactly of the classes satisfying both sets of conditions, namely $C$. Thus localization induces an injection
\[
S_i/C\hookrightarrow\bigoplus_v
\frac{\alpha_v(E^{D^i})}{\alpha_v(E)\cap\alpha_v(E^D)}.
\]
Since the two local Kummer images have equal dimension, the right-hand side has dimension $\sum_vh_v$ for either $i$. If $\delta_K=0$, this sum is one, so $u_i:=\dim(S_i/C)\in\{0,1\}$. Moreover,
\[
u_0+u_1=\dim S_0+\dim S_1-2\dim C
\equiv a_0+a_1\equiv1\pmod2.
\]
Consequently, $(u_0,u_1)$ is either $(0,1)$ or $(1,0)$. The first possibility gives $S_0=C\subset S_1$ with codimension one; the second gives the reverse inclusion. For this combination of localization and parity, see \cite[\S3]{MR}; the general parity theorem for local conditions is \cite[Theorem 3.9]{KMR}.
\end{proof}

Under \Exclusion, the smaller group does not contain $\xi_E$; if this class belongs to their sum, it lies only in the larger group.

\begin{lemma}\label{lem:unramified-odd}
Let $F$ be a nonarchimedean local field of characteristic zero, let $F'/F$ be unramified quadratic, and let $C/F$ be an elliptic curve with odd Tamagawa number $c_F(C)$. Then
\[
N_{F'/F}C(F')=C(F),\qquad
\alpha_F(C)=\alpha_F(C^\chi),
\]
where $\chi$ is the character of $F'/F$. Moreover, unramified quadratic twisting preserves the parity of the Tamagawa number.
\end{lemma}
\begin{proof}
N\'eron models commute with unramified base change. Their identity components have connected smooth commutative special fibers over finite fields. Lang's theorem and the norm formula for finite cyclic groups give surjectivity of the norm on these fibers. On each formal-group filtration quotient, the norm is the surjective residue-field trace. Successive lifting and completeness therefore give
$NC_0(F')=C_0(F)$; see also \cite[\S6, printed p.~327]{KT}.
For $P\in C(F)$, the odd order of $C(F)/C_0(F)$ allows us to choose $R\in C(F)$ with $P-2R\in C_0(F)$. Write $P-2R=NQ$. Then $P=N(Q+R)$. Equality of Kummer images follows from Lemma~\ref{lem:norm-intersection}.

Let $\Phi$ be the geometric component group of the N\'eron model and let $\mathrm{Fr}$ be residue-field Frobenius. Unramified quadratic twisting replaces $\mathrm{Fr}$ by $-\mathrm{Fr}$, and
\[
\Phi[2]^{\mathrm{Fr}=1}=\Phi[2]^{-\mathrm{Fr}=1}.
\]
Since $c_F(C)=\#\Phi^{\mathrm{Fr}=1}$, this number is even exactly when the displayed $2$-torsion invariant space is nonzero. Thus the Tamagawa numbers before and after twisting have the same parity.
\end{proof}

If $\Delta_E>0$ and $\delta_K=0$, the real place already contributes one to \eqref{eq:total-local}. Hence $h_v=0$ at every finite place, or equivalently,
\begin{equation}\label{eq:finite-local-equal}
\alpha_v(E)=\alpha_v(E^D)\qquad(v<\infty).
\end{equation}
The two sets of local conditions differ only at the real place. If in addition $T_E=0$, all the original Tamagawa numbers are odd, and Lemma~\ref{lem:unramified-odd} permits odd support size when $\Delta_E>0$: the negative twisting parameter exchanges the two real Kummer lines while preserving the finite local conditions.

\section{The mod-$4$ obstruction and simultaneous Selmer descent}\label{sec:descent}
Assume \Heegner\ and \Surjectivity{L} throughout this section. The simultaneous descent uses the twisting method of \cite[\S\S3--5]{MR} and the mixed Selmer structures of \cite[\S\S3--5]{KMR}.

\subsection{The cohomological kernel over division fields}\label{subsec:xi}\label{subsec:exclude}
Let $\Gamma_2=\GL_2(\Z/4\Z)$. Reduction gives an exact sequence
\[
1\longrightarrow U=1+2M_2(\F_2)\longrightarrow\Gamma_2
\longrightarrow S_3\longrightarrow1.
\]
The subgroup $U$ acts trivially on $V$, and $(1+2B)(1+2C)=1+2(B+C)$. The restriction of a cocycle to $U$ is therefore a linear map $M_2(\F_2)\to V$. Compatibility with conjugation requires
\[
F(hBh^{-1})=hF(B),\qquad h\in S_3.
\]
Choose the complement $H=\langle s,t\rangle\simeq S_3$, where
\[
s=\begin{pmatrix}0&1\\1&0\end{pmatrix},\qquad
t=\begin{pmatrix}0&3\\1&3\end{pmatrix}\quad\text{over }\Z/4\Z.
\]
Direct calculation gives $s^2=t^3=(st)^2=1$, and $H$ maps isomorphically onto $S_3$ modulo $2$.

\begin{lemma}\label{lem:f}
Identify $V=\F_2^2$ using the standard basis. Then
\begin{equation}\label{eq:f}
\Hom_{S_3}(M_2(\F_2),V)=\F_2f,\qquad
f\begin{pmatrix}a&b\\c&d\end{pmatrix}=\binom{a+d+b}{a+d+c}.
\end{equation}
\end{lemma}
\begin{proof}
Since $F(I)\in V^{S_3}=0$, we have $F(E_{11})=F(E_{22})$. Conjugation by $\bar s$ shows that both values are $\lambda(e_1+e_2)$. The identities
\[
\bar tE_{11}\bar t^{-1}=E_{21}+E_{22},\qquad
\bar sE_{21}\bar s^{-1}=E_{12}
\]
then give $F(E_{21})=\lambda e_2$ and $F(E_{12})=\lambda e_1$. Linearity yields $F=\lambda f$. Conversely, direct substitution verifies equivariance of $f$ under $\bar s$ and $\bar t$.
\end{proof}

Each $A\in\Gamma_2$ has a unique expression $(1+2B)h$ with $h\in H$. Define
\begin{equation}\label{eq:psi}
\psi((1+2B)h)=f(B).
\end{equation}
The semidirect-product law and equivariance of $f$ give
\[
\psi(AA')=\psi(A)+A\psi(A'),
\]
so $\psi$ is a cocycle. The subgroup $C_3$ has odd order and $V^{C_3}=0$; inflation--restriction gives $H^1(S_3,V)=0$. Thus, after subtracting a coboundary, any $\Gamma_2$-cocycle vanishes on $H$, and its remaining values are determined by \eqref{eq:f}. Every coboundary vanishes on $U$, whereas $\psi(1+2E_{12})=e_1\ne0$. Hence
\begin{equation}\label{eq:H1Gamma}
H^1(\Gamma_2,V)=\F_2[\psi].
\end{equation}
Inflation along $\rho_{E,4}$ gives the nonzero global class $\xi_E$.

\begin{lemma}\label{lem:kernel-stable}
For every $j\ge2$, inflation induces an isomorphism
\[
H^1(\Gamma_2,V)\simeq H^1(\Gamma_j,V).
\]
If $A/\Q$ is a finite abelian extension and $F=\Q(E[2^L])A$, then
\begin{equation}\label{eq:restriction-kernel}
\ker\bigl(H^1(\Q,V)\longrightarrow H^1(F,V)\bigr)=\F_2\xi_E.
\end{equation}
\end{lemma}
\begin{proof}
The restriction of a cocycle to the kernel of $\Gamma_3\to\Gamma_2$ is an equivariant homomorphism, hence either zero or $f$. Take $u=1+2E_{12}$. Since $u$ acts trivially on $V$, we have $c(u^2)=2c(u)=0$. But $u^2=1+4E_{12}$ and $f(E_{12})\ne0$, ruling out the nonzero restriction. For $j\ge3$, every kernel element $1+2^jB$ satisfies
\[
1+2^jB\equiv(1+2^{j-1}B)^2\pmod{2^{j+1}},
\]
which again forces the restriction to vanish. Induction proves the first assertion.

Put $G=\Gal(F/\Q)$ and $C=\ker(G\to\Gamma_L)$. In the fiber product $G\subseteq\Gamma_L\times\Gal(A/\Q)$, the subgroup $C$ is central. The image of any equivariant homomorphism $C\to V$ is therefore contained in $V^{S_3}=0$. Hence $H^1(G,V)=H^1(\Gamma_L,V)$. Applying inflation--restriction to $G_\Q\to G$ gives \eqref{eq:restriction-kernel}.
\end{proof}

\paragraph{Local exclusion.}
Suppose $\Delta_E>0$, and let $c$ denote complex conjugation. Since $c$
acts trivially on $V$, write its action on $T_2E$ as
$I+2\widetilde B$ and put $B=\widetilde B\bmod2$.
The identities $c^2=I$ and $\tr(c)=0$ give
$\widetilde B^2=-\widetilde B$ and $\tr(\widetilde B)=-1$.
Thus $B$ is a rank-one idempotent.
Every class in $E(\R)/2E(\R)$ is represented by a point of $E(\R)[2]=V$:
indeed, $E(\R)\simeq\R/\Z\oplus\Z/2\Z$ as a Lie group. For $P\in V$,
choose $Q\in E[4]$ with $2Q=P$. Evaluation of its Kummer cocycle at $c$
gives
\[
\delta_{E,\R,2}(P)(c)=cQ-Q=BP.
\]
Since $c$ acts trivially on $V$, evaluation identifies $H^1(\R,V)$ with
$V$, and hence $\alpha_\R(E)=\im B$. The negative twist replaces $c$ by
$-c\equiv I+2(I+B)\pmod4$; applying the same calculation under the
identification $E^D[2]=V$ gives $\alpha_\R(E^D)=\im(I+B)$.
For $B=\operatorname{diag}(1,0)$, the vector $f(B)=(1,1)$ belongs to neither
line. All rank-one idempotents are conjugate, and $f$ is equivariant.
Thus the localization of $\xi_E$ lies in neither real Kummer image, so
$\xi_E$ is excluded from both Selmer groups.

For positive discriminant, residual surjectivity also forces surjectivity modulo $4$. Indeed, on $T_2E$ write $c=I+2C$ for complex conjugation. The relations $c^2=I$ and $\det c=-1$ give $C^2=-C$ and $\tr(C)\equiv1\pmod2$, so $\overline C$ is a rank-one idempotent. The reduction kernel is an $S_3$-stable subspace of $M_2(\F_2)$ containing $\overline C$. Its conjugates include
\[
E_{11},\quad E_{22},\quad E_{11}+E_{12},\quad E_{11}+E_{21},
\]
which span $M_2(\F_2)$. The kernel has order $16$, and the mod-$4$ image has order $96$.

\begin{proposition}\label{prop:odd-exclusion}
Let $r\nmid2N$ and suppose $\Frob_r$ acts on $V$ as a transposition. Then
\begin{equation}\label{eq:xi-local-test}
\loc_r\xi_E\ne0\quad\Longleftrightarrow\quad a_r(E)\equiv2\pmod4.
\end{equation}
If also $r\mid D$, the condition on the right implies $\xi_E\notin\Sel_2(E^D/\Q)$.
\end{proposition}
\begin{proof}
At an odd good prime, $E[4]$ is unramified, so the localization of $\xi_E$ belongs to the Kummer line of $E$. After conjugation, write Frobenius as
\[
A=(1+2B)s,\qquad B=\begin{pmatrix}u&v\\w&z\end{pmatrix}.
\]
Its cocycle value is $f(B)=(u+z+v,u+z+w)$. The nonzero functional on $V/(s-1)V$ is $(x,y)\mapsto x+y$, so the local class is nonzero exactly when $v+w=1$. Since
\[
\tr A=2(v+w)\pmod4,
\]
we obtain \eqref{eq:xi-local-test}. If $r\mid D$, the curve $E^D$ is a ramified twist, and Lemma~\ref{lem:local-twists} gives trivial intersection of the two Kummer lines. A nonzero class in the original line therefore fails the new local condition.
\end{proof}

When $\Delta_E<0$, membership of $\xi_E$ in the relevant Kummer images can be checked at finite places, in particular by Proposition~\ref{prop:odd-exclusion}. Appendix~\ref{app:squareclass} gives explicit squareclass coordinates for $\xi_E$ in the cubic algebra.

\subsection{A fixed family of deep auxiliary primes}\label{subsec:prime-family}
Put
\begin{equation}\label{eq:F0BD}
F_L^0=\Q\bigl(E[2^L],\zeta_8,\sqrt\ell:\ell\mid N,\ \ell\text{ odd}\bigr),
\qquad B_D=\Q\bigl(\sqrt r:r\mid D,\ r\text{ odd}\bigr).
\end{equation}
The field $F_L^0$ can ramify only at primes dividing $2N$, whereas every nontrivial quadratic subfield of $B_D$ ramifies at some odd $r\mid D$. Since $(D,N)=1$, we have $F_L^0\cap B_D=\Q$.

If $\Delta_E<0$, prescribe actual complex conjugation on $F_L^0$. If $\Delta_E>0$, prescribe on the division field the matrix
\[
A_L=\begin{pmatrix}1&1\\0&-1\end{pmatrix}
\]
and choose a lift to $F_L^0$ compatible with the following cyclotomic prescription. In both cases prescribe the identity on $B_D$. If $D$ is even, also prescribe $\zeta_8\mapsto\zeta_8^{-1}$. For $L\ge3$, the determinant already imposes this condition. If $D$ is even and $L=2$, put $S_N=\Q(\sqrt\ell:\ell\mid N,\ \ell\text{ odd})$. We claim that
\[
\bigl(\Q(E[4])S_N\bigr)\cap\Q(\zeta_8)=\Q(i).
\]
Indeed, the abelianization of $\Gamma_2$ is generated by the determinant and the residual permutation-sign character. In the fiber product describing $\Gal(\Q(E[4])S_N/\Q)$, the commutator subgroup is $[\Gamma_2,\Gamma_2]\times\{1\}$: any two elements of $\Gamma_2$ admit lifts, and the second coordinates of their commutator are trivial. The maximal abelian subextension is therefore $\Q(i,\sqrt{\Delta_E})S_N$. When $D$ is even, good reduction at $2$ gives $v_2(\Delta_E)=0$. Every quadratic squareclass generated by these radicands has even valuation at $2$, so $\sqrt2$ does not belong to $\Q(E[4])S_N$. Since this field contains $\Q(i)$, the asserted intersection follows. The prescribed matrix has determinant $-1$ modulo $4$, so its action on $i$ agrees with $\zeta_8\mapsto\zeta_8^{-1}$. The prescriptions consequently extend to the entire compositum, including the square roots of the odd bad primes.

\begin{definition}\label{def:PL}
Fix the resulting element $g_L$ of $\Gal(F_L/\Q)$, where $F_L=F_L^0B_D$. Let $\mathcal P_L$ be the set of primes $q\ge5$, $q\nmid ND$, for which $\Frob_q(F_L/\Q)$ is conjugate to $g_L$. Any prescribed finite set may be removed.
\end{definition}
Chebotarev's theorem gives $\mathcal P_L$ positive density. The construction ensures that
\begin{equation}\label{eq:PL-properties}
\begin{gathered}
q\equiv-1\pmod{2^L},\quad a_q(E)\equiv0\pmod{2^L},\quad
\dim_{\F_2}\mathcal E(\F_q)[2]=1,\\
q\text{ is inert in }K,\qquad
-q\in\Q_r^{\times2}\quad(r\mid D,\ r\text{ odd}).
\end{gathered}
\end{equation}
The first three properties follow from the determinant, trace, and residual fixed space of the prescribed matrix. Write $K=\Q(\sqrt{-2^\epsilon R})$ with $R$ odd and squarefree and $\epsilon\in\{0,1\}$. The element $g_L$ fixes $\sqrt R$ and, when needed, $\sqrt2$, but changes the sign of $i$. It therefore acts nontrivially on $K$, proving inertness.

For odd $r\mid D$, fixing $\sqrt r$ gives $(r/q)=1$. Since $q\equiv3\pmod4$, quadratic reciprocity yields
\begin{equation}\label{eq:reciprocity}
\left(\frac{-q}{r}\right)
=\left(\frac{-1}{r}\right)\left(\frac qr\right)
=\left(\frac rq\right)=1.
\end{equation}
As $q\ne r$, Hensel's lemma lifts a nonzero square root modulo $r$ to $\Q_r$, proving the local-square condition.

Similarly, the local squareclasses of $-q$ at the original odd bad primes are fixed, and its squareclass at $2$ is determined by the prescribed residue class modulo $8$. If $\Delta_E<0$, the choice of actual complex conjugation makes each $-q$ a square at every finite $v\mid2ND$. If $\Delta_E>0$, each $-q$ is a unit at the original odd bad primes and gives a trivial or unramified twist. At $2$, the congruence $q\equiv3\pmod4$ puts $-q$ in squareclass $1$ or $5$, again giving a trivial or unramified twist. When $D$ is even, the prescribed action on $\zeta_8$ further gives $q\equiv7\pmod8$, so $-q$ is a square in $\Q_2$. For an even support, these fixed local squareclasses cancel, so $d_T$ is a square at every finite $v\mid2ND$. At intrinsic index zero, odd supports are also allowed: Lemma~\ref{lem:unramified-odd} treats the original finite places, and Proposition~\ref{prop:actual-pair} treats the exchange at the real place.

\subsection{Simultaneous detection and changes of local lines}\label{subsec:detection}
\begin{lemma}\label{lem:detection}
Let $S_0,S_1\subset H^1(\Q,V)$ be finite-dimensional, and put $S=S_0+S_1$ and $J=S\cap\F_2\xi_E$. Away from the primes where these classes ramify, $\loc_qS\subset H^1_{\ur}(\Q_q,V)$. For $q\in\mathcal P_L$, this last space is one-dimensional and hence has a unique identification with $\F_2$. For every $\lambda\in(S/J)^*$, there are infinitely many $q\in\mathcal P_L$ such that
\[
\loc_q(c)=\lambda(c+J)\qquad(c\in S).
\]
These primes may avoid any prescribed finite set. If the images of both $S_0$ and $S_1$ in $S/J$ are nonzero, one may make both localization maps nonzero. One may also make both maps zero.
\end{lemma}
\begin{proof}
We use the cocycle-evaluation and Chebotarev argument of Mazur--Rubin \cite[Lemma 3.5]{MR}, with Lemma~\ref{lem:kernel-stable} identifying the kernel remaining after the deep division field has been fixed.

First, $\xi_E$ localizes to zero at these primes. The prescribed Frobenius matrix $A$ modulo $4$ satisfies $A^2=1$, and $\bar A$ is a transposition. Since $\psi(A^2)=\psi(A)+\bar A\psi(A)=0$, we have
\[
\psi(A)\in\ker(1+\bar A)=\im(\bar A-1).
\]
Thus $\psi(A)$ vanishes in $V/(\bar A-1)V$. The Frobenius-evaluation description \eqref{eq:ur-line} gives $\loc_q\xi_E=0$, so localization factors through $S/J$.

Choose a complement to $J$ in $S$ and cocycles $c_1,\ldots,c_t$ representing a basis. Since $F_L$ contains the $2$-division field, $G_{F_L}$ acts trivially on $V$, and the restrictions $c_i|_{G_{F_L}}$ are continuous homomorphisms. They are linearly independent: if a linear combination vanishes on $G_{F_L}$, Lemma~\ref{lem:kernel-stable} places its class in $\F_2\xi_E$, whose intersection with the chosen complement is zero.

Consider the joint homomorphism and its image
\[
\Phi:G_{F_L}\longrightarrow V^t,\qquad
h\longmapsto(c_1(h),\ldots,c_t(h)),\qquad W=\im\Phi.
\]
The cocycle identity gives $c_i(ghg^{-1})=g\,c_i(h)$, so $W$ is an $S_3$-submodule. The natural module $V$ is simple and $\End_{S_3}(V)=\F_2$. Since $V^t$ is a direct sum of copies of this simple module, it is semisimple. If $W\ne V^t$, the quotient $V^t/W$ has a simple quotient isomorphic to $V$. There is therefore a nonzero equivariant map
\[
(v_1,\ldots,v_t)\longmapsto\sum_i\lambda_i v_i
\quad(\lambda_i\in\F_2)
\]
vanishing on $W$, contrary to linear independence of the restricted homomorphisms. Thus $\Phi$ is surjective.

Fix a lift $g_0\in G_\Q$ of $g_L$. For $h\in G_{F_L}$,
\[
c_i(g_0h)=c_i(g_0)+g_0c_i(h).
\]
Surjectivity of $\Phi$ allows the values on the right to be prescribed independently in $V$. In particular, their images in $V/(\bar A-1)V$ may be chosen to equal $\lambda(c_i+J)$. The element $g_0h$ still restricts to $g_L$ on $F_L$.

The cocycles are continuous with finite image. Taking the compositum of $F_L$ with the fixed fields of the corresponding finite semidirect-product representations gives a finite Galois extension $\Omega/\Q$ on which all these evaluation conditions are determined. Chebotarev's theorem supplies infinitely many primes with Frobenius conjugate to the image of the chosen $g_0h$. Conjugation transports the coinvariant quotient and changes the cocycle value only by a coboundary, so its value in the one-dimensional $\F_2$-quotient is unchanged. Removing the ramified primes of $\Omega$ and the prescribed finite set gives the required primes.

Finally, if the images $\overline S_i\subset S/J$ are both nonzero, their annihilators $\overline S_i^\perp\subset(S/J)^*$ are proper subspaces. Two proper subspaces cannot cover a vector space, so some $\lambda$ is nonzero on both images. Taking $\lambda=0$ makes both localizations zero.
\end{proof}

\paragraph{The transverse line from ramified twists.}
For $q\in\mathcal P_L$, denote the original Kummer line by $\alpha_q$. Since $\dim \mathcal E(\F_q)[2]=1$, the $2$-primary part of the reduction group is cyclic. Moreover, $4\mid q+1-a_q(E)$, so the unique nonzero $\Q_q$-rational $2$-torsion point $P$ has a $\Q_q$-rational half-point $Q$ of order four. Here we use compatibility of prime-to-residue-characteristic torsion with good reduction.

For a ramified quadratic character $\chi(g)=(-1)^{\bar\chi(g)}$, formula \eqref{eq:twist-cocycle} represents the twisted Kummer class of this $2$-torsion point by
\[
\chi(g)gQ-Q=(\chi(g)-1)Q=\bar\chi(g)P.
\]
This cocycle is nonzero on inertia and hence generates the one-dimensional Kummer image of the ramified twist. The difference between the two ramified quadratic characters is unramified; its associated cocycle takes the value $P$ at Frobenius. Since residual Frobenius is a transposition, $V^{\Frob_q=1}=(\Frob_q-1)V=\F_2P$, so this difference is cohomologically trivial. The two ramified twists therefore give the same line, denoted by $\beta_q$, and Lemma~\ref{lem:local-twists} gives $\alpha_q\cap\beta_q=0$. A subsequent unramified twist merely interchanges the two ramified characters and preserves $\beta_q$. This is the local twisting map of \cite[\S5]{KMR} in the present setting.

\begin{definition*}
A \emph{Selmer structure} on $V$ is a collection $\mathcal F=(\mathcal F_v)_v$ of subspaces $\mathcal F_v\subset H^1(\Q_v,V)$ that equal the unramified subgroups at almost all primes. Its Selmer group is
\[
H^1_{\mathcal F}(\Q,V)=
\ker\left(H^1(\Q,V)\longrightarrow
\bigoplus_vH^1(\Q_v,V)/\mathcal F_v\right).
\]
All local conditions used here are Kummer images of local quadratic twists. Choosing these twists independently at each place gives a mixed Selmer structure. For a finite set of primes $U$, define the \emph{relaxed} and \emph{strict} Selmer groups by
\begin{align*}
H^1_{\mathcal F^{U}}(\Q,V)
&=\{c\in H^1(\Q,V):\loc_vc\in\mathcal F_v\text{ for all }v\notin U\},\\ 
H^1_{\mathcal F_{U}}(\Q,V)
&=\{c\in H^1_{\mathcal F^{U}}(\Q,V):\loc_vc=0\text{ for all }v\in U\}.
\end{align*}
Thus at places in $U$ we replace the local condition by all of $H^1(\Q_v,V)$ or by zero, respectively, leaving the other conditions unchanged. In particular,
\[
H^1_{\mathcal F_U}(\Q,V)\subset H^1_{\mathcal F}(\Q,V)
\subset H^1_{\mathcal F^U}(\Q,V).
\]
We follow the strict and relaxed conventions of \cite[Definition 3.1]{MR} and the mixed-structure framework of \cite[Definition 3.8, \S4]{KMR}.
\end{definition*}

\begin{lemma}[Parity Lemma]\label{lem:switch}
Let $\mathcal F$ be a mixed Selmer structure as above, and put $S=H^1_{\mathcal F}(\Q,V)$. Choose a new prime $q\in\mathcal P_L$ with $\mathcal F_q=\alpha_q$, and let $S'$ be obtained by replacing $\alpha_q$ with $\beta_q$. Then
\begin{equation}\label{eq:switch}
\begin{array}{ll}
\loc_qS\ne0:& S'=\ker(\loc_q|_S),\quad\dim S'=\dim S-1,\\[2pt]
\loc_qS=0:&S\subset S',\quad\dim S'=\dim S+1.
\end{array}
\end{equation}
\end{lemma}
\begin{proof}
Put $S_{\mathrm{str}}=H^1_{\mathcal F_{\{q\}}}(\Q,V)$ and $S_{\mathrm{rel}}=H^1_{\mathcal F^{\{q\}}}(\Q,V)$. Since $S$ and $S'$ differ only at $q$, they have the same strict and relaxed groups. Let
\[
W_q=\loc_qS_{\mathrm{rel}}\subset H^1(\Q_q,V).
\]
The kernel of localization is $S_{\mathrm{str}}$, so
\begin{equation}\label{eq:switch-quotients}
S_{\mathrm{rel}}/S_{\mathrm{str}}\simeq W_q,\qquad
S/S_{\mathrm{str}}\simeq W_q\cap\alpha_q,\qquad
S'/S_{\mathrm{str}}\simeq W_q\cap\beta_q.
\end{equation}

Each local Kummer condition is self-dual. Poitou--Tate duality therefore gives $W_q=W_q^\perp$, and hence
\[
\dim W_q=\tfrac12\dim H^1(\Q_q,V)=1.
\]
This is the strict/relaxed dimension formula of \cite[Lemma 3.2]{MR}. Its duality input is \cite[Theorem I.4.10]{Milne}; for the same argument with mixed structures, see \cite[proof of Theorem 3.9]{KMR}.

If $\loc_qS\ne0$, then $W_q$ contains the one-dimensional space $\alpha_q$, so $W_q=\alpha_q$. Using $\alpha_q\cap\beta_q=0$ in \eqref{eq:switch-quotients} gives $S'=S_{\mathrm{str}}=\ker(\loc_q|_S)$, and the dimension decreases by one.

If $\loc_qS=0$, then $S=S_{\mathrm{str}}\subset S'$, and $\dim(S'/S)\le1$. To show equality, we use the quadratic structure in characteristic two. The theta-group central extension associated with $\mathcal O_E(2[O])$,
\[
1\longrightarrow\overline{\Q}_v^\times\longrightarrow\Theta_E
\longrightarrow V\longrightarrow0
\]
defines a canonical local quadratic form $q_{E,v}:H^1(\Q_v,V)\to\Q/\Z$ whose polarization is the local Tate pairing. Kummer images are maximal totally isotropic subspaces for this form \cite[\S4, Proposition 4.11]{PR}. The identifications of $2$-torsion for quadratic twists identify these quadratic forms as well, so every $\mathcal F_v$ is Lagrangian for the same quadratic structure \cite[Lemmas 5.2(ii), 5.4]{KMR}. Global Brauer reciprocity also gives
\[
\sum_v q_{E,v}(\loc_vc)=0\qquad(c\in H^1(\Q,V)).
\]
The parity theorem of Klagsbrun--Mazur--Rubin consequently applies to $\mathcal F$ and the modified structure \cite[Theorem 3.9]{KMR}, yielding
\[
\dim S-\dim S'
\equiv\dim\frac{\alpha_q}{\alpha_q\cap\beta_q}=1\pmod2.
\]
It follows that $\dim(S'/S)=1$, as required. For structures arising from global quadratic twists, this recovers the one-place variation formula of \cite[Proposition 3.3, Corollary 3.4]{MR}.
\end{proof}

\subsection{Terminal pairs with uniformly bounded support}\label{subsec:terminal-proof}
For finite $T\subset\mathcal P_L$, start with $\Sel_2(E/\Q)$ and $\Sel_2(E^D/\Q)$ and replace $\alpha_q$ by $\beta_q$ only at $q\in T$. Denote the resulting mixed Selmer groups by $S_i(T)$. At previously chosen primes, subsequent unramified twists preserve $\beta_q$, so the construction is compatible under enlarging $T$. Throughout this subsection, assume \Heegner, \Surjectivity{L}, and \Exclusion.

\begin{proposition}\label{prop:actual-pair}
Put $d=d_T$. For every $T\in\mathscr A_L$, the following unordered pairs agree in $H^1(\Q,V)$:
\begin{equation}\label{eq:actual-pair}
\{S_0(T),S_1(T)\}
=\{\Sel_2(E^d/\Q),\Sel_2(E^{dD}/\Q)\}.
\end{equation}
The original labels agree if $\Delta_E<0$ or $|T|$ is even. If $\Delta_E>0$ and $|T|$ is odd, the two labels are interchanged.
\end{proposition}
\begin{proof}
If $\Delta_E<0$ or $|T|$ is even, the construction of $\mathcal P_L$ makes $d$ a square at every finite $v\mid2ND$. Thus both original Kummer conditions are unchanged there. At a new prime $q\in T$, both become $\beta_q$; at all other finite places, Lemma~\ref{lem:local-twists} gives unramified invariance at good reduction. If $\Delta_E<0$, the real cohomology vanishes; if $|T|$ is even, $d>0$ preserves the real conditions. This proves the assertion with the original labels.

It remains to consider $\Delta_E>0$ and odd $|T|$. By the definition of $\mathscr A_L$, this case has $\intr(E,K)=0$. At the original odd bad primes, $D$ is a local square and $d$ gives a trivial or unramified twist; $T_E=0$ and Lemma~\ref{lem:unramified-odd} preserve both Kummer conditions. At odd $r\mid D$, the parameter $d$ is a local square. At $2$, if $D$ is even, $d$ is a square and there is nothing to prove. Suppose now that $D$ is odd. Both $D$ and $d$ give trivial or unramified twists. If $E$ has good reduction at $2$, so does $E^D$, and Lemma~\ref{lem:local-twists} applies. If $E$ has bad reduction at $2$, the Heegner hypothesis makes $D$ a square, while $c_2(E)$ is odd, so Lemma~\ref{lem:unramified-odd} applies. The condition at $2$ is therefore preserved in every case. At new primes it is $\beta_q$, and the remaining finite places are covered by unramified invariance at good reduction.

Now $d<0$ exchanges the real Kummer lines of $E$ and $E^D$. Equation~\eqref{eq:finite-local-equal} says that the initial finite conditions agree, and both are changed to $\beta_q$ at each new prime. Exchanging the real conditions therefore exchanges the entire Selmer structures, proving \eqref{eq:actual-pair}.
\end{proof}

Membership of $\xi_E$ in $S_i(T)$ is unchanged throughout the construction: it localizes to zero at every new prime and hence satisfies both local conditions there, while all original conditions are retained. Thus \Exclusion\ continues to hold for the two mixed groups.

\begin{proposition}\label{prop:general-terminal}
Assume \Heegner, \Surjectivity{L}, and \Exclusion. Put $r_*=\max(a_0,a_1)$ and
\[
\varepsilon_i=\dim_{\F_2}\bigl(\Sel_2(E^{D^i}/\Q)\cap\F_2\xi_E\bigr).
\]
Among terminal supports, restricting to even size when $\Delta_E>0$, the possible sizes are exactly
\[
\begin{cases}
 r_*-1+\Z_{\ge0},
   &\Delta_E<0,\quad\varepsilon_0+\varepsilon_1=0,\\
 \max(a_0-\varepsilon_0,a_1-\varepsilon_1)+2\Z_{\ge0},
   &\Delta_E<0,\quad\varepsilon_0+\varepsilon_1=1,\\
 2\left\lceil(r_*-1)/2\right\rceil+2\Z_{\ge0},
   &\Delta_E>0.
\end{cases}
\]
If $\varepsilon_0+\varepsilon_1=1$, the ordered terminal dimensions are
$(\varepsilon_0,\varepsilon_1)$. Every positive size in this list occurs
for infinitely many supports, which may avoid any prescribed finite set.
For every such support, $n_T\in\mathcal N_L(E,K)$.
\end{proposition}
\begin{proof}
Condition \Exclusion\ gives $\varepsilon_0+\varepsilon_1\le1$, and the real calculation in Section~\ref{subsec:xi} gives $\varepsilon_0=\varepsilon_1=0$ when $\Delta_E>0$. By \eqref{eq:kramer-parity}, $a_0+a_1$ is odd. Each switch changes both
dimensions by one, so their sum stays odd. Lemma~\ref{lem:detection}
allows simultaneous nonzero localization whenever both groups have
nonzero images modulo $\F_2\xi_E$. It also allows both localizations
to vanish. If one group is $0$ or $\F_2\xi_E$, its localization is
necessarily zero, while the other may be detected provided its image
modulo $\F_2\xi_E$ is nonzero. Lemma~\ref{lem:switch} turns these choices
into the dimension changes used below.

First suppose $\varepsilon_0=\varepsilon_1=0$. Relabel so that
$a_0\le a_1$. Detect both groups for $a_0$ steps, reaching dimensions
$(0,a_1-a_0)$. The remaining nonzero dimension is odd. For every odd
$y\ge3$, two new primes give
\[
 (0,y)\longrightarrow(1,y-1)\longrightarrow(0,y-2).
\]
At the first prime the nonzero group is detected; at the second both
nonzero groups are detected. Thus $(0,1)$ is reached after exactly
$a_1-1=r_*-1$ switches. From a terminal pair, detect its nonzero group;
one further switch gives $(1,0)$, and repeating this gives a terminal
pair at every larger support size. When $\Delta_E>0$, take precisely the
even sizes in this construction.

Now suppose $(\varepsilon_0,\varepsilon_1)=(0,1)$. The second group
contains $\xi_E$ throughout, so its terminal dimension must be one.
Detect both groups for $\min(a_0,a_1-1)$ steps. If $a_0\ge a_1-1$,
the dimensions become $(x,1)$ with $x=a_0-a_1+1$ even. The
one-dimensional second group is now $\F_2\xi_E$. For every $x\ge2$,
use
\[
 (x,1)\longrightarrow(x-1,2)\longrightarrow(x-2,1).
\]
The first prime detects the first group; the second detects both groups,
whose images modulo $\F_2\xi_E$ are nonzero. This reaches $(0,1)$
after a total of $a_0$ switches. If $a_0<a_1-1$, the dimensions instead
become $(0,y)$ with $y=a_1-a_0$ odd and at least three. The previous
moves $(0,y)\to(1,y-1)\to(0,y-2)$ apply: in the middle state the
second group has dimension at least two, so its image modulo the kernel
class is nonzero. They reach $(0,1)$ in a total of $a_1-1$ switches.
The resulting minimum is therefore $\max(a_0,a_1-1)$.
From $(0,1)$, choose a prime where both localizations vanish and then
one detecting both enlarged groups, giving
\[
 (0,1)\longrightarrow(1,2)\longrightarrow(0,1).
\]
This realizes every larger size of the same parity. The case
$(\varepsilon_0,\varepsilon_1)=(1,0)$ is symmetric.

For the lower bounds, every switch changes each dimension by one, hence
\[
 \dim S_i(T)\ge a_i-|T|,\qquad
 \dim S_i(T)\equiv a_i+|T|\pmod2.
\]
When both $\varepsilon_i$ vanish, a terminal pair has each dimension
at most one, so $|T|\ge r_*-1$. In the positive-discriminant case
the stated restriction further requires even size. When exactly one
$\varepsilon_i$ is one, the terminal dimensions are forced to be
$(\varepsilon_0,\varepsilon_1)$, yielding
\[
 |T|\ge\max(a_0-\varepsilon_0,a_1-\varepsilon_1),\qquad
 |T|\equiv a_i-\varepsilon_i\pmod2.
\]
The two expressions on the right in the maximum have the same parity,
since $a_0+a_1$ is odd. These lower bounds match the constructions.
Proposition~\ref{prop:actual-pair} identifies the final
mixed groups with the actual twist groups. At each step the required
primes are infinite and can avoid every previously chosen prime and any
prescribed finite set. This proves the infinitude assertions for every
positive support size. Each displayed minimum is at most $r_*$. If the minimum is zero, size two is permitted; hence infinitely many nonempty supports have a common size at most $\max\{r_*,2\}$, as asserted in Theorem~\ref{thm:main-descent}. Finally, \eqref{eq:PL-properties} gives
$n_T\in\mathcal N_L(E,K)$.
\end{proof}

\begin{proposition}\label{prop:terminal}
Under \Heegner, \Surjectivity{L}, \Exclusion, and $\intr(E,K)=0$, the minimum support size among $T\in\mathscr T_L$ is $R=\min(a_0,a_1)$. Infinitely many minimizing supports exist when $R>0$, and may avoid any finite set. If $R=0$, two primes give infinitely many nonempty terminal supports.
\end{proposition}
\begin{proof}
By Proposition~\ref{prop:nesting}, label the initial groups
$S_{\mathrm{small}}\subset S_{\mathrm{large}}$, with dimensions $R,R+1$. Condition \Exclusion\ gives
$S_{\mathrm{small}}\cap\F_2\xi_E=0$.
As long as the smaller group is nonzero, Lemma~\ref{lem:detection}, applied to the sum modulo the kernel class, supplies a new prime $q$ with $\loc_qS_{\mathrm{small}}\ne0$. Inclusion then gives $\loc_qS_{\mathrm{large}}\ne0$. Lemma~\ref{lem:switch} yields simultaneously
\[
S'_{\mathrm{small}}=\ker(\loc_q|_{S_{\mathrm{small}}})
\subset
\ker(\loc_q|_{S_{\mathrm{large}}})=S'_{\mathrm{large}}.
\]
Both dimensions decrease by one, while nesting and exclusion of the kernel class from the smaller group persist. After exactly $R$ steps, the dimensions are $(0,1)$. Proposition~\ref{prop:actual-pair} identifies these with the actual twist groups.

If $R=0$, the empty support is already terminal. To obtain a nonempty support, first choose a prime at which both localizations vanish. Lemma~\ref{lem:switch} gives dimensions $(1,2)$. The first new group still excludes $\xi_E$, and the second remains nonzero after quotienting by the at most one-dimensional kernel. Lemma~\ref{lem:detection} therefore supplies a second prime detecting both groups, returning the dimensions to $(0,1)$. Each choice admits infinitely many primes and may avoid the existing support and any prescribed finite set. This proves the infinitude assertions. The support size depends only on the initial Selmer dimensions, not on $L$.

For the lower bound, let $T\in\mathscr T_L$ be arbitrary and order its primes in any way. At each step Lemma~\ref{lem:switch} changes each mixed Selmer dimension by exactly one in absolute value. Therefore
\[
\dim_{\F_2}S_i(T)\ge a_i-|T|\qquad(i=0,1).
\]
By Proposition~\ref{prop:actual-pair}, one of $S_0(T),S_1(T)$ is zero. For that index $i$, $|T|\ge a_i\ge R$. The constructed support attains this lower bound, proving the intrinsic-index-zero assertion of Theorem~\ref{thm:main-descent}.
\end{proof}

\begin{remark}\label{rem:nonempty-terminal}
Assume $\intr(E,K)=0$ and $R=0$, and let $B$ be the initial one-dimensional larger Selmer group. If $B\ne\F_2\xi_E$, Lemma~\ref{lem:detection} supplies a prime detecting $B$, while the zero group has zero localization. One switch changes the mixed dimensions from $(0,1)$ to $(1,0)$, so the smallest nonempty terminal support has size one. If $B=\F_2\xi_E$, its localization vanishes at every prime of $\mathcal P_L$. Every single switch then gives $(1,2)$, and no singleton is terminal. The two-step construction above shows that the smallest nonempty size is exactly two. The identification with actual twists is unordered, so this argument also covers the real-place exchange.
\end{remark}

\section{Full Heegner derivatives and genus points}\label{sec:cm}
For finite $T\subset\mathcal P_L$, we construct a point $Q_T\in E^{d_T}(K)$ whose mod-$2$ Kummer class is $\kappa^{\std}_{n_T,1}(\varphi)$ under the natural identification $E^{d_T}[2]=E[2]$. We obtain this point by expressing a quadratic character projection as an integral linear combination of Heegner traces. We follow the Heegner-point and derivative conventions of \cite[\S\S3--5]{GrossKolyvagin} and \cite[\S\S4--5]{Stein}.

Genus traces and character decompositions are central to Tian's work on congruent numbers. In \cite[Lemma 4.3]{TianCongruent}, traces to quadratic subfields are related to a genus trace by a power of $2$, with a torsion correction in one case. Tian--Yuan--Zhang \cite[\S\S3.1, 3.3]{TYZGenus} develop this method further, constructing corrected genus points by subtracting contributions indexed by proper factorizations. Here we fix $K$ and vary the ring class conductor. The coefficients $a_q(E)/2$ in the correction below come directly from the Heegner norm relations; these give an integral representative of the required character component and its comparison with the standard derivative modulo $2$.

\subsection{Ring class fields, genus characters, and integral traces}
Fix $T\subset\mathcal P_L$. For $S\subseteq T$, put
\[
n_S=\prod_{q\in S}q,\qquad\mathcal O_{n_S}=\Z+n_S\mathcal O_K,
\qquad H_S=H_{n_S}.
\]
Here $H_S$ is the ring class field of the order $\mathcal O_{n_S}$, and $H_\varnothing=H_1$ is the Hilbert class field. Choose an ideal of norm $N$ using the Heegner hypothesis. This gives compatible CM points
\[
x_S\in X_0(N)(H_S),\qquad y_S=\varphi(x_S)\in E(H_S).
\]
For $q\in S$, the prime $q$ is inert in $K$, and the $q+1$ conjugates of $x_S$ in the relative extension correspond to the terms of the Hecke correspondence $T_q$. The modular parametrization converts the $T_q$ action to multiplication by $a_q(E)$, giving \cite[Proposition 4.2]{Stein}
\begin{equation}\label{eq:CM-norm}
\Tr_{H_S/H_{S\setminus\{q\}}}y_S=a_q(E)y_{S\setminus\{q\}}.
\end{equation}
The normalization $\varphi(\infty)=O$ fixes the translation at the cusp.

Let $G=\Gal(H_T/K)$. The ring class group exact sequence and Artin reciprocity \cite[\S3]{GrossKolyvagin}, \cite[Proposition 6.2]{Stein} identify the relative group as
\[
\Gal(H_T/H_1)\simeq
\frac{(\mathcal O_K/n_T\mathcal O_K)^\times}{(\Z/n_T\Z)^\times}
\simeq\prod_{q\in T}\frac{\F_{q^2}^\times}{\F_q^\times}.
\]
The image of $\mathcal O_K^\times=\{\pm1\}$ is already contained in the denominator. The second isomorphism follows from the Chinese remainder theorem, and each factor is cyclic of order $q+1$. Thus
\begin{equation}\label{eq:ring-exact}
1\longrightarrow G_{\mathrm{rel}}=\Gal(H_T/H_1)\simeq\prod_{q\in T}C_{q+1}
\longrightarrow G\longrightarrow\Cl(K)\longrightarrow1.
\end{equation}
Choose a generator $\sigma_q$ of each relative cyclic factor, and put
\[
F_S=K(\sqrt{-q}:q\in S).
\]
We verify that $F_S\subset H_S$. Since $q\equiv3\pmod4$, the integer $-q$ is the fundamental discriminant of $\Q(\sqrt{-q})$. Let $\vartheta_q$ be its quadratic Hecke character, ramified at no finite prime other than $q$. Writing $\mathbb A_F^\times$ for the id\`ele group of a number field $F$, define
\[
\chi_q(x)=\vartheta_q\bigl(N_{K/\Q}x\bigr),
\qquad x\in\mathbb A_K^\times.
\]
Compatibility of Artin reciprocity with norm and restriction identifies the extension cut out by $\chi_q$ with $K(\sqrt{-q})/K$. To see that it is a ring class character, it remains to check triviality on the local unit groups of $\mathcal O_q$.

Away from $q$, the character $\vartheta_q$ is unramified and is therefore trivial on norms of local units. At $q$, the extension $K_q/\Q_q$ is unramified quadratic, and
\[
\mathcal O_q\otimes\Z_q=\Z_q+q\mathcal O_{K_q}.
\]
Every unit in this order reduces to some $a\in\F_q^\times$, whose norm modulo $q$ is $a^2$. Thus $\vartheta_q$ is also trivial on these norms. At infinity, norms are positive real numbers, on which it is again trivial. Hence $\chi_q$ factors through the ring class group
\[
\Cl(\mathcal O_q)\simeq
K^\times\backslash\mathbb A_{K,f}^\times/\widehat{\mathcal O}_q^\times,
\]
and $K(\sqrt{-q})\subset H_q$. Here $\mathbb A_{K,f}^\times$ is the finite id\`ele group. This applies the reciprocity description of ring class fields to the character in question \cite[\S3]{GrossKolyvagin}, \cite[\S4, Proposition 6.2]{Stein}.

On the relative cyclic factor, this character is
\[
\F_{q^2}^\times/\F_q^\times\longrightarrow\{\pm1\},\qquad
\bar u\longmapsto
\left(\frac{N_{\F_{q^2}/\F_q}(\bar u)}{q}\right).
\]
The finite-field norm is surjective, so this quadratic character is nontrivial and its kernel is the subgroup of squares. Its conductor exponent at the unique prime of $K$ above $q$ is one, and it is unramified at all other finite places. The characters for distinct $q$ are linearly independent: any nonempty product remains ramified at each selected prime. Their compositum therefore has degree $2^{|T|}$, and each nontrivial quadratic subextension ramifies at some $q\in T$. Since the Hilbert class field is everywhere unramified,
\begin{equation}\label{eq:genus-degrees}
[F_T:K]=2^{\nu},\qquad F_T\cap H_1=K,\qquad \nu=|T|.
\end{equation}
Let $\tau_q\in\Gal(F_T/K)$ change the sign of $\sqrt{-q}$ and fix the other square roots. Then $\sigma_q|_{F_T}=\tau_q$.

Put $G_F=\Gal(H_T/F_T)$. Since $F_T\cap H_1=K$, restriction $G_F\to\Gal(H_1/K)=\Cl(K)$ is surjective. Its kernel is $G_F\cap G_{\mathrm{rel}}$, and the character calculation above gives
$G_F\cap G_{\mathrm{rel}}=G_{\mathrm{rel}}^2=\prod_q\langle\sigma_q^2\rangle$. We may therefore lift each ideal class to $G_F$, obtaining a complete set $\mathcal C$ of representatives for $G/G_{\mathrm{rel}}$. Every element of $G_F$ has a unique expression as a representative times an element of $G_{\mathrm{rel}}^2$; thus
\begin{equation}\label{eq:CR2}
\mathcal C G_{\mathrm{rel}}^2=G_F\quad\text{with each element occurring once}.
\end{equation}

Define the unnormalized trace points
\[
R_S=\Tr_{H_S/F_S}y_S\in E(F_S).
\]
View all $R_S$ as points of $E(F_T)$. For $q\in S$, the operator $1+\tau_q$ is the trace for $F_S/F_{S\setminus\{q\}}$. Computing the same trace through the towers $H_S/F_S/F_{S\setminus\{q\}}$ and $H_S/H_{S\setminus\{q\}}/F_{S\setminus\{q\}}$, and using \eqref{eq:CM-norm}, gives
\begin{equation}\label{eq:R-norm}
(1+\tau_q)R_S=a_q(E)R_{S\setminus\{q\}}\qquad(q\in S).
\end{equation}

\paragraph{Character projection and integral correction.}
Let $d=d_T$ and let $\chi_T$ be the character of $K(\sqrt d)/K$, viewed as a character of $\Gal(F_T/K)$. Thus $\chi_T(\tau_q)=-1$ for every $q\in T$. On $E(F_T)\otimes_{\Z}\Z[1/2]$, the projector to this character component is
\begin{equation}\label{eq:genus-projector}
e_{\chi_T}
=\frac{1}{2^\nu}\sum_{\gamma\in\Gal(F_T/K)}\chi_T(\gamma)\gamma
=\prod_{q\in T}\frac{1-\tau_q}{2}.
\end{equation}
We seek a point of $E(F_T)$ representing $e_{\chi_T}R_T$ and satisfying the character relation exactly, so that it descends to $E^d(K)$.

For a single prime $T=\{q\}$, the norm relation \eqref{eq:R-norm} gives
\[
\frac{1-\tau_q}{2}R_{\{q\}}
=R_{\{q\}}-\frac{a_q(E)}2R_\varnothing
\quad\text{in }E(F_{\{q\}})\otimes_{\Z}\Z[1/2].
\]
Since $a_q(E)$ is even, the right-hand side defines an actual point
\[
Z_{\{q\}}=R_{\{q\}}-\frac{a_q(E)}2R_\varnothing\in E(F_{\{q\}}).
\]
Moreover, $(1+\tau_q)Z_{\{q\}}=0$, because $\tau_q$ fixes $R_\varnothing$. Subtracting this lower-conductor trace therefore gives an integral representative of the anti-invariant component.

For general $T$, put $b_q=a_q(E)/2\in2\Z$. Expanding the product of projectors in \eqref{eq:genus-projector} using \eqref{eq:R-norm} leads to
\begin{equation}\label{eq:ZT}
Z_T=\sum_{S\subseteq T}(-1)^{|T\setminus S|}
\left(\prod_{q\in T\setminus S}b_q\right)R_S\in E(F_T).
\end{equation}
All coefficients are integers. For two primes, for example,
\[
Z_{\{p,q\}}=R_{\{p,q\}}-b_pR_{\{q\}}-b_qR_{\{p\}}
+b_pb_qR_\varnothing.
\]
The following lemma proves the character relation in $E(F_T)$ itself and identifies the corresponding full character trace. The stronger divisibility $4\mid a_q(E)$ makes every $b_q$ even; it will give the congruence with the standard derivative in Section~\ref{subsec:mod2}.

\begin{lemma}\label{lem:genus-point}
For every $q\in T$, we have $(1+\tau_q)Z_T=0$. Regard $\chi_T$ as a character of $G$ by inflation. Then
\begin{equation}\label{eq:full-character}
P_{\chi_T}:=\sum_{g\in G}\chi_T(g)g(y_T)
=\prod_{q\in T}(1-\tau_q)R_T=2^{\nu}Z_T.
\end{equation}
Consequently, there is a unique $Q_T\in E^d(K)$ with $Z_T=\iota_d(Q_T)$.
\end{lemma}
\begin{proof}
Fix $q\in T$ and pair each subset $S$ not containing $q$ in \eqref{eq:ZT} with $S\cup\{q\}$. This gives
\[
Z_T=\sum_{S\subseteq T\setminus\{q\}}
(-1)^{|T\setminus(S\cup\{q\})|}
\left(\prod_{r\in T\setminus(S\cup\{q\})}b_r\right)
\bigl(R_{S\cup\{q\}}-b_qR_S\bigr).
\]
Since $\tau_qR_S=R_S$, applying $1+\tau_q$ changes each parenthesis to $(a_q-2b_q)R_S=0$, proving anti-invariance.

Next expand $1-\tau_q=2-(1+\tau_q)$ and use the trace relation for each factor $1+\tau_q$. We obtain
\[
\begin{aligned}
\prod_{q\in T}(1-\tau_q)R_T
&=\sum_{U\subseteq T}(-1)^{|U|}2^{\nu-|U|}
\left(\prod_{q\in U}a_q(E)\right)R_{T\setminus U}\\
&=2^\nu Z_T.
\end{aligned}
\]
On the other hand, sum the full character trace first over cosets of $G_F$. The inner sum is $R_T$, and the signed sum on $G/G_F=\langle\tau_q:q\in T\rangle$ is exactly $\prod_q(1-\tau_q)$. This proves \eqref{eq:full-character}.

Each $\tau_q$ acts as $-1$ on $Z_T$, so $gZ_T=\chi_T(g)Z_T$ for $g\in G_K$. Since $g\iota_d=\chi_T(g)\iota_d$, the two character factors cancel. Thus $\iota_d^{-1}Z_T$ is fixed by $G_K$ and defines $Q_T\in E^d(K)$. Uniqueness follows from injectivity of $\iota_d$.
\end{proof}

Thus $\iota_d(Q_T)=Z_T$ represents $e_{\chi_T}R_T$, and the integral formula \eqref{eq:ZT} constructs it from the trace points. For $T=\varnothing$, the character is trivial and $e_{\chi_T}=1$, $d_T=1$, and the empty-product conventions give $Q_\varnothing=Z_\varnothing=R_\varnothing=y_K$. The identity $P_{\chi_T}=2^\nu\iota_d(Q_T)$ will relate the lattice index of $Q_T$ to the height of the full character trace.

\subsection{Descent of the standard derivative and comparison modulo two}\label{subsec:mod2}
We use the full derivative and coset-sum construction of \cite[\S\S5.1--5.2]{Stein}.
Put $n=n_T$. For $m=2^M$, $1\le M\le L$, define
\[
D_{q+1}(\sigma_q)=\sum_{j=0}^{q}j\sigma_q^j,\qquad
D_n=\prod_{q\mid n}D_{q+1}(\sigma_q).
\]
Choose a complete set $\mathcal C$ of representatives for $G/G_{\mathrm{rel}}$, and put
\begin{equation}\label{eq:Wn}
W_n=\left(\sum_{a\in\mathcal C}a\right)D_n y_T\in E(H_T).
\end{equation}
The group-ring identity
\begin{equation}\label{eq:derivative-identity}
(\sigma_q-1)D_{q+1}(\sigma_q)=(q+1)-\Norm_q,
\qquad\Norm_q=\sum_{j=0}^{q}\sigma_q^j
\end{equation}
and the divisibilities $m\mid q+1,a_q$ imply that $D_ny_T$ is invariant modulo $m$ under $G_{\mathrm{rel}}$. Indeed, applying the other derivative operators to $\Norm_qy_T=a_q(E)y_{T\setminus\{q\}}$ puts the right-hand side of \eqref{eq:derivative-identity}, applied to the point, in $mE(H_T)$. Replacing a representative $a$ by $ar$, $r\in G_{\mathrm{rel}}$, therefore does not change $aD_ny_T$ modulo $m$. An element $g\in G$ merely permutes the cosets. Hence $W_n$ modulo $m$ is independent of the representatives and is $G$-invariant.

\begin{lemma}\label{lem:no-torsion}
Under \Heegner\ and residual image $S_3$,
\[
K\cap\Q(E[2])=\Q,\qquad E(H_T)[2^\infty]=0.
\]
Consequently, restriction induces an isomorphism
\begin{equation}\label{eq:unique-descent}
H^1(K,E[2^M])\xrightarrow{\ \sim\ }
H^1(H_T,E[2^M])^{G}.
\end{equation}
\end{lemma}
\begin{proof}
If an odd prime $r$ divides $D$, then $E$ has good reduction at $r$. The $2$-division field is unramified at $r$, whereas $K$ ramifies there, so $K$ cannot be its quadratic subfield. If $D$ has no odd prime factor, the only possibility with $D<-4$ is $D=-8$. Then $E$ has good reduction at $2$ and $v_2(\Delta_E)=0$, so the unique quadratic subfield $\Q(\sqrt{\Delta_E})$ of the $2$-division field cannot be $\Q(\sqrt{-2})$. Thus the residual image of $G_K$ remains $S_3$.

If the abelian extension $H_T/K$ contained a nonzero $2$-torsion point, the $G_{H_T}$-invariants would form a nonzero $S_3$-submodule of $V$. Irreducibility would force $K(E[2])\subseteq H_T$, contradicting $\Gal(K(E[2])/K)=S_3$. Any nonzero point of $2$-power order has a multiple of order two, so $E(H_T)[2^\infty]=0$. Now apply the inflation--restriction sequence
\[
\begin{aligned}
0&\longrightarrow H^1(G,E(H_T)[2^M])
\longrightarrow H^1(K,E[2^M])\\
&\longrightarrow H^1(H_T,E[2^M])^G
\longrightarrow H^2(G,E(H_T)[2^M]).
\end{aligned}
\]
Both end terms vanish, proving \eqref{eq:unique-descent}; see \cite[\S4]{GrossKolyvagin} and \cite[\S5.2]{Stein}.
\end{proof}

\begin{definition}\label{def:standard}
Define $\kappa^{\std}_{n,M}(\varphi)$ to be the unique descent of $\delta_{E,H_T,2^M}(W_n)$ under \eqref{eq:unique-descent}. For $n=1$, use the Kummer class of $\Tr_{H_1/K}y_1$.
\end{definition}
For $j\le M$, write $r_{M,j}=[2^{M-j}]:E[2^M]\to E[2^j]$ and $i_{j,M}:E[2^j]\hookrightarrow E[2^M]$. Functoriality of Kummer maps gives
\begin{equation}\label{eq:coefficient-compatibility}
r_{M,j,*}\kappa^{\std}_{n,M}=\kappa^{\std}_{n,j},\qquad
[2^{M-j}]\kappa^{\std}_{n,M}=i_{j,M,*}\kappa^{\std}_{n,j}.
\end{equation}
Each $i_{j,M,*}$ is injective. Indeed, the cohomology sequence of $0\to E[2^j]\to E[2^M]\to E[2^{M-j}]\to0$ shows that its kernel comes from $E(K)[2^{M-j}]$, which vanishes by Lemma~\ref{lem:no-torsion}.

The standard derivative construction and
\eqref{eq:coefficient-compatibility} apply to every
$n\in\mathcal N_M(E,K)$. Set $T=\operatorname{Supp}(n)$ and $H_T=H_n$,
and choose representatives of
$\Gal(H_n/K)/\Gal(H_n/H_1)$. The construction uses the same CM norm
relation and the admissibility divisibilities; the torsion argument of
Lemma~\ref{lem:no-torsion} applies because $H_n/K$ is abelian.
For $n=1$, use the basic trace as in Definition~\ref{def:standard}.

\paragraph{Comparison modulo two.}
\begin{proposition}\label{prop:mod2-comparison}
For finite $T\subset\mathcal P_L$, under the natural identification of $2$-torsion,
\begin{equation}\label{eq:mod2-equality}
\kappa^{\std}_{n,1}(E,K;\varphi)
=\delta_{E^{d_T},K,2}(Q_T)\quad\text{in }H^1(K,V).
\end{equation}
In particular, $\kappa^{\std}_{n,1}\ne0$ if and only if $Q_T\notin2E^{d_T}(K)$.
\end{proposition}
\begin{proof}
In \eqref{eq:Wn}, choose $\mathcal C\subset G_F$. The odd powers of $\sigma_q$ form the coset $\sigma_q\langle\sigma_q^2\rangle$, so in the group ring modulo two,
\[
D_{q+1}(\sigma_q)\equiv\sigma_q\sum_{h\in\langle\sigma_q^2\rangle}h\pmod2.
\]
By \eqref{eq:CR2},
\begin{equation}\label{eq:W-mod2}
W_n\equiv\left(\prod_{q\in T}\tau_q\right)R_T
\equiv R_T\equiv Z_T\pmod{2E(H_T)}.
\end{equation}
For the intermediate congruences, \eqref{eq:R-norm} gives $\tau_qR_T=a_qR_{T\setminus\{q\}}-R_T\equiv R_T\pmod{2E(H_T)}$, allowing each $\tau_q$ to be removed. Every term of \eqref{eq:ZT} except $S=T$ contains an even factor $b_q$, so $Z_T\equiv R_T\pmod{2E(H_T)}$. The assertion $b_q\in2\Z$ uses $4\mid a_q$, ensured by $L\ge2$.

The isomorphism $\iota_d$ is defined over $H_T$, so Kummer functoriality identifies the restrictions of the two classes to $H_T$. Explicitly, we have a commutative diagram
\[
\begin{tikzcd}[column sep=large,row sep=large]
E^d(K)/2 \arrow[r,"\delta_{E^d,K,2}"] \arrow[d,"\iota_d"']
& H^1(K,V) \arrow[d,"\res"]\\
E(H_T)/2 \arrow[r,"\delta_{E,H_T,2}"']
& H^1(H_T,V).
\end{tikzcd}
\]
The right vertical map is injective by Lemma~\ref{lem:no-torsion}. Equality over $H_T$ therefore descends to $K$, proving \eqref{eq:mod2-equality}.
\end{proof}

Assume $T_E=0$. For a support $T\subset\mathcal P_L$, \eqref{eq:mod2-equality} also places the standard class in the Selmer group for $E/K$ with transverse conditions at $n=n_T$. Indeed, for $q\mid n$ and $\lambda\mid q$, the prescribed Frobenius on $E[4]$ has square one, so $E[4]\subset E(K_\lambda)$. The extension $K_\lambda/\Q_q$ is unramified quadratic, so every $\Q_q$-unit becomes a square there, giving $K_\lambda(\sqrt d)=K_\lambda(\sqrt q)$. The twisted Kummer cocycles $g\mapsto\bar\chi(g)P$, $P\in V$, give
\[
\begin{aligned}
\alpha_{K_\lambda}(E^d)&=H^1_{\trans}(K_\lambda,V)\\
&=\ker\bigl(H^1(K_\lambda,V)\to H^1(K_\lambda(\sqrt q),V)\bigr).
\end{aligned}
\]
At the original bad primes, the Heegner hypothesis makes $K/\Q$ split. The assumption $T_E=0$ and Lemma~\ref{lem:unramified-odd} handle the trivial or unramified twist by $d$. At odd $r\mid D$, the parameter $d$ is a square; if $D$ is even, it is also a square at $2$. At the remaining good places, including a good dyadic place, the local extension is split or unramified, and Lemma~\ref{lem:local-twists} preserves the usual Kummer condition. The infinite places of $K$ are complex. Thus the class satisfies the transverse Selmer conditions at $n$ and the usual conditions elsewhere; cf.\ \cite[\S6]{GrossKolyvagin}.

\section{The index formula and nonvanishing}\label{sec:index}
\subsection{Normalization of modular parametrizations}\label{sec:normalization}
\paragraph{The optimal quotient and its differential.}
The newform $f_E$ determines an optimal elliptic quotient
\[
p_{\mathrm{opt}}:J_0(N)\twoheadrightarrow E_{\mathrm{opt}}
\]
over $\Q$, with connected kernel, in the isogeny class of $E$. Composing it with the Abel--Jacobi map $P\mapsto[P]-[\infty]$ gives
\[
\pi:X_0(N)\longrightarrow E_{\mathrm{opt}},\qquad
\pi(P)=p_{\mathrm{opt}}([P]-[\infty]),\qquad \pi(\infty)=O.
\]
For a minimal N\'eron differential $\omega_{E_{\mathrm{opt}}}$, the optimal Manin constant theorem \cite[Theorem 1.1]{CLN} gives
\[
\pi^*\omega_{E_{\mathrm{opt}}}
=c_\pi f_E(q)\frac{dq}{q},\qquad c_\pi=\pm1.
\]
Here $q=e^{2\pi i\tau}$ and $f_E(q)=\sum_{m\ge1}a_mq^m$, so $f_E(q)\,dq/q=2\pi i f_E(\tau)\,d\tau$, as in our definition of $c_\varphi$.

\paragraph{Normalization on the fixed target curve.}
Recall the divisibility exponent $b_\varphi$ from \eqref{eq:bphi}.
\begin{proposition}\label{prop:manin-normalization}
Suppose that $E[2]$ is irreducible as a $G_\Q$-module. For every fixed parametrization $\varphi:X_0(N)\to E$ with $\varphi(\infty)=O$, there is a rational isogeny $\alpha_0:E_{\mathrm{opt}}\to E$ of odd degree such that, on setting $\varphi_0=\alpha_0\circ\pi$,
\begin{equation}\label{eq:manin-normalization}
\varphi=[2^{b_\varphi}]\circ\varphi_0,\qquad
v_2(c_{\varphi_0})=0,\qquad b_\varphi=v_2(c_\varphi).
\end{equation}
In particular, $\varphi$ is $2$-indivisible if and only if $c_\varphi$ is a $2$-adic unit.
\end{proposition}
\begin{proof}
Since $\varphi(\infty)=O$, the universal property of the Jacobian gives a homomorphism $h:J_0(N)\to E$ over $\Q$. Put $B=\ker p_{\mathrm{opt}}$. The pullbacks $h^*\omega_E$ and $p_{\mathrm{opt}}^*\omega_{E_{\mathrm{opt}}}$ lie in the same $f_E$-line, so $dh$ vanishes on $\operatorname{Lie}B$. In characteristic zero the image $h(B)$ is therefore zero-dimensional; connectedness of $B$ gives $h(B)=0$. Factoring through $J_0(N)/B$ yields a unique rational isogeny $\alpha$ such that
\[
\varphi=\alpha\circ\pi,\qquad \alpha:E_{\mathrm{opt}}\to E
\]
see also \cite[proof of Lemma 6.5]{CNS}. Choose the largest $b$ for which $\alpha=[2^b]\circ\alpha_0$; it exists since $4^b$ divides $\deg\alpha$. If $\varphi=[2^k]\circ\psi$ as in \eqref{eq:bphi}, then $\psi^*\omega_E=2^{-k}\varphi^*\omega_E$ lies in the same $f_E$-line. Thus $\psi=\beta\circ\pi$, and surjectivity of $\pi$ gives $\alpha=[2^k]\circ\beta$. Conversely, every such factorization of $\alpha$ gives one of $\varphi$. Hence $b=b_\varphi$.

If $\deg\alpha_0$ were even, the finite group $\ker\widehat\alpha_0(\overline\Q)$ would contain an element of order two. Thus $\ker\widehat\alpha_0\cap E[2]$ is a nonzero Galois-stable subspace, and residual irreducibility forces it to be all of $E[2]$. The quotient by $E[2]$ gives a rational isogeny $\gamma:E\to E_{\mathrm{opt}}$ with
\[
\widehat\alpha_0=\gamma\circ[2]_E,\qquad
\alpha_0=[2]_E\circ\widehat\gamma.
\]
Consequently $\alpha=[2^{b+1}]_E\circ\widehat\gamma$, contradicting maximality. Thus $\deg\alpha_0$ is odd.

For minimal differentials, write
\[
\alpha_0^*\omega_E=u\omega_{E_{\mathrm{opt}}},\qquad
\widehat\alpha_0^{\,*}\omega_{E_{\mathrm{opt}}}=u'\omega_E.
\]
Isogenies extend to N\'eron models, so $u,u'\in\Z$. Pulling back a differential along $\widehat\alpha_0\circ\alpha_0=[\deg\alpha_0]$ gives $uu'=\deg\alpha_0$; hence both are $2$-adic units. By \cite[Theorem 1.1]{CLN}, $c_\pi=\pm1$. Setting $\varphi_0=\alpha_0\circ\pi$ and using $[2^b]^*\omega_E=2^b\omega_E$ gives
\[
c_{\varphi_0}=u c_\pi\in\Z_2^\times,\qquad
c_\varphi=2^b c_{\varphi_0},\qquad v_2(c_\varphi)=b=b_\varphi.
\]
\end{proof}

\begin{remark}
The normalization follows from the $2$-adic unit property of $c_\pi$. This also follows from $E_{\mathrm{opt}}[|c_\pi|]\subseteq E_{\mathrm{opt}}(\Q)$ in \cite[Proposition 2.2]{CLN}: the odd-degree isogeny $\alpha_0$ identifies $E_{\mathrm{opt}}[2]$ with the irreducible module $E[2]$, ruling out $2\mid c_\pi$. When $4\nmid N$, the unit property follows from \cite[Theorem 1.2]{Cesnavicius}.
\end{remark}

The height calculations retain $c_\varphi$; Section~\ref{subsec:map-shift} applies the normalization to the orders of standard classes.

\subsection{From finite Selmer dimensions to low-rank comparison curves}
Assume \Heegner, \Surjectivity{L}, and \Exclusion, and choose $T\in\mathscr T_L$. We compute the index of $Q_T$ for the fixed parametrization $\varphi$, retaining its Manin factor $c_\varphi$.
Write $d=d_T$ and $\nu=|T|$, and label $E^d,E^{dD}$ as $A,B$ so that
\[
d_2(A)=0,\qquad d_2(B)=1.
\]
Residual irreducibility ensures that neither curve has rational $2$-torsion. The usual Kummer exact sequence
\begin{equation}\label{eq:finite-kummer}
0\longrightarrow C(\Q)/2C(\Q)\longrightarrow\Sel_2(C/\Q)
\longrightarrow\Sha(C/\Q)[2]\longrightarrow0
\end{equation}
and the alternating Cassels--Tate pairing for elliptic curves \cite[\S1, Corollary 7]{PS} give
\[
d_2(C)-s_2(C)\in2\Z_{\ge0}.
\]
Indeed, the divisible part of $\Sel_{2^\infty}(C/\Q)$ contributes $s_2(C)$ independent $2$-torsion directions, and the elementary divisors of the finite part occur in pairs under the Cassels--Tate pairing. Since $C(\Q)[2]=0$, the finite $2$-Selmer group is the $2$-torsion of the infinite Selmer group. The difference is therefore nonnegative and even; Appendix~\ref{subsec:finite-size} gives the group-theoretic argument and the finite-level identification. Thus $d_2(A)=0$ forces $s_2(A)=0$, while $d_2(B)=1$ and parity force $s_2(B)=1$.

Applying \LC, we obtain
\begin{equation}\label{eq:low-rank}
\begin{gathered}
L(A,1)\ne0,\qquad L(B,1)=0,\qquad L'(B,1)\ne0,\\
\rank A(\Q)=0,\qquad\rank B(\Q)=1,\qquad
\Sha(A/\Q),\Sha(B/\Q)\text{ are finite}.
\end{gathered}
\end{equation}
Apply \eqref{eq:finite-kummer} again. For $A$, the Selmer group is zero. For $B$, rank one and the absence of rational $2$-torsion imply that $B(\Q)/2B(\Q)$ already accounts for the entire one-dimensional Selmer group. Thus $\Sha[2]=0$ on both sides. Every nonzero element of $2$-power order has a multiple of order two, so
\begin{equation}\label{eq:sha-rational-vanish}
\Sha(A/\Q)[2^\infty]=\Sha(B/\Q)[2^\infty]=0.
\end{equation}
\subsection{Heights, periods, and the full Mordell--Weil lattice}\label{subsec:height}
\begin{center}
\captionof{table}{Normalization conventions in the leading-term comparison}\label{tab:normalizations}
\small
\begin{tabularx}{\textwidth}{@{}lX@{}}
\toprule
Object & Convention\\
\midrule
$H$ & Absolute canonical height associated with $2[O]$\\
$\widehat h_K$ & Relative height $[K:\Q]H=2H$\\
$\Reg_C$ & Determinant of $B_H$ on the full free Mordell--Weil lattice\\
$\Omega_C$ & $\int_{C(\R)}|\omega_C|$, including all real components\\
$\mathcal A_E$ & $\int_{E(\C)}|\omega_E\wedge\overline{\omega_E}|$, twice Euclidean area\\
$P_{\chi_T}$ & Full sum of character-weighted conjugates, without averaging\\
\bottomrule
\end{tabularx}
\end{center}
Let $h_x$ be the absolute logarithmic height of the $x$-coordinate and set
\begin{equation}\label{eq:height-convention}
H(P)=\lim_{j\to\infty}4^{-j}h_x([2^j]P),\qquad
B_H(P,Q)=\frac{H(P+Q)-H(P)-H(Q)}2.
\end{equation}
Then $B_H(P,P)=H(P)$. The function $H$ is the absolute canonical height associated with $2[O]$; if $\widehat h$ is associated with $[O]$, then $H=2\widehat h$. It is unchanged by extension of the field of definition or by an isomorphism of elliptic curves preserving the origin. The regulator $\Reg_C$ in \eqref{eq:BSD2-input} is the determinant of $B_H$ on the full lattice $C(\Q)/C(\Q)_{\tors}$, with value one in rank zero.
Equivalently, the bilinear pairing used here is
\[
B_H(P,Q)=\widehat h(P+Q)-\widehat h(P)-\widehat h(Q).
\]
Thus the rank-one regulator is $H(P)=2\widehat h(P)$ for a generator $P$ of the free Mordell--Weil group.

Define the total real period and the complex area by
\begin{equation}\label{eq:periods}
\Omega_C=\int_{C(\R)}|\omega_C|,\qquad
\mathcal A_E=\int_{E(\C)}|\omega_E\wedge\overline{\omega_E}|.
\end{equation}
The integral defining $\Omega_C$ runs over all real connected components, so their number is already included. If $E(\C)=\C/\Lambda$ and $\omega_E=dz$, then $\mathcal A_E$ is twice the Euclidean area of a fundamental parallelogram.

Artin formalism gives
\[
L(E/K,\chi_T,s)=L(E^d,s)L(E^{dD},s).
\]
The character $\chi_T$ has conductor $n$: the local characters of Section~\ref{sec:cm} are nontrivial on inertia at each $q\mid n$ and unramified at the remaining finite places. The equalities $(n,N)=(D,N)=1$ and the Heegner hypothesis verify the two local assumptions of \cite[Theorem 1.1]{CST}. In its notation, $\mu(N,D)=0$, and $D<-4$ gives $u=[\mathcal O_n^\times:\Z^\times]=1$. Applied to the full character trace in \eqref{eq:full-character}, that theorem gives
\[
L'(E/K,\chi_T,1)=
\frac{8\pi^2(f_E,f_E)_{\Gamma_0(N)}}{n\sqrt{|D|}\deg\varphi}
\widehat h_K(P_{\chi_T}),
\]
where $(f_E,f_E)_{\Gamma_0(N)}=\int_{\Gamma_0(N)\backslash\mathfrak H}|f_E(\tau)|^2\,du\,dv$ for $\tau=u+iv$. Substituting $\varphi^*\omega_E=2\pi i c_\varphi f_E(\tau)\,d\tau$ in the area integral gives
\[
\deg(\varphi)\mathcal A_E
=8\pi^2c_\varphi^2(f_E,f_E)_{\Gamma_0(N)}.
\]
The height $\widehat h_K$ in this convention is relative to $K$ and associated with $2[O]$, so $\widehat h_K=[K:\Q]H=2H$; see \cite[Chapter I, (4.3), (6.4)]{GZ} for these normalizations. For points defined over $H_T$, the same convention gives $\widehat h_{H_T}=[H_T:K]\widehat h_K$. Substitute these conversions and differentiate the Artin factorization at $s=1$. By \eqref{eq:low-rank}, the term containing $L(B,1)$ vanishes, yielding
\begin{equation}\label{eq:GZ}
L(A,1)L'(B,1)=\frac{2\mathcal A_E}{c_\varphi^2 n\sqrt{|D|}}H(P_{\chi_T}).
\end{equation}
Equations~\eqref{eq:low-rank} and \eqref{eq:GZ} imply that $Q_T$ is non-torsion. The rational rank decomposition for quadratic twists gives $\rank E^d(K)=1$. Define the lattice index of the distinguished point by
\begin{equation}\label{eq:mT}
m_T=[E^d(K)/E^d(K)_{\tors}:\Z\overline{Q_T}]\in\Z_{>0}.
\end{equation}
\begin{lemma}\label{lem:full-lattice}
Let $C/\Q$ be an elliptic curve and $K=\Q(\sqrt D)$ a quadratic field. Suppose $C(K)[2]=0$ and
$\{\rank C(\Q),\rank C^D(\Q)\}=\{0,1\}$, and let $B$ be the rank-one curve among $C,C^D$. The natural inclusion, composed with the twisting isomorphism when $B=C^D$, induces an isomorphism
\[
B(\Q)/B(\Q)_{\tors}\xrightarrow{\ \sim\ }C(K)/C(K)_{\tors}
\]
preserving the absolute canonical height $H$.
\end{lemma}
\begin{proof}
Let $c$ generate $\Gal(K/\Q)$. The rational decomposition
\[
C(K)\otimes\Q=(C(\Q)\otimes\Q)\oplus(C^D(\Q)\otimes\Q)
\]
shows that $C(K)/C(K)_{\tors}$ has rank one and that $c$ acts on it by a sign $s\in\{\pm1\}$ corresponding to $B$. Choose a generator represented by $P\in C(K)$ and put $t=cP-sP\in C(K)_{\tors}$. Then $ct=-st$.
The torsion group has odd order, so there is a unique torsion point $u$ with $2su=t$. It satisfies $cu=-su$, and hence
\[
c(P+u)-s(P+u)=t-2su=0.
\]
Thus the image of $B(\Q)$ contains a generator of the full free lattice of $C(K)$. The inclusion is injective modulo torsion, so it induces the claimed isomorphism. Absolute canonical heights are preserved by the twisting isomorphism and are unchanged by adding torsion.
\end{proof}

Apply Lemma~\ref{lem:full-lattice} with $C=E^d$. Its torsion hypothesis follows from $E^d[2]=E[2]$ and Lemma~\ref{lem:no-torsion}. The image of a generator of $B(\Q)/B(\Q)_{\tors}$ generates the full lattice of $E^d(K)$. Consequently,
\begin{equation}\label{eq:height-index}
H(Q_T)=m_T^2\Reg_B,\qquad
H(P_{\chi_T})=2^{2\nu}m_T^2\Reg_B.
\end{equation}
The second equality also uses $P_{\chi_T}=2^\nu\iota_d(Q_T)$ and the quadraticity of the canonical height.

\subsection{Cancellation of local factors}
We use the fundamental discriminant $D$ itself as the twist parameter. The differential multiplier depends on this representative; its comparison with the squarefree representative $D/4$ for even $D$ is given in Appendix~\ref{app:dyadic}. For $u=d,dD$, write the comparison of minimal differentials as
\[
\iota_u^*\omega_E=\sqrt u\,\eta_u\omega_{E^u},\qquad \eta_u\in\Q^\times.
\]
Define
\[
e_2(D)=v_2(\eta_D),\qquad
C_2(D)=v_2\!\left(\frac{c_2(E^D)}{c_2(E)}\right).
\]
For every support under consideration,
\begin{equation}\label{eq:dyadic-support}
v_2(\eta_d)=0,\qquad v_2(\eta_{dD})=e_2(D).
\end{equation}
There are two cases. If $D$ is even, each auxiliary prime satisfies $q\equiv7\pmod8$, so $d$ is a square in $\Q_2$ and the differential comparison applies directly. If $D$ is odd, the integers $d,D,dD$ are all congruent to one modulo four. Their squareclasses at $2$ are therefore $1$ or $5$, corresponding to trivial or unramified twists. Minimal N\'eron differentials remain minimal after unramified base change, and $\sqrt d,\sqrt D,\sqrt{dD}$ are units. All three differential multipliers thus have $2$-adic valuation zero. In particular $e_2(D)=0$, giving \eqref{eq:dyadic-support}.

Appendix~\ref{app:dyadic} establishes, in both the ordinary and supersingular cases, the local identity
\begin{equation}\label{eq:CminusE}
C_2(D)-e_2(D)=h_2(E,K).
\end{equation}

To compare real periods, define
\[
E(\C)^\pm=\{P:cP=\pm P\},\qquad
\Omega_E^\pm=\int_{E(\C)^\pm}|\omega_E|.
\]
The integrals again include all connected components of the respective eigensubgroups, and $\Omega_E^+=\Omega_E$. The period lattice of a real elliptic curve has one of the following forms, with $a,b>0$:
\[
\begin{array}{c|c|c|c|c}
\Lambda&\epsilon_\infty&\Omega_E^+&\Omega_E^-&\mathcal A_E\\ \hline
a\Z+ib\Z&1&2a&2b&2ab\\
a\Z+\frac{a+ib}{2}\Z&0&a&b&ab
\end{array}
\]
In the first row, the positive eigensubgroup consists of two circles of length $a$, and the negative eigensubgroup of two circles of length $b$. In the second row, the oblique period identifies the relevant pairs of lines, leaving one component in each eigensubgroup. Both rows give
\[
\Omega_E^+\Omega_E^-=2^{\epsilon_\infty}\mathcal A_E.
\]
If $u>0$, the map $\iota_u$ sends $E^u(\R)$ to $E(\C)^+$; if $u<0$, it sends it to $E(\C)^-$. Integrating the differential pullback formula over every component yields
\[
\Omega_{E^u}=\frac{\Omega_E^{\operatorname{sgn}(u)}}{\sqrt{|u|}\,|\eta_u|}.
\]
Since $d$ and $dD$ have opposite signs and $|d|=n$, multiplication gives
\begin{equation}\label{eq:twist-periods}
n\sqrt{|D|}\,\Omega_{E^d}\Omega_{E^{dD}}
=\frac{2^{\epsilon_\infty}\mathcal A_E}{|\eta_d\eta_{dD}|}.
\end{equation}
Combining \eqref{eq:GZ}, \eqref{eq:height-index}, and \eqref{eq:twist-periods}, we obtain
\begin{equation}\label{eq:JT}
\mathcal J_T:=\frac{L(A,1)L'(B,1)}{\Omega_A\Omega_B\Reg_B}
=\frac{2|\eta_d\eta_{dD}|}{2^{\epsilon_\infty}c_\varphi^2}\,2^{2\nu}m_T^2.
\end{equation}
Taking $2$-adic valuations gives
\begin{equation}\label{eq:height-side}
v_2(\mathcal J_T)=1-\epsilon_\infty-2v_2(c_\varphi)+2\nu+2v_2(m_T)+e_2(D).
\end{equation}

We now apply \BSD. By \eqref{eq:sha-rational-vanish} and the absence of rational $2$-torsion, neither $\Sha$ nor the rational torsion groups contribute powers of two. Only the Tamagawa factors remain:
\begin{center}
\captionof{table}{Local $2$-adic contributions of the terminal comparison curves}\label{tab:valuation}
\small
\begin{tabularx}{\textwidth}{@{}p{.2\textwidth}X p{.17\textwidth}@{}}
\toprule
Place & Local behavior of the two curves & Contribution to $v_2(\mathcal J_T)$\\ 
\midrule
Original bad $\ell\mid N$ & The permitted supports preserve the $2$-adic Tamagawa valuations of both curves & $2T_E$\\ 
New $q\in T$ & Both are ramified twists of a good curve, each with $c_q=\#\mathcal E(\F_q)[2]=2$ & $2$ for each $q$\\ 
Odd $r\mid D$ & One curve is good; the other is a ramified twist with $c_r=2^{t_r}$ & $t_r$\\ 
The prime $2$ & Additional Tamagawa correction relative to $2T_E$ & $C_2(D)$\\ 
\bottomrule
\end{tabularx}
\end{center}
The odd-prime entries follow from \cite[Theorem 4.1, the $I_0$ row of Table 4]{LocalTwist}. Over a local field of odd residue characteristic, write a good curve as $y^2=g(x)$ with separable reduction $\overline g$. A ramified quadratic twist has type $I_0^*$, and Tate's algorithm gives
\[
c_F(E^u)=1+\#\{\alpha\in k_F:\overline g(\alpha)=0\}
=\#\widetilde E(k_F)[2].
\]
The cited result includes residue characteristic three. At $q\in T$, residual Frobenius is a transposition and the fixed $2$-torsion space is one-dimensional, so each twist contributes $v_2(c_q)=1$. At odd $r\mid D$, the parameter $d$ is a unit and $E^d$ remains good. Twisting preserves the $2$-torsion representation, so the Tamagawa number of the other curve is $2^{t_r}$.

If $\Delta_E<0$ or $|T|$ is even, both $d$ and $D$ are local squares at every original bad prime. The two curves therefore contribute $2v_2(c_\ell(E))$ at each $\ell\mid N$, with total $2T_E$. The remaining case, $\Delta_E>0$ and odd $|T|$, is allowed only when $\intr(E,K)=0$; then $T_E=0$ and Lemma~\ref{lem:unramified-odd} preserve the odd Tamagawa numbers, giving the same total.

At $2$, if $D$ is odd, then $C_2(D)=0$. When $E$ is good, all relevant twists remain good. When $E$ is bad, $D$ is a square; either $d$ is a square as well, or the intrinsic-index-zero case just discussed applies. If $D$ is even, $E$ is good at $2$ and $d$ is a square, so the additional contribution is exactly $C_2(D)$. Hence
\begin{equation}\label{eq:BSD-side}
v_2(\mathcal J_T)=2T_E+2\nu+\sum_{r\mid D,\ r\text{ odd}}t_r+C_2(D).
\end{equation}

\begin{theorem}\label{thm:index}
Fix $L\ge2$ and assume \Heegner, \Surjectivity{L}, and \Exclusion. Let $T\in\mathscr T_L$, and put $n_T=\prod_{q\in T}q$ and $d_T=\prod_{q\in T}(-q)$. For a fixed nonconstant modular parametrization $\varphi:X_0(N)\to E$ over $\Q$ with $\varphi(\infty)=O$, let $Q_T\in E^{d_T}(K)$ be the point of Lemma~\ref{lem:genus-point}, and let
\[
m_T=[E^{d_T}(K)/E^{d_T}(K)_{\tors}:\Z\overline{Q_T}].
\]
Then
\begin{equation}\label{eq:exact-index}
v_2(m_T)=v_2(c_\varphi)+\intr(E,K).
\end{equation}
Consequently,
\[
\kappa^{\std}_{n_T,1}(\varphi)\ne0
\quad\Longleftrightarrow\quad
v_2(c_\varphi)+\intr(E,K)=0.
\]
For a $2$-indivisible parametrization, this is equivalent to $\intr(E,K)=0$.
\end{theorem}
\begin{proof}
Comparing \eqref{eq:height-side} and \eqref{eq:BSD-side} gives
\[
2v_2(m_T)=2v_2(c_\varphi)+2T_E+\epsilon_\infty-1
+\sum_{r\mid D,\ r\text{ odd}}t_r+C_2(D)-e_2(D).
\]
Substituting \eqref{eq:CminusE} and dividing by two gives
\[
\begin{aligned}
v_2(m_T)
&=v_2(c_\varphi)+T_E+
\frac{\epsilon_\infty-1+h_2(E,K)+
\sum_{r\mid D,\ r\text{ odd}}t_r}{2}\\
&=v_2(c_\varphi)+\intr(E,K).
\end{aligned}
\]
This proves \eqref{eq:exact-index}.

In $E^d(K)/E^d(K)_{\tors}$, write $[Q_T]=m_T[P]$, choosing the sign of the generator appropriately. The torsion group has odd order, so multiplication by $2$ is invertible on it. Hence $Q_T\in2E^d(K)$ exactly when $m_T$ is even. The Kummer map $E^d(K)/2E^d(K)\hookrightarrow H^1(K,E^d[2])$ is injective, and Proposition~\ref{prop:mod2-comparison} identifies the Kummer image of $Q_T$ with $\kappa^{\std}_{n_T,1}(\varphi)$. This proves the final equivalence.
\end{proof}

The low-corank converse and the Gross--Zagier formula establish non-torsion; \BSD\ determines the index needed to detect $Q_T$ modulo $2$ in Proposition~\ref{prop:mod2-comparison}. For a fixed nonempty support, the conditions $2^M\mid q+1,a_q(E)$ bound the available coefficient depth, whereas non-torsion alone does not control the point's $2$-divisibility. The empty support has no such depth bound, and \LC\ alone gives strong nonvanishing by Proposition~\ref{prop:empty-support}.

\subsection{Orders of the standard classes}\label{subsec:map-shift}
\begin{proof}[Proof of Theorem~\ref{thm:main-finite}]
Proposition~\ref{prop:manin-normalization} gives $v_2(c_{\varphi_0})=0$ and $v_2(c_\varphi)=b_\varphi\ge0$. Theorem~\ref{thm:index} therefore proves \eqref{eq:residual-criterion} and the assertion for arbitrary $\varphi$, since $\intr(E,K)\ge0$.

Suppose now that $\intr(E,K)=0$. Then $\kappa^{\std}_{n_T,1}(\varphi_0)\ne0$. Coefficient compatibility \eqref{eq:coefficient-compatibility} and injectivity of the coefficient inclusion give
\[
2^{M-1}\kappa^{\std}_{n_T,M}(\varphi_0)
=i_{1,M,*}\kappa^{\std}_{n_T,1}(\varphi_0)\ne0.
\]
The cohomology group is killed by $2^M$, so this class has order exactly $2^M$.
\end{proof}

\begin{proposition}\label{prop:map-shift}
For a terminal support satisfying the hypotheses of Theorem~\ref{thm:main-finite} with $\intr(E,K)=0$, normalize $\varphi=[2^b]\circ\varphi_0$ as in Proposition~\ref{prop:manin-normalization}. Then \eqref{eq:main-order} holds, and for $b<M\le L$,
\begin{equation}\label{eq:shift-inclusion}
\kappa^{\std}_{n_T,M}(\varphi)
=i_{M-b,M,*}\kappa^{\std}_{n_T,M-b}(\varphi_0).
\end{equation}
\end{proposition}
\begin{proof}
Linearity of CM points, derivatives, and Kummer maps gives
$\kappa^{\std}_{n_T,M}(\varphi)=2^b\kappa^{\std}_{n_T,M}(\varphi_0)$.
The full-order assertion of Theorem~\ref{thm:main-finite} gives the claimed orders and vanishing. Coefficient compatibility gives \eqref{eq:shift-inclusion}. Together with Proposition~\ref{prop:manin-normalization}, this also proves Corollary~\ref{cor:main-shift}.
\end{proof}

\section{Strong nonvanishing}\label{sec:strong}
\subsection{Coefficient depth and additional precision}
Write
\[
\mathcal H=H^1(K,E[2^\infty]).
\]
The coefficient sequence
\[
0\longrightarrow E[2^M]\longrightarrow E[2^\infty]
\xrightarrow{\,2^M\,}E[2^\infty]\longrightarrow0
\]
and $E(K)[2^\infty]=0$ identify $H^1(K,E[2^M])$ with $\mathcal H[2^M]$. These identifications respect coefficient inclusions, placing the standard classes at all levels in a common group. Let $k_0$ be the fixed additional precision used to construct classes in \cite[printed pp.~10--11]{KolyvaginPreprint}. It controls $2$-power torsion in the compositum of the ring class fields. That compositum is abelian over $K$, so the argument of Lemma~\ref{lem:no-torsion} makes its $2$-power torsion trivial as well. Thus $k_0=0$, and the images $\tau'_{\lambda,n}$ used in the original formulation identify with the classes themselves. Appendix~\ref{app:kolyvagin-structure} determines the uniform losses in signed local coordinates and Chebotarev detection. For $r\ge0$, define
\begin{equation}\label{eq:VMk}
\mathcal V^r_{M,k}
=\left\langle\kappa^{\std}_{n,M}(\varphi):
n\in\mathcal N_{M+k}(E,K),\ \nu(n)=r\right\rangle
\subseteq\mathcal H[2^M].
\end{equation}

\begin{definition}\label{def:strong}
The standard system satisfies \emph{strong nonvanishing} if there is a fixed $r\ge0$ such that
\begin{equation}\label{eq:strong}
\forall k\ge k_0\quad\exists M\ge1:\quad\mathcal V^r_{M,k}\ne0.
\end{equation}
The least such $r$ is denoted by $\fstr$.
\end{definition}
This is the formulation of \cite[printed p.~11, equation (2)]{KolyvaginPreprint}. The support size $r$ is fixed before $k$ varies. The coefficient depth $M$ and conductor $n$ may depend on $k$, which imposes the stronger congruences $2^{M+k}\mid q+1,a_q(E)$.

Kolyvagin's prime convention requires $\min\{v_2(q+1),v_2(a_q(E))\}>1$ \cite[printed p.~3]{KolyvaginPreprint}. Our set $\mathcal N_1$ is slightly larger, but for $k\ge1$ the condition $n\in\mathcal N_{M+k}$ already imposes this inequality. Since a witness at a larger $k$ also works at a smaller $k$, the two conventions give the same admissible fixed support sizes and the same minimum $\fstr$.

The choice of cyclic generators does not affect this minimum. If $\sigma_q$ is replaced by $\sigma_q^u$, with $(u,q+1)=1$, reindexing the derivative sum gives
\[
D_{q+1}(\sigma_q^u)\equiv u^{-1}D_{q+1}(\sigma_q)
\pmod{2^M}
\qquad\bigl(2^M\mid q+1\bigr).
\]
Here $u$ is odd, so this multiplies the descended class by a unit. Independence of the trace representatives was established in Section~\ref{subsec:mod2}. These choices thus preserve both nonvanishing and the strong-nonvanishing invariant.

\begin{proposition}\label{prop:strong-scaling}
Suppose $\varphi=[2^b]\circ\psi$ with $b\ge0$. The standard systems for $\varphi$ and $\psi$ satisfy strong nonvanishing simultaneously. When they do,
\[
\fstr(\varphi)=\fstr(\psi).
\]
\end{proposition}
\begin{proof}
For each fixed $r$, compare the quantified condition in Definition~\ref{def:strong}. If it holds for $\varphi$, then for the same $M,k,n$,
$0\ne\kappa^{\std}_{n,M}(\varphi)=2^b\kappa^{\std}_{n,M}(\psi)$,
so it also holds for $\psi$.

Conversely, suppose $\psi$ satisfies strong nonvanishing at support size $r$. Given $k\ge k_0$, apply that condition with additional precision $k+b$ to obtain
\[
n\in\mathcal N_{M+k+b},\qquad \nu(n)=r,\qquad
\kappa^{\std}_{n,M}(\psi)\ne0.
\]
The class at coefficient depth $M+b$ is also defined. Coefficient compatibility and linearity in the parametrization give
\[
\kappa^{\std}_{n,M+b}(\varphi)
=2^b\kappa^{\std}_{n,M+b}(\psi)
=i_{M,M+b,*}\kappa^{\std}_{n,M}(\psi)\ne0,
\]
where injectivity uses $E(K)[2^\infty]=0$. Since $n\in\mathcal N_{(M+b)+k}$, this is exactly the condition required for $\varphi$. The two systems admit the same fixed support sizes, and hence the same minimum.
\end{proof}

\noindent
The same coefficient-lifting argument explains how strong nonvanishing accommodates a fixed loss of precision. Given $t\ge0$ and a target additional precision $k$, choose a nonzero class at depth $M$ with $n\in\mathcal N_{M+k+t}$ and the prescribed support size. Then
\[
2^t\kappa^{\std}_{n,M+t}
=i_{M,M+t,*}\kappa^{\std}_{n,M}\ne0,
\qquad n\in\mathcal N_{(M+t)+k}.
\]
Thus the classes witnessing strong nonvanishing can have arbitrarily large order at the same additional precision and support size.

\subsection{Bounded support at arbitrarily large depth}
\begin{proposition}\label{prop:strong-transfer}
Suppose there are fixed integers $R,B\ge0$ and an unbounded set $\mathscr L$ of depths such that, for every sufficiently large $L\in\mathscr L$, there is $n_L\in\mathcal N_L(E,K)$ with $\nu(n_L)\le R$ and
\[
\ord\kappa^{\std}_{n_L,M}(\varphi)=2^{M-b_L}
\quad(b_L<M\le L),\qquad 0\le b_L\le B.
\]
Then the standard system satisfies strong nonvanishing, and $\fstr\le R$.
\end{proposition}
\begin{proof}
Fix $M_0=B+1$. For sufficiently large $L\in\mathscr L$, we have $b_L<M_0\le L$, so the order formula gives $\kappa^{\std}_{n_L,M_0}\ne0$. Partition these depths according to $\nu(n_L)\in\{0,\ldots,R\}$. If every part were bounded, their finite union would be bounded, a contradiction. Thus some $r\le R$ occurs on an unbounded subset $\mathscr L'$. Given $k\ge k_0$, choose $L\in\mathscr L'$ with $L\ge M_0+k$. Admissibility becomes stronger as the depth increases, so $\mathcal N_L\subseteq\mathcal N_{M_0+k}$. Therefore
\[
n_L\in\mathcal N_{M_0+k}(E,K),\quad\nu(n_L)=r,\quad
\kappa^{\std}_{n_L,M_0}\ne0.
\]
This proves \eqref{eq:strong}.
\end{proof}

Intrinsic index zero and odd Manin factor give $B=0$, $M_0=1$. A fixed factor $2^b$ in the parametrization gives $B=b$, $M_0=b+1$ by Proposition~\ref{prop:map-shift}.

\label{subsec:depth-hypothesis}
The unbounded-image hypothesis admits a finite test.

\begin{lemma}\label{lem:mod-eight}
A closed subgroup $H\subseteq\GL_2(\Z_2)$ equals $\GL_2(\Z_2)$ if and only if its reduction modulo $8$ is surjective. Consequently, full $2$-adic image, full mod-$8$ image, and full image at unbounded finite depths are equivalent.
\end{lemma}
\begin{proof}
Only sufficiency needs proof. Suppose $m\ge3$ and $H$ is surjective modulo $2^m$. For any $B\in M_2(\F_2)$, choose an integral lift. Surjectivity allows us to take
\[
h=I+2^{m-1}B+2^mC\in H,\qquad C\in M_2(\Z_2).
\]
Since $2m-2\ge m+1$,
\[
h^2\equiv I+2^mB\pmod{2^{m+1}}.
\]
These squares give the entire kernel of $\GL_2(\Z/2^{m+1}\Z)\to\GL_2(\Z/2^m\Z)$. Together with the inductive surjectivity assumption, they give surjectivity at the next level. Induction from $m=3$ and closedness of $H$ yield $H=\GL_2(\Z_2)$. Surjectivity at unbounded finite depths implies surjectivity modulo $8$ by projection, proving the remaining equivalence.
\end{proof}

Under \Surjectivity{3}, Lemmas~\ref{lem:mod-eight} and \ref{lem:kernel-stable} identify the restriction kernel with $\F_2\xi_E$ at every depth $L\ge2$. Thus \Exclusion\ is independent of $L$. Under \Heegner\ and \Exclusion, Proposition~\ref{prop:general-terminal} supplies terminal supports with a uniform size bound at each such depth. If $\intr(E,K)=0$, Proposition~\ref{prop:terminal} gives the sharper fixed size $R$.

\begin{proof}[Proof of Theorem~\ref{thm:main-bound}]
Fix $M_0=b_\varphi+1$. Given $k\ge k_0$, choose $L\ge\max\{2,M_0+k\}$. By Lemma~\ref{lem:mod-eight}, \Surjectivity{L} holds, so Theorem~\ref{thm:main-descent} supplies $T_L\in\mathscr T_L$ of size $R$. Put $n_L=n_{T_L}$. Corollary~\ref{cor:main-shift} gives
\[
n_L\in\mathcal N_L\subseteq\mathcal N_{M_0+k},\qquad
\nu(n_L)=R,\qquad
\ord\kappa^{\std}_{n_L,M_0}(\varphi)=2.
\]
This proves strong nonvanishing with fixed support size $R$ and fixed coefficient depth $M_0$, hence $\fstr\le R$. Proposition~\ref{prop:strong-scaling} identifies the minimum support sizes before and after normalization.
\end{proof}

Proposition~\ref{prop:strong-transfer} treats the more general case of uniformly bounded support sizes.

\begin{proposition}\label{prop:empty-support}
Assume \Heegner, residual $S_3$ image, and
$\{d_2(E),d_2(E^D)\}=\{0,1\}$. Then for every fixed nonconstant parametrization $\varphi$ the basic trace $y_K$ is non-torsion and $\fstr(\varphi)=0$.
\end{proposition}
\begin{proof}
The finite/infinite Selmer comparison and Cassels--Tate alternation used in Section~\ref{sec:index} identify the two coranks with $0$ and $1$. By \LC, the analytic ranks are $0$ and $1$. Hence Artin factorization has a simple zero at $s=1$, and the Gross--Zagier formula for the trivial ring class character gives $H(y_K)>0$. In the finitely generated group $E(K)$, a non-torsion point cannot lie in $2^M E(K)$ for all $M$. Fix a level $M$ at which its Kummer class is nonzero. Since $1\in\mathcal N_{M+k}$ for every $k$, this single class proves strong nonvanishing with empty support.
\end{proof}

\section{Strong nonvanishing and the structure of Selmer groups}\label{sec:selmer-bound}
We describe the Selmer groups associated with a strongly nonvanishing standard system. The descent and group-theoretic proofs are collected in Appendix~\ref{app:kolyvagin-structure}.

\subsection{The corank theorem}
Put
\[
S_K=\Sel_{2^\infty}(E/K),\qquad
s_K=\corank_{\Z_2}S_K,\qquad r_K=\rank E(K).
\]
Let $s^\pm$ be the dimensions of the two conjugation eigenspaces of the rational Tate space
$\mathscr V_K=(\varprojlim_M S_K[2^M])\otimes_{\Z_2}\Q_2$.
Restriction and corestriction identify these dimensions with the rational Selmer coranks:
\begin{equation}\label{eq:rational-signs}
s^+=s_2(E),\qquad s^-=s_2(E^D),\qquad s_K=s^++s^-.
\end{equation}

\begin{theorem}\label{thm:Kolyvagin}
Assume \Heegner, \Surjectivity{3}, and $T_E=0$. Suppose the standard system satisfies strong nonvanishing. If $f=\fstr$ and $\epsilon_f=(-1)^f(-w(E))$, then
\begin{equation}\label{eq:Kolyvagin-signed}
s^{\epsilon_f}=f+1,\qquad s^{-\epsilon_f}\le f.
\end{equation}
Consequently,
\begin{equation}\label{eq:Kolyvagin-corank}
f=\max\{s_2(E),s_2(E^D)\}-1,\qquad f+1\le s_K\le2f+1.
\end{equation}
\end{theorem}

\paragraph{Proof of the corank formula.}
Coefficient compatibility turns strong nonvanishing into classes of
unbounded order at each fixed additional precision. The uniform
character-detection lemma
(Lemma~\ref{lem:deep-character-detection}) and the adjacent-conductor
identity (Lemma~\ref{lem:deep-derivative-local}) then give
$s^{-\epsilon_f}\le f$ and $s^{\epsilon_f}\le f+1$ by global duality.
For the reverse inequality, at sufficiently large additional precision,
minimality of $f$ makes the support-$f$ classes usual Selmer classes with
zero localization at their conductor primes. Repeated prime replacement
(Lemma~\ref{lem:prime-replacement}) produces $f+1$ such classes with a
triangular matrix of detecting coordinates and a loss bounded solely
in terms of $f$. Their unbounded orders force
$s^{\epsilon_f}\ge f+1$. The detailed local, detection, and independence
arguments are in Appendix~\ref{app:kolyvagin-structure}.

Theorem~\ref{thm:Kolyvagin} gives the support--corank identification \eqref{eq:main-corank}. Under the hypotheses of Theorem~\ref{thm:main-bound}, it combines with $f\le R$ to give
\begin{equation}\label{eq:combined-bound}
\rank E(K)\le s_K\le2f+1\le2R+1=a_0+a_1.
\end{equation}
For comparison, the last numerical bound also follows directly from
$s_K=s_2(E)+s_2(E^D)\le d_2(E)+d_2(E^D)=a_0+a_1$.
The additional information in \eqref{eq:Kolyvagin-signed} is the identification of the principal corank with $f+1$.

\subsection{The Selmer group structure under intrinsic index zero}
Combining the corank identification with the standard abstract structure supplied by Cassels--Tate alternation (Proposition~\ref{prop:usual-structure}) gives the following description. The corank parameters are expressed in terms of $f$.

\begin{theorem}\label{thm:full-selmer-structure}
Assume the hypotheses of Theorem~\ref{thm:main-bound}. Let $f=\fstr\le R$, and let $C_{\rm pr}=E$ if $\epsilon_f=+1$ and $C_{\rm pr}=E^D$ if $\epsilon_f=-1$; write $C_{\rm op}$ for the other curve. There are an integer $h$ with $0\le h\le\lfloor f/2\rfloor$ and finite abelian $2$-groups $A_{\rm pr},A_{\rm op},A_K$ such that
\begin{align}
\Sel_{2^\infty}(C_{\rm pr}/\Q)
&\simeq(\Q_2/\Z_2)^{f+1}\oplus A_{\rm pr}^{\oplus2},\label{eq:principal-full-structure}\\
\Sel_{2^\infty}(C_{\rm op}/\Q)
&\simeq(\Q_2/\Z_2)^{f-2h}\oplus A_{\rm op}^{\oplus2},\label{eq:opposite-full-structure}\\
\Sel_{2^\infty}(E/K)
&\simeq(\Q_2/\Z_2)^{2f+1-2h}\oplus A_K^{\oplus2}.\label{eq:quadratic-full-structure}
\end{align}
For each of these curve--field pairs $(C,F)$, let $s_F$ be the displayed corank, $r_F=\rank C(F)$, and $A_F$ the corresponding finite group. Then, for every $m\ge1$,
\begin{align}
\Sha(C/F)[2^\infty]
&\simeq(\Q_2/\Z_2)^{s_F-r_F}\oplus A_F^{\oplus2},\label{eq:three-sha-structure}\\
\Sel_{2^m}(C/F)
&\simeq(\Z/2^m\Z)^{s_F}\oplus A_F[2^m]^{\oplus2}.\label{eq:three-finite-structure}
\end{align}
All isomorphisms are noncanonical. In particular $\Sha(C/F)[2^\infty]$ is finite if and only if $r_F=s_F$.

If $a_{\rm pr}=d_2(C_{\rm pr})$, $a_{\rm op}=d_2(C_{\rm op})$, and
$t_\bullet=\dim_{\F_2}A_\bullet[2]$, then
\begin{equation}\label{eq:rational-finite-factor-counts}
t_{\rm pr}=\frac{a_{\rm pr}-f-1}{2},\qquad
t_{\rm op}=\frac{a_{\rm op}-f+2h}{2},\qquad
t_{\rm pr}+t_{\rm op}=R-f+h.
\end{equation}
\end{theorem}

The proof is in Appendix~\ref{app:three-group-proof}. These are the possible abstract structures: $h$ and the finite groups $A_{\rm pr},A_{\rm op},A_K$ are not computed by the theorem. Their elementary divisors can be recovered from the higher Selmer size profiles when those profiles are known, as explained in Appendix~\ref{subsec:finite-size}.

\subsection{Terminal twists over the quadratic field}
For a terminal auxiliary curve $C=E^{d_T}$, the low-rank comparison also determines the form of the groups over $K$.

\begin{proposition}\label{prop:terminal-sha}
Suppose $\{d_2(C),d_2(C^D)\}=\{0,1\}$. Then
\[
\rank C(K)=1,\qquad
2\Sha(C/K)[2^\infty]=0.
\]
In particular, the $2$-primary Tate--Shafarevich group is finite. If $d_K(C)=\dim_{\F_2}\Sel_2(C/K)$, then
\[
\Sha(C/K)[2^\infty]\simeq(\Z/2)^{d_K(C)-1},
\qquad
\Sel_{2^m}(C/K)\simeq\Z/2^m\oplus(\Z/2)^{d_K(C)-1},
\]
where $d_K(C)-1$ is a nonnegative even integer.
\end{proposition}

The proof is in Appendix~\ref{app:terminal-sha-proof}. Thus one further $2$-descent over $K$ determines the $2$-primary Tate--Shafarevich group of a terminal auxiliary curve, and this group vanishes exactly when $d_K(C)=1$.

\section{A finite-depth example with one auxiliary prime}\label{sec:example}
Consider the minimal equation
\begin{equation}\label{eq:example-curve}
E:\quad y^2+y=x^3+x^2-2x,\qquad D=-7,\qquad K=\Q(\sqrt{-7}).
\end{equation}
Its invariants are $c_4=112$, $c_6=-856$, and $\Delta_E=389$. Since the discriminant is squarefree, the equation is globally minimal, the only bad prime is $389$, and its reduction type is $I_1$. Thus $N=389$ and $c_{389}(E)=1$. The fundamental discriminant $D=-7$ is prime to $N$, and
$\left(\frac{-7}{389}\right)=1$, so \Heegner\ holds.

\paragraph{Residual image and intrinsic index.}
The nonzero $2$-torsion points have $x$-coordinates satisfying
\[
f(x)=4x^3+4x^2-8x+1=0,\qquad \disc(f)=16\cdot389.
\]
This cubic has no root modulo $7$, and hence is irreducible over $\Q$. Its discriminant is nonsquare, so the residual image is $S_3$. Thus $\mathcal E(\F_7)[2]=0$, giving $t_7=0$. Since $D\equiv1\pmod8$, the prime $2$ splits in $K$ and $h_2(E,K)=0$. Finally $\Delta_E>0$ gives $\epsilon_\infty=1$. Consequently
\[
T_E=0,\qquad
\delta_K=\frac{1-1+0+0}{2}=0,\qquad \intr(E,K)=0.
\]
The real local condition excludes $\xi_E$ from both initial Selmer groups.

\paragraph{Full image modulo four.}
Residual surjectivity and $\Delta_E>0$ give \Surjectivity{2} by the calculation in Section~\ref{subsec:xi}. Together with the splitting and real exclusion above, this verifies \Heegner, \Surjectivity{2}, and \Exclusion.

\paragraph{The prime family and the terminal pair.}
Take $L=2$ and choose the lift in Definition~\ref{def:PL} acting by
$\zeta_8\mapsto\zeta_8^7$. For $q=31$ the exact finite-field computations give
\[
\#\mathcal E(\F_{31})=28,\quad a_{31}(E)=4,\quad
31\equiv7\pmod8,\quad
\left(\frac{7}{31}\right)=1,\quad
\left(\frac{-7}{31}\right)=-1.
\]
The cubic $f$ has precisely one root modulo $31$, so residual Frobenius is a transposition. Its mod-$4$ characteristic polynomial is $X^2-1$; the residual nonscalar condition supplies a cyclic vector, so this matrix is conjugate to $A_2$ of Section~\ref{subsec:prime-family}. Also $\Q(\sqrt{389})\subset\Q(E[2])$, and
$\left(\frac{389}{31}\right)=-1$ agrees with the residual sign. These facts, together with the chosen cyclotomic lift and $\left(\frac{7}{31}\right)=1$, verify $31\in\mathcal P_2$. In particular $n=31$ is admissible at depth two, and $-31$ is a square in $\Q_7$ since $\left(\frac{-31}{7}\right)=1$.

The following full $2$-Selmer groups were computed unconditionally with Magma V2.29-10 using \texttt{TwoSelmerGroup(C : Bound := -1)} \cite[Selmer groups]{MagmaHandbook}.
\begin{center}
\captionof{table}{The $2$-Selmer groups of the four twists}\label{tab:example-selmer}
\small
\begin{tabular}{@{}rllc@{}}
\toprule
$d$ & Minimal coefficients $[a_1,a_2,a_3,a_4,a_6]$ & $\Sel_2(E^d/\Q)$ & $d_2(E^d)$\\
\midrule
$1$   & $[0,1,1,-2,0]$ & $(\Z/2)^2$ & $2$\\
$-7$  & $[0,-1,1,-114,-302]$ & $\Z/2$ & $1$\\
$-31$ & $[0,-1,1,-2242,-28768]$ & $0$ & $0$\\
$217$ & $[0,1,1,-109874,10087074]$ & $\Z/2$ & $1$\\
\bottomrule
\end{tabular}
\end{center}
Hence $(a_0,a_1)=(2,1)$, $R=1$, and $T=\{31\}$ is a terminal support of the minimum possible size. Here $d_T=-31$ and $d_TD=217$. The odd support exchanges the real local conditions, exactly as in Proposition~\ref{prop:actual-pair}.

The original curve has rank $2$ over $\Q$ \cite[Example H133E28]{MagmaHandbook}. On the displayed model of $E^{-7}$, the point $(14,25)$ is rational, and direct counting gives
\[
\#\mathcal E^{(-7)}(\F_3)=2,\qquad \#\mathcal E^{(-7)}(\F_5)=3.
\]
Both primes are of good reduction, so $E^{-7}(\Q)_{\tors}=0$. The point has infinite order, and the computed value $d_2(E^{-7})=1$ therefore gives $\rank E^{-7}(\Q)=1$. The quadratic rank decomposition yields
\[
\rank E(K)=\rank E(\Q)+\rank E^{-7}(\Q)=3.
\]

For reproducibility, the essential descent calculation is
\begin{verbatim}
SetSeed(1);
E := EllipticCurve([Rationals()|0,1,1,-2,0]);
for d in [1,-7,-31,217] do
    C := MinimalModel(QuadraticTwist(E,d));
    S := TwoSelmerGroup(C : Bound := -1);
    print d, aInvariants(C), Invariants(S);
end for;
\end{verbatim}
The invariant lists are respectively \texttt{[2,2]}, \texttt{[2]}, \texttt{[]}, and \texttt{[2]}. With \texttt{Bound := -1}, the class-group computations are unconditional when performed in a fresh Magma session without precomputed conditional class or unit groups, as described in \cite{MagmaHandbook}.

For the terminal pair $E^{-31},E^{217}$, any parametrization $\varphi_0$ with odd Manin factor gives
\[
\ord\kappa^{\std}_{31,1}(\varphi_0)=2,\qquad
\ord\kappa^{\std}_{31,2}(\varphi_0)=4.
\]
An arbitrary fixed parametrization gives the shift of Corollary~\ref{cor:main-shift}. These are the available coefficient depths for $n=31$, since $a_{31}(E)=4$.

\appendix

\section{Descent coordinates for the kernel class}\label{app:squareclass}
Assume \Surjectivity{2} throughout this appendix. We compute the squareclass representing $\xi_E$ in the cubic-algebra model for $2$-descent \cite{Schaefer}, giving coordinates for the local exclusion condition in Proposition~\ref{prop:odd-exclusion}.

Write $E:y^2=f_0(x)=x^3+Ax+B$, and put $D_f=\disc(f_0)$ and $L_f=\Q[\theta]/(f_0(\theta))$. Under the residual $S_3$ hypothesis, $f_0$ is irreducible and separable. Thus $L_f$ is a cubic root field, and the $2$-division field is its Galois closure of degree $6$. The discriminant of this short Weierstrass model is $16D_f$. It differs from the minimal discriminant by a rational twelfth power, so $D_f$ and $\Delta_E$ define the same rational squareclass.

Let $e_1,e_2,e_3$ be the roots, and let $P_i=(e_i,0)$ be the three nonzero $2$-torsion points. Set $\overline L_f=L_f\otimes\overline\Q\simeq\overline\Q^3$, with Galois action permuting the factors according to the roots. The Weil pairing gives an equivariant embedding
\[
\iota:V\longrightarrow\mu_2(\overline L_f),\qquad
P\longmapsto(e_2(P,P_1),e_2(P,P_2),e_2(P,P_3)).
\]
Its image consists precisely of the sign vectors whose product is $1$. Hence there is a split exact sequence
\begin{equation}\label{eq:Weil-norm-sequence}
1\longrightarrow V\xrightarrow{\iota}\operatorname{Res}_{L_f/\Q}\mu_2
\xrightarrow{N}\mu_2\longrightarrow1.
\end{equation}
The splitting is $\epsilon\mapsto(\epsilon,\epsilon,\epsilon)$, since $\epsilon^3=\epsilon$. In particular, the induced map on $H^1$ at the left is injective. Shapiro's lemma and Kummer theory yield
\begin{equation}\label{eq:descent-algebra}
H^1(\Q,V)\simeq
\ker\left(N:L_f^\times/L_f^{\times2}\longrightarrow\Q^\times/\Q^{\times2}\right).
\end{equation}
Compatibility with the norm map can be checked directly on cocycles. If $\beta^2=\alpha$, then
\[
N(g\beta/\beta)=g(N\beta)/(N\beta),\qquad (N\beta)^2=N\alpha.
\]
Thus $[\alpha]$ lies in the kernel in \eqref{eq:descent-algebra} if and only if $N\alpha\in\Q^{\times2}$.

\begin{proposition}\label{prop:xi-squareclass}
Under the identification \eqref{eq:descent-algebra},
\begin{equation}\label{eq:xi-squareclass}
\xi_E=[-D_f f_0'(\theta)].
\end{equation}
\end{proposition}
\begin{proof}
First, $N(f_0'(\theta))=-D_f$, whence $N(-D_ff_0'(\theta))=D_f^4$. The displayed squareclass therefore belongs to the norm kernel.

Choose $Q_i\in E[4]$ with $2Q_i=P_i$. The addition law gives
\[
r_i:=x(Q_i)-e_i,\qquad r_i^2=f_0'(e_i).
\]
Indeed, after translating by $u=x-e_i$, the equation is $y^2=u(u^2+\alpha_i u+\beta_i')$, where $\beta_i'=f_0'(e_i)\ne0$. The line joining $(u,y)$ to $(0,0)$ has slope $y/u$, so the sum has first coordinate
\[
\left(\frac yu\right)^2-\alpha_i-u=\frac{\beta_i'}u.
\]
For $Q_i$, the identity $Q_i+P_i=-Q_i$ makes the first coordinates equal, giving $r_i^2=\beta_i'$.

For any $R\in V$, the point $Q_i+R$ is another half of $P_i$. The preceding square identity shows that the ratio of their $u$-coordinates is $\pm1$. This ratio is $1$ precisely when $Q_i+R=\pm Q_i$, or equivalently $R\in\{O,P_i\}$. The Weil character $R\mapsto e_2(P_i,R)$ is the nontrivial character with this kernel. Therefore
\[
\frac{x(Q_i+R)-e_i}{x(Q_i)-e_i}=e_2(P_i,R).
\]
If the mod-$4$ image of $g$ is $1+2B$, then $gQ_i-Q_i=BP_i$, and hence $g(r_i)/r_i=(-1)^{\langle P_i,BP_i\rangle}$, where $\langle\ ,\ \rangle$ denotes the additive Weil pairing.

Put
\[
\Delta_0=(e_1-e_2)(e_1-e_3)(e_2-e_3),\qquad\Delta_0^2=D_f.
\]
Since $\Delta_0$ belongs to the $2$-division field, it is fixed by these $g$. The $4$-Weil pairing gives $i\in\Q(E[4])$, and
\[
\frac{g(i)}i=(-1)^{\tr B},\qquad
\det(1+2B)=1+2\tr B\pmod4.
\]
It follows that $\beta_i=i\Delta_0r_i$ is a square root of $-D_ff_0'(e_i)$ and satisfies
\begin{equation}\label{eq:sign-f}
\frac{g\beta_i}{\beta_i}=(-1)^{\tr B+\langle P_i,BP_i\rangle}
=(-1)^{\langle P_i,f(B)\rangle}.
\end{equation}
To verify the last equality, write $B=\left(\begin{smallmatrix}a&b\\c&d\end{smallmatrix}\right)$ and identify the three nonzero points with $(1,0),(0,1),(1,1)$ in the chosen symplectic basis. The two exponents have the following common values:
\[
\begin{array}{c|ccc}
P&(1,0)&(0,1)&(1,1)\\ \hline
\tr B+\langle P,BP\rangle&a+d+c&a+d+b&b+c\\
\langle P,f(B)\rangle&a+d+c&a+d+b&b+c
\end{array}
\]
Here we use $f(B)=(a+d+b,a+d+c)$ from \eqref{eq:f}.

All $\beta_i$ belong to the $4$-division field, so each coordinate becomes a square over that field and the cohomology class restricts to zero there. Equation~\eqref{eq:sign-f} identifies the restriction of its cocycle to the mod-$4$ kernel with $f$. For example, $f(E_{12})=(1,0)\ne0$, whereas every coboundary vanishes on this trivially acting kernel. The class is therefore nonzero. By the one-dimensional kernel calculation \eqref{eq:H1Gamma}, it is exactly $\xi_E$.
\end{proof}

For the twist $f_d(x)=d^3f_0(x/d)$, the root fields are identified by $\theta_d\mapsto d\theta$, and
\[
\disc(f_d)=d^6D_f,\qquad f_d'(d\theta)=d^2f_0'(\theta).
\]
Thus the representative in \eqref{eq:xi-squareclass} is multiplied by the square $d^8$. It defines the same kernel class under the natural identification of $2$-torsion.

These coordinates also give a finite local test. At an odd prime $\ell$ of good reduction, the $4$-division field is unramified. Since $\xi_E$ is inflated from its Galois group, its localization is unramified and lies in the usual Kummer subgroup. If $\Delta_E<0$, complex conjugation acts on $V$ as a transposition, and
\[
H^1(\R,V)=\ker(1+c)/\im(c-1)=0,
\]
because the two subspaces are the same line. Consequently, when $\Delta_E<0$, testing $\xi_E\in\Sel_2(E/\Q)$ requires only checking membership in the local Kummer images at primes dividing $2N$. For a twist, one also includes its new bad primes. At an odd ramified prime, Proposition~\ref{prop:odd-exclusion} gives a direct test using Frobenius congruences.

\section{Dyadic corrections and a local test for intrinsic index zero}\label{app:dyadic}
We verify \eqref{eq:CminusE} and make the condition $\intr(E,K)=0$ explicit. Write $\Delta=\Delta_E$ for the minimal discriminant, and retain the notation of Section~\ref{sec:index}:
\[
C_2(D)=v_2(c_2(E^D)/c_2(E)),\qquad e_2(D)=v_2(\eta_D).
\]
If $D$ is odd, then $D\equiv1$ or $5\pmod8$. In the first case, $2$ splits and the twist is $\Q_2$-isomorphic to $E$. In the second, the local quadratic extension is unramified and the Heegner hypothesis forces good reduction at $2$. Unramified twisting preserves good reduction, and a minimal differential remains minimal after unramified base change. Together with the surjectivity of the norm in Lemma~\ref{lem:local-twists}, this gives
\[
C_2(D)=e_2(D)=h_2(E,K)=0.
\]
If $D$ is even, then $(D,N)=1$ gives good reduction at $2$, so $v_2(\Delta)=0$ and $c_2(E)=1$; moreover, $v_2(D)=2$ or $3$. In a minimal generalized Weierstrass equation, reduction is ordinary when $a_1$ is odd and supersingular when $a_1$ is even. Set $\chi_2=(D,\Delta)_2$. The Kodaira type, minimal discriminant, and Tamagawa number of the twist are given in \cite[Theorem 5.1, the $I_0$ rows of Tables 5--6]{LocalTwist}; the norm indices are given in \cite[Propositions 4--5]{Kramer}. Together these results give Table~\ref{tab:dyadic}.
\begin{center}
\captionof{table}{Differential, Tamagawa, and norm corrections at $2$ for ramified quadratic extensions}\label{tab:dyadic}
\small
\begin{tabular}{@{}ccccccc@{}}
\toprule
Reduction of $E$ & $v_2(D)$ & $\operatorname{Kod}(E^D)$ & $v_2(\Delta_{E^D})$ & $e_2$ & $C_2$ & $h_2$\\
\midrule
ordinary & $2$ & $I_4^*$ & $12$ & $0$ & $(3+\chi_2)/2$ & $(3+\chi_2)/2$\\
ordinary & $3$ & $I_8^*$ & $18$ & $0$ & $(3+\chi_2)/2$ & $(3+\chi_2)/2$\\
supersingular & $2$ & $II^*$ & $12$ & $0$ & $0$ & $0$\\
supersingular & $3$ & $II$ & $6$ & $-1$ & $0$ & $1$\\
\bottomrule
\end{tabular}
\end{center}
In the ordinary case, choose a $2$-strongly-minimal model
\[
E:\ y^2+a_1xy+a_3y=x^3+a_2x^2+a_4x+a_6.
\]
By \cite[Proposition 3.5, the $I_0$ rows of Table 1]{LocalTwist},
$a_1$ is odd, $4\mid a_3$, and exactly one of the following holds:
\[
\begin{array}{ll}
\textup{(a)}&2\mid a_4,\quad a_6\text{ odd},\\
\textup{(b)}&a_4\text{ odd},\quad2\mid a_6.
\end{array}
\]
The discriminant of this model differs from the fixed minimal
discriminant by the twelfth power of a $2$-adic unit, so its residues
modulo $8$ and its Hilbert symbols are the same.
The discriminant formula reduces to
\[
\Delta\equiv a_3a_4-a_6+a_4^2+4a_2a_6\pmod8.
\]
Here the factor $a_1$ in the term $a_1a_3a_4$ may be omitted because
$a_1$ is odd and $4\mid a_3$.

Write $D=4m$ when $v_2(D)=2$, so that $m\equiv3\pmod4$.
The $I_0$ row of \cite[Theorem 5.1, Table 5]{LocalTwist} gives
$c_2(E^D)=4$ for $a_6\equiv0,3\pmod4$ and $c_2(E^D)=2$ for
$a_6\equiv1,2\pmod4$. Since
$\Delta\equiv a_4^2-a_6\pmod4$, these conditions give
\[
\begin{array}{c|c|c|c}
a_6\bmod4&\Delta\bmod4&(D,\Delta)_2&c_2(E^D)\\ \hline
0,3&1&1&4\\
1,2&3&-1&2
\end{array}
\]
Here $(D,\Delta)_2=(-1)^{(\Delta-1)/2}$.
This is the local calculation in
\cite[\S4, proof of Theorem 1.3, Case 1 and Subcase 1a]{BarriosMok}.

Now write $D=8m$, with $m$ odd, so the local twist parameter may be
taken to be $2m$. In cases \textup{(a)} and \textup{(b)}, respectively, put
\[
\begin{aligned}
P_1&=4+16a_2+8a_4+4a_6-2m-2ma_6^2-4ma_6,\\
P_2&=a_3^2-2ma_6^2+4a_6.
\end{aligned}
\]
The ordinary $I_0$ cases of
\cite[Theorem 5.1, Table 6; proof for $v(d)=1$, Subcases 1a--1b]{LocalTwist}
give
\[
c_2(E^D)=
\begin{cases}
4,&P_j\equiv0\pmod{32},\\
2,&P_j\equiv16\pmod{32},
\end{cases}
\qquad
j=1\text{ in \textup{(a)}},\quad j=2\text{ in \textup{(b)}}.
\]
Here
\[
(D,\Delta)_2
=(-1)^{(\Delta^2-1)/8+((m-1)/2)((\Delta-1)/2)}.
\]
To express the coefficient criterion in terms of $\Delta$, set
\[
(x,k)=
\begin{cases}
(a_2+a_4/2,\ (a_6+1)/2),&\text{in \textup{(a)}},\\
(a_3/4,\ a_6/2),&\text{in \textup{(b)}}.
\end{cases}
\]
The discriminant formula and the two polynomials above give
\[
\Delta\equiv1+4x-2k\pmod8,\qquad
\frac{P_j}{16}\equiv x+\frac{k-mk^2}{2}\pmod2.
\]
Here $(k-mk^2)/2$ is integral because $m$ is odd. For
$\delta=1+4x-2k$, the exponent in the Hilbert symbol satisfies
\[
\frac{\delta^2-1}{8}
+\frac{m-1}{2}\frac{\delta-1}{2}
\equiv x+\frac{k-mk^2}{2}\pmod2.
\]
Indeed, the difference is congruent modulo $2$ to
$(m+1)k(k-1)/2$, which is even. Thus the coefficient criterion is
equivalent to the following four residue conditions:
\[
\begin{array}{c|c|c|c|c}
m\bmod4&\Delta\bmod8&P_j\bmod32&(D,\Delta)_2&c_2(E^D)\\ \hline
1&1,7&0&1&4\\
1&3,5&16&-1&2\\
3&1,3&0&1&4\\
3&5,7&16&-1&2
\end{array}
\]
These are the residues listed in
\cite[\S4, proof of Theorem 1.3, Case 3, pp.~28--29]{BarriosMok}.
Although that paper assumes a positive global discriminant, the
calculation used here is local: \cite[Theorem 5.1]{LocalTwist} applies
to $d\in\Q_2^\times$ without a sign restriction. The displayed
congruences hold for every odd $m$, including negative $m$.
Equivalently, one may replace $m$ by a positive odd integer congruent
to it modulo $8$, which preserves its $\Q_2$-squareclass.

In both ordinary cases, therefore,
\[
c_2(E^D)=
\begin{cases}4,&\chi_2=1,\\2,&\chi_2=-1,\end{cases}
\]
Since $c_2(E)=1$, it follows that
$C_2(D)=v_2(c_2(E^D))=(3+\chi_2)/2$.
Kramer's norm formula \cite[Proposition 5]{Kramer} gives the same values
for $h_2(E,K)$, namely $2$ when $\chi_2=1$ and $1$ when $\chi_2=-1$.

In the supersingular case, Kramer's Proposition~4 gives $h_2=0$ when the twist parameter has even valuation and $h_2=1$ when it has odd valuation. With the squarefree representative $d_0=D/4$, these correspond to $v_2(D)=2$ and $3$, respectively. The local tables give types $II^*$ and $II$, both with Tamagawa number $1$, and minimal discriminant valuations $12$ and $6$. Here $\Delta\equiv5\pmod8$, so the Hilbert symbol formulas also give $\chi_2=1$ and $-1$, respectively.

Finally, the scaling relation between minimal discriminants and minimal differentials gives
\[
\Delta_{E^D}=D^6\eta_D^{12}\Delta_E,
\qquad e_2(D)=\frac{v_2(\Delta_{E^D})-6v_2(D)}{12}.
\]
Substitution of the discriminant valuations gives $e_2=0$ in the first three rows and $e_2=-1$ in the last. These calculations use the fundamental discriminant $D$ as the twist parameter. If one uses $d_0=D/4$, then $\sqrt D=2\sqrt{d_0}$ and the differential comparison gives $\eta_D=\eta_{d_0}/2$. Thus all four rows satisfy
\[
C_2-e_2=h_2(E,K),\qquad(-1)^{h_2(E,K)}=(D,\Delta_E)_2.
\]

By \eqref{eq:index-definition}, intrinsic index zero is equivalent to oddness of all original Tamagawa numbers together with
\begin{equation}\label{eq:intrinsic-local-test}
h_2(E,K)+\sum_{r\mid D,\ r\text{ odd}}t_r=1-\epsilon_\infty.
\end{equation}
For an odd prime $r\mid D$, the value of $t_r$ is $2$, $1$, or $0$ according as residual Frobenius is the identity, a transposition, or a $3$-cycle. Assuming all original Tamagawa numbers are odd, \eqref{eq:intrinsic-local-test} becomes the following test.
\begin{itemize}
\item If $\Delta_E>0$, then $\epsilon_\infty=1$. One requires $h_2=0$ and a $3$-cycle at every odd $r\mid D$. The dyadic condition is automatic for odd $D$. For even $D$, it is equivalent to supersingular reduction and $v_2(D)=2$.
\item If $\Delta_E<0$, then $\epsilon_\infty=0$. There are two possibilities: either $h_2=0$ and exactly one odd $r\mid D$ has transposition Frobenius, with $3$-cycles at the others; or $h_2=1$ and all such Frobenius elements are $3$-cycles. The first case allows odd $D$, or even $D$ with supersingular reduction and $v_2(D)=2$. The second corresponds to ordinary reduction with $\chi_2=-1$, or supersingular reduction with $v_2(D)=3$.
\end{itemize}
\section{Finite-level Selmer groups and the finite quotient of $\Sha$}\label{subsec:finite-size}
\begin{proposition}\label{prop:usual-structure}
Let $C$ be an elliptic curve over a number field $F$, with $C(F)[2]=0$. Put
\[
S_F=\Sel_{2^\infty}(C/F),\qquad X_F=\Sha(C/F)[2^\infty],\qquad
r_F=\rank C(F),\qquad s_F=\corank_{\Z_2}S_F,
\]
and let $\overline X_F$ be the quotient of $X_F$ by its maximal divisible subgroup. There exist $\eta_F\ge0$ and a finite multiset $e_1\ge\cdots\ge e_t\ge1$ such that there are noncanonical isomorphisms of abstract groups
\begin{align}
X_F&\simeq(\Q_2/\Z_2)^{\eta_F}\oplus\bigoplus_{i=1}^{t}(\Z/2^{e_i}\Z)^2,\label{eq:X-structure}\\
S_F&\simeq(\Q_2/\Z_2)^{s_F}\oplus\bigoplus_{i=1}^{t}(\Z/2^{e_i}\Z)^2,
\qquad s_F=r_F+\eta_F.\label{eq:S-structure}
\end{align}
For every $m\ge1$, coefficient inclusion identifies
\begin{equation}\label{eq:finite-Sel-structure}
\Sel_{2^m}(C/F)=S_F[2^m]\simeq(\Z/2^m\Z)^{s_F}
\oplus\bigoplus_i(\Z/2^{\min(m,e_i)}\Z)^2.
\end{equation}
In particular, $X_F$ is finite if and only if $r_F=s_F$.
\end{proposition}
\begin{proof}
The usual Kummer exact sequence is
\begin{equation}\label{eq:infty-kummer}
0\longrightarrow C(F)\otimes\Q_2/\Z_2
\longrightarrow S_F\longrightarrow X_F\longrightarrow0.
\end{equation}
By the Mordell--Weil theorem, the left term is isomorphic to $(\Q_2/\Z_2)^{r_F}$. A divisible abelian group is injective, so the sequence splits noncanonically. The Pontryagin dual of the Selmer group is a finitely generated $\Z_2$-module. The structure theorem therefore expresses both $S_F$ and its quotient $X_F$ as a direct sum of finitely many copies of $\Q_2/\Z_2$ and a finite $2$-group. In particular, $\overline X_F$ is finite.

The principal polarization of an elliptic curve is induced by the rational divisor $[O]$, so its polarization obstruction vanishes. The Cassels--Tate pairing on $\overline X_F$ is therefore perfect and alternating \cite[\S1, Corollary 7]{PS}. If $\overline X_F\ne0$, choose $x$ of maximal order $2^e$. Perfectness provides $y$ such that $\langle x,y\rangle$ also has order $2^e$. Multiplying $y$ by an odd integer, we may arrange $\langle x,y\rangle=2^{-e}\bmod\Z$. Maximality implies that $y$ has order $2^e$. If $ux+vy=0$, pairing with $x$ and $y$ gives $2^e\mid u,v$. Hence
\[
\langle x,y\rangle_{\mathrm{grp}}\simeq(\Z/2^e\Z)^2.
\]
The restriction of the pairing to this subgroup is perfect. Subtracting a suitable $ux+vy$ from any $z\in\overline X_F$ makes it orthogonal to both generators. Thus $\overline X_F$ is the direct sum of this subgroup and its orthogonal complement. Induction gives the paired elementary divisors, as in the group-theoretic conclusions of \cite[Theorem 8, Corollary 9]{PS}. Finally, comparison of divisible parts in \eqref{eq:infty-kummer} yields $s_F=r_F+\eta_F$ and the finiteness criterion.

For $J=F$ or a completion $F_v$, the finite- and infinite-level Kummer sequences are
\[
\begin{aligned}
0&\longrightarrow C(J)/2^mC(J)\longrightarrow H^1(J,C[2^m])
\longrightarrow H^1(J,C)[2^m]\longrightarrow0,\\
0&\longrightarrow C(J)\otimes\Q_2/\Z_2\longrightarrow H^1(J,C[2^\infty])
\longrightarrow H^1(J,C)[2^\infty]\longrightarrow0.
\end{aligned}
\]
The coefficient inclusion induces a commutative map between these rows whose rightmost map is injective. A finite-level class therefore maps into the infinite-level Kummer subgroup if and only if its image in the right term of the first row is zero, equivalently if it lies in the finite-level Kummer subgroup. Since $C(F)[2]=0$ implies $C(F)[2^\infty]=0$, the coefficient sequence gives the global identification $H^1(F,C[2^m])\simeq H^1(F,C[2^\infty])[2^m]$. Intersecting the corresponding local conditions yields
\[
\Sel_{2^m}(C/F)=S_F[2^m].
\]
Taking finite torsion in \eqref{eq:S-structure} proves \eqref{eq:finite-Sel-structure}.
\end{proof}

We apply this proposition to $C=E$ over $F=K$, using the notation $S_K,s_K,r_K,X_K,\overline X_K$ and $e_i$ for this pair.

\begin{corollary}\label{cor:sha-finite}
Suppose $s_K\le B$ and there exist $B$ independent $K$-rational points. Then
\[
r_K=s_K=B,\qquad\Sha(E/K)[2^\infty]\text{ is finite}.
\]
Under the intrinsic-index-zero hypothesis, one may take $B=2R+1$ by \eqref{eq:combined-bound}.
\end{corollary}
\begin{proof}
We have $B\le r_K\le s_K\le B$. Apply Proposition~\ref{prop:usual-structure}.
\end{proof}
Set $u_m=\log_2\#\Sel_{2^m}(E/K)$, $u_0=0$, and $\Delta_m=u_m-u_{m-1}$. By \eqref{eq:finite-Sel-structure},
\begin{align}
u_m&=ms_K+2\sum_i\min(m,e_i),\label{eq:size-profile}\\
\Delta_m&=s_K+2\#\{i:e_i\ge m\},\qquad
\#\{i:e_i=m\}=\frac{\Delta_m-\Delta_{m+1}}2.\label{eq:size-differences}
\end{align}

In particular,
\[
\dim_{\F_2}\Sel_2(E/K)=s_K+2t,
\qquad \#\overline X_K=2^{2\sum_i e_i}.
\]
Since $\min(m,e_i)-\min(m-1,e_i)$ is $1$ precisely when $e_i\ge m$, \eqref{eq:size-differences} shows that $\Delta_m$ is nonincreasing. Each decrease records the number of pairs of elementary divisors ending at that level. The finite group $\overline X_K$ determines the eventual constant term, while the corank determines the slope. Combining \eqref{eq:size-profile} with $s_K=s_2(E)+s_2(E^D)\le a_0+a_1$ gives
\[
\lim_{m\to\infty}\frac{u_m}{m}=s_K\le a_0+a_1,
\qquad
\#\Sel_{2^m}(E/K)\le\#\overline X_K\,2^{m(a_0+a_1)}.
\]
When $\intr(E,K)=0$, Proposition~\ref{prop:nesting} gives $a_0+a_1=2R+1$, so the exponent becomes $m(2R+1)$. For $m\ge\max_i e_i$, we have $u_m=ms_K+\log_2\#\overline X_K$; if $t=0$, this identity holds for every $m$.

If $s_K$ is known independently and exact descent gives $\Delta_m=s_K$ at some level, then every $e_i<m$, so
\[
\log_2\#\overline X_K=u_m-ms_K.
\]
A factor $(\Z/2^e)^2$ contributes $2$ to each of the first $e$ differences, producing a plateau of length $e$. Thus finite plateaus can include contributions from the finite quotient. Once Corollary~\ref{cor:sha-finite} or another argument proves $r_K=s_K$, the divisible part of $\Sha$ vanishes, and the displayed formula computes the order of the entire group $\Sha(E/K)[2^\infty]$.

\section{Proofs of the Kolyvagin and Selmer structure theorems}\label{app:kolyvagin-structure}
For the descent argument, assume \Heegner, \Surjectivity{3}, and $T_E=0$. Lemma~\ref{lem:mod-eight} gives
\begin{equation}\label{eq:structure-hypotheses}
\rho_{E,2^\infty}(G_\Q)=\GL_2(\Z_2),\qquad T_E=0.
\end{equation}
These assumptions are used for uniform Chebotarev detection and for the local conditions at the bad primes, respectively.

The argument follows the local duality, prime replacement, and independence constructions of \cite{KolyvaginSha,KolyvaginPreprint}. Kolyvagin's set $B(E)$ consists of odd primes, so the applicable comparison is the arbitrary-prime formulation on printed pp.~10--12 of \cite{KolyvaginPreprint}. The original derivative $I_q=-D_q$ changes our class by the unit $(-1)^{\nu(n)}$. Section~\ref{sec:strong} identifies the original coefficient convention and the images $\tau'_{\lambda,n}$ with our classes, including the choice $k_0=0$. The CM norm, reduction, and conjugation identities used below are those of \cite[\S1, equations (3)--(5)]{KolyvaginSha}. We adapt the detection and prime-replacement arguments to the present hypotheses, keeping their losses uniform in the coefficient depth. These estimates and the resulting corank argument form the proof given here.

Put $R_M=\Z/2^M\Z$, $A_M=E[2^M]$, and abbreviate $\kappa_{n,M}=\kappa^{\std}_{n,M}(\varphi)$ for the fixed parametrization. For a finite-order element $x$, write
\[
\ell(x)=\log_2\ord(x),\qquad \ell(0)=0.
\]
We use conductors $n\in\mathcal N_{M+k}$ with $k\ge2$. Constants in the following estimates are independent of $M$, of the additional precision $k$, and of the auxiliary primes; they may depend on a specified bound for $\nu(n)$.

\subsection{The local calculation for standard derivatives}\label{app:kolyvagin-local}
The Weil pairing and local invariant give the perfect local Tate pairing
\[
\langle\ ,\ \rangle_v:
H^1(K_v,A_M)\times H^1(K_v,A_M)
\longrightarrow \tfrac1{2^M}\Z/\Z.
\]
The Kummer subgroup $H^1_{\fin}(K_v,A_M)$ is its own orthogonal complement. These are the self-duality statements for $[2^M]$ in \cite[Chapter I, Corollaries 2.3 and 3.4 and Remark 3.5]{Milne}. Consequently the singular localization of a global class defines a functional on usual Selmer classes by
\begin{equation}\label{eq:testing-functional}
\ell_{v,z}(x)=\langle\loc_vx,\loc_vz\rangle_v.
\end{equation}
Brauer reciprocity gives, for any two global classes,
\begin{equation}\label{eq:global-reciprocity}
\sum_v\langle\loc_vx,\loc_vz\rangle_v=0.
\end{equation}
Terms vanish wherever both localizations lie in the same self-orthogonal local condition.

Let $q$ be inert in $K$, with $2^{M+k}\mid q+1,a_q$, and write $\pi=\Frob_q$ on the reduction of $E$. Its characteristic polynomial implies $\pi^2=1$ on $A_M$, so $A_M$ is rational over $K_q$. Local class field theory gives a decomposition
\begin{equation}\label{eq:deep-local-splitting}
H^1(K_q,A_M)=H^1_{\fin}(K_q,A_M)\oplus H^1_{\trans}(K_q,A_M).
\end{equation}
Here the first summand consists of unramified classes, and the second consists of characters vanishing on the uniformizer $q$. Equivalently, the latter is the kernel of restriction to the degree-$2^M$ subextension of the local ring class extension. Each summand is free of rank two over $R_M$ and is self-orthogonal. For the transverse summand this also follows directly from the tame Hilbert symbol: $q$ is a norm from this extension and $-1$ is a $2^M$th power in $K_q$. The pairing between the two summands is perfect.

Choose a generator $\sigma_q$ of the ring class inertia group. Evaluation on unramified Frobenius and on $\sigma_q$ identifies the two summands with $A_M$; write these coordinates as $U_q$ and $V_q$, respectively. The normalization of $\sigma_q$ changes $V_q$ by an odd unit. Under these coordinates the finite--transverse pairing is the Weil pairing, with the corresponding unit normalization.

\begin{lemma}\label{lem:deep-derivative-local}
If $n\in\mathcal N_{M+k}$ and $k\ge2$, then $\kappa_{n,M}$ is finite at every place outside $n$ and transverse at every $q\mid n$. There is an isomorphism
\[
\Phi_q:H^1_{\fin}(K_q,A_M)\longrightarrow H^1_{\trans}(K_q,A_M),
\qquad V_q\Phi_q=\pi U_q,
\]
such that
\begin{equation}\label{eq:deep-finite-transverse}
\loc_q\kappa_{n,M}=\Phi_q(\loc_q\kappa_{n/q,M}).
\end{equation}
Moreover,
\begin{equation}\label{eq:derivative-sign}
c\kappa_{n,M}=\epsilon_{\nu(n)}\kappa_{n,M},
\qquad \epsilon_r=(-1)^r(-w(E)).
\end{equation}
The map $\Phi_q$ reverses the cohomological conjugation sign.
\end{lemma}
\begin{proof}
Write $m=2^M$ and use compatible representatives in the full trace defining $W_n$ in Section~\ref{sec:cm}. The ideal $(q)$ is principal in the order of conductor $n/q$, so it splits completely in $H_{n/q}/K$. Thus $(H_n)_w/K_q$ is totally ramified cyclic of degree $q+1$, with residue field $\F_{q^2}$. The CM congruence $\widetilde{g y_n}=\pi\widetilde{g y_{n/q}}$ for the compatible conjugates $g$ \cite[\S1, equation (4)]{KolyvaginSha} and the definition of $D_q$ give
\begin{equation}\label{eq:local-W-reduction}
\widetilde W_n=\frac{q(q+1)}2\,\pi\widetilde W_{n/q}.
\end{equation}
The coefficient on the right is divisible by $m$. Multiplication by $m$ is an isomorphism on the formal group at $q$, so a residue-field division point lifts to $Q\in E((H_n)_w)$ with $mQ=W_n$. The restriction of $\kappa_{n,M}$ to this local extension is therefore zero. Its kernel is precisely the transverse subgroup in \eqref{eq:deep-local-splitting}.

The descent cocycle is
\[
z(g)=gQ-Q+A_g,\qquad mA_g=(1-g)W_n,\quad A_g\in E(H_n).
\]
The global point $A_g$ exists by invariance modulo $m$ and is unique by Lemma~\ref{lem:no-torsion}. Let $J$ be the sum of the compatible coset representatives and put $B_{n,q}=JD_{n/q}y_n$. The identity
$(1-\sigma_q)D_q=\Norm_q-(q+1)$ and the norm relation give
\[
A_{\sigma_q}=\frac{a_q}{m}W_{n/q}-\frac{q+1}{m}B_{n,q},
\qquad
\widetilde B_{n,q}=\pi\widetilde W_{n/q}.
\]
The coefficients are integers by admissibility. Since $\sigma_q$ acts
trivially on the residue field, reducing the cocycle gives
\[
V_q(\loc_q\kappa_{n,M})
=\left(\frac{a_q}{m}-\frac{q+1}{m}\pi\right)\widetilde W_{n/q}.
\]
To compare with the finite coordinate, choose $R$ over the algebraic residue field with $mR=\widetilde W_{n/q}$ and put $u=(\pi^2-1)R$. Since $\pi^2\widetilde W_{n/q}=\widetilde W_{n/q}$, we have $u\in A_M$, and $u$ is the finite Kummer coordinate. From $\pi^2-a_q\pi+q=0$,
\[
\pi u=(a_q-(q+1)\pi)R+a_qu
=\left(\frac{a_q}{m}-\frac{q+1}{m}\pi\right)\widetilde W_{n/q}.
\]
Here $a_qu=0$. This proves \eqref{eq:deep-finite-transverse}. Conjugation acts on the finite coordinate by $\pi$ and on the tame coordinate by $q^{-1}\pi=-\pi$ modulo $m$, proving the sign reversal.

Outside $n$, the ring class extension is unramified. The image of $\kappa_{n,M}$ in $H^1(K_v,E)[m]$ is therefore unramified. At a place of good reduction it is zero by the N\'eron model and Lang's theorem, including when $v\mid2$. At a bad place the Heegner hypothesis makes $K/\Q$ split, and $T_E=0$ makes the rational Tamagawa number odd. The unramified torsor group is $H^1(k_v,\Phi_v)$ for the component group \cite[Chapter I, Proposition 3.8]{Milne}. On its finite geometric points, Frobenius invariants and coinvariants have the same order; the invariants have order $c_v(E)$. Its $2$-primary part is consequently zero. This proves the finite condition outside $n$.

Finally, the conjugation formula for Heegner points \cite[\S1, equation (5)]{KolyvaginSha} is $cy_n=-w(E)\sigma y_n$ modulo rational torsion, for a ring class automorphism $\sigma$. Conjugation inverts each $\sigma_q$, and
$D_q(\sigma_q^{-1})\equiv-D_q(\sigma_q)$ modulo the norm operator and $q+1$. After application to the points, the norm terms are multiples of $a_q$, hence vanish modulo $m$. The full trace removes $\sigma$ modulo $m$, and rational torsion has odd order by the residual $S_3$ hypothesis. This gives \eqref{eq:derivative-sign}.
\end{proof}

For $\eta=\pm1$, a superscript $\eta$ denotes $\ker(c-\eta)$. On $T_2E$, an involution of determinant $-1$ has a primitive $+1$-eigenvector. In a basis beginning with this vector its matrix is
$\left(\begin{smallmatrix}1&a\\0&-1\end{smallmatrix}\right)$. Changing the second basis vector changes $a$ by an element of $2\Z_2$, so the two integral normal forms are
\[
\begin{pmatrix}1&0\\0&-1\end{pmatrix},
\qquad
\begin{pmatrix}1&1\\0&-1\end{pmatrix}.
\]
Their primitive eigenvectors span a sublattice of index one and two, respectively.

This applies uniformly to every prime above. Indeed, on $T_2E$ put
\[
B_q=\pi-a_q/2,\qquad
u_q^2=a_q^2/4-q,\quad u_q\equiv1\pmod{2^{M+k-1}}.
\]
The square root exists since $M+k\ge3$, and $C_q=B_q/u_q$ is an integral involution congruent to $\pi$ modulo $2^M$. Thus each signed finite or transverse group is one cyclic direction, with a kernel and cokernel killed by a fixed power of $2$. Choose primitive eigenvectors $e_{q,\pm}$ for $C_q$. One may take scalar coordinates
\[
\psi_q^\eta(x)=[U_q(x),e_{q,-\eta}]_M\in R_M
\]
on the signed finite subgroup, where brackets denote the additive Weil pairing after a choice of roots of unity. On a global class, $\psi_q^\eta$ denotes its composition with localization. Their kernels are killed by $2$. For finite classes $x$ of sign $\eta$ and $y$ of sign $-\eta$, one may take $c_{\rm loc}=1$ in
\begin{equation}\label{eq:signed-pairing-estimate}
\ell\bigl(\langle x,\Phi_qy\rangle_q\bigr)
\ge \ell(x)+\ell(y)-M-c_{\rm loc}
\end{equation}
whenever the right side is positive. Indeed, in the first normal form, the coordinates of classes of opposite signs have the shape
\[
U_q(x)=(a,2^{M-1}b),\qquad U_q(y)=(2^{M-1}c,d).
\]
For $M\ge2$, their pairing is $ad$ up to an odd unit. Positivity of the asserted lower bound forces both orders to exceed two, so these orders are determined by $a$ and $d$. In the second normal form the coordinates are $(a,0)$ and $(b,-2b)$, and the pairing is $2ab$ up to an odd unit, giving exactly the stated loss. These computations also show that the kernels of $\psi_q^\eta$ are killed by $2$. Interchanging the signs preserves all order estimates. For $M=1$ the lower bound is never positive, so the assertion remains valid.

\subsection{Chebotarev detection with a uniform loss}\label{app:kolyvagin-detection}
\begin{lemma}\label{lem:deep-character-detection}
There is an integer $c_{\rm ch}\ge0$, independent of the coefficient
depth, the prescribed finite subgroups, and the auxiliary primes, with
the following property; the proof below gives $c_{\rm ch}=8$.
For every $M\ge1$, every pair of finite subgroups
$C_\eta\subset H^1(K,A_M)^\eta$, every pair of characters
$\chi_\eta:C_\eta\to R_M$ for $\eta=\pm1$, every $H\ge M+2$,
and every finite set $\Sigma$ of rational primes, there are infinitely
many primes $q\notin\Sigma$, $q\nmid2ND$, inert in $K$, such that
$2^H\mid q+1,a_q(E)$ and every class in $C_++C_-$ is unramified at $q$.
For each such prime, local coordinates can be chosen so that
\begin{equation}\label{eq:prescribed-local-characters}
\psi_q^\eta(x)=2^{c_{\rm ch}}\chi_\eta(x)
\qquad(x\in C_\eta,\ \eta=\pm1).
\end{equation}
For other primitive eigenvector and root-of-unity normalizations,
the right side is $u_{q,\eta}2^{c_{\rm ch}}\chi_\eta(x)$ with
$u_{q,\eta}\in R_M^\times$ independent of $x$.
In particular, given $x_\eta\in C_\eta$, the characters can be chosen
so that the same prime satisfies
\[
\ell(\psi_q^\eta(x_\eta))\ge\ell(x_\eta)-c_{\rm ch},
\qquad
\ell(\loc_qx_\eta)\ge\ell(x_\eta)-c_{\rm ch}
\qquad(\eta=\pm1).
\]
\end{lemma}
\begin{proof}
Let $L=K(E[2^H])$ and $T=T_2E$. The image of $G_K$ has index at most two in $\GL_2(\Z_2)$ and contains $1+8M_2(\Z_2)$. For the latter assertion, every $1+8B$ is a square in $1+4M_2(\Z_2)$, by successive square-root lifting, and an index-two quotient kills squares. In particular the scalar $9I$ belongs to $\Gal(L/K)$. The central-scalar cocycle identity
\[
8z(g)=(g-1)z(9I)
\]
shows that $H^1(\Gal(L/K),A_M)$ is killed by $8$. Inflation--restriction therefore says that the kernel of restriction to $G_L$ is killed by $8$.

Choose primitive eigenvectors $e_\pm$ for actual complex conjugation on $T$. On the external direct sum $C=C_+\oplus C_-$ define
\[
F:C\longrightarrow\Hom(G_L,R_M),\qquad
F(h_+,h_-)(z)=2[h_+(z),e_-]_M+2[h_-(z),e_+]_M.
\]
Restrictions to $G_L$ are genuine homomorphisms. We claim that $\ker F$ is killed by $2^8$. If $F(h_+,h_-)=0$, conjugate $z$ by an element of $G_K$ with image $1+8B$ and subtract the original equality. Equivariance of restricted cocycles gives
\[
16\bigl([Bh_+(z),e_-]_M+[Bh_-(z),e_+]_M\bigr)=0.
\]
Let $B^*$ be the adjoint for the Weil pairing. Since the eigenlattice has index at most two, for every $x\in T$ one can choose integral $B^*$ with
$B^*e_+=0$ and $B^*e_-=2x$. Hence $32[h_+(z),x]_M=0$ for every $x,z$. Interchanging the signs gives $32h_\pm|_{G_L}=0$, and the restriction-kernel bound gives $256h_\pm=0$, as asserted.

The character $2^8(\chi_++\chi_-)$ kills $\ker F$ and hence factors
through $F(C)$. The evaluation map
\[
G_L\longrightarrow\operatorname{Hom}(F(C),R_M)
\]
is surjective: after quotienting by the common kernel of the characters
in $F(C)$, this is the usual duality for a finite abelian group of
exponent dividing $2^M$. Choose $t\in G_L$ with
\[
F(x_+,x_-)(t)=2^8\bigl(\chi_+(x_+)+\chi_-(x_-)\bigr)
\qquad(x_\eta\in C_\eta).
\]
Take a finite normal extension of $\Q$ containing $L$ and the fields
cut out by the restricted cocycles, and exclude its ramified primes,
the ramified primes of the classes in $C_++C_-$, the primes dividing
$2ND$, and $\Sigma$.
Chebotarev gives infinitely many primes with Frobenius conjugate to
$tc$; choose the place above each prime so that this representative is
used. Since $t\in G_L$, such a prime is inert in $K$ and its action on
$E[2^H]$ is that of $c$, so $2^H\mid q+1,a_q(E)$.
For $x\in C_\eta$, the finite coordinate is
\[
U_q(\loc_qx)=x((tc)^2)=(1+\eta c)x(t).
\]
The equality is independent of the cocycle representative because
$(tc)^2\in G_L$ acts trivially on $A_M$. Moreover,
\[
[(1+\eta c)x(t),e_{-\eta}]_M
=2[x(t),e_{-\eta}]_M
=2^8\chi_\eta(x).
\]
This gives \eqref{eq:prescribed-local-characters} with $c_{\rm ch}=8$
for coordinates transported from the fixed complex-conjugation basis.

These coordinates agree with the local integral-eigenvector
coordinates used above. Indeed, $\pi\equiv c\pmod{2^H}$ and
$C_q=(\pi-a_q/2)/u_q$ satisfy
$C_q\equiv c\pmod{2^{H-1}}$. Thus
\[
e_{q,\eta}:=\frac{1+\eta C_q}{2}e_\eta
\]
is an integral primitive $\eta$-eigenvector for $C_q$, and
$e_{q,\eta}\equiv e_\eta\pmod{2^M}$ because $H\ge M+2$.
Any other primitive generator of the same integral eigenline differs
by a unit. A change of the root-of-unity normalization also contributes
a unit, proving the stated form with $u_{q,\eta}$.
Finally, a character of the finite group $C_\eta$ can be chosen to take
$x_\eta$ to an element of the same order. Equation
\eqref{eq:prescribed-local-characters} then gives both order estimates.
\end{proof}

\subsection{Relaxed Selmer groups and replacement of primes}\label{app:kolyvagin-replacement}
For squarefree $\lambda\in\mathcal N_{M+k}$ and $\delta\mid\lambda$, with $\beta=\lambda/\delta\ne1$, define
\[
\mathcal R_{\lambda,\delta,M}
=\{x\in H^1(K,A_M):
\loc_vx\in H^1_{\fin}\ (v\nmid\lambda),\
\loc_px\in H^1_{\trans}\ (p\mid\delta)\}.
\]
There is no condition at the primes of $\beta$.

\begin{lemma}\label{lem:relaxed-large-order}
For each bound $r$ there is a constant $c_r$ such that, if $\nu(\lambda)\le r$, each group $\mathcal R_{\lambda,\delta,M}^{\eta}$ contains an element of order at least $2^{M-c_r}$. One may take $c_r=18$ for every $r$.
\end{lemma}
\begin{proof}
Put $b=\nu(\beta)\ge1$ and let $Z_M$ be the localization image in
\[
\mathcal H_{\beta,M}=\bigoplus_{p\mid\beta}H^1(K_p,A_M)\simeq R_M^{4b}.
\]
At the primes dividing $\beta$, the strict condition $0$ and the relaxed condition $H^1(K_p,A_M)$ are dual; the prescribed conditions elsewhere are self-dual. Poitou--Tate duality \cite[Chapter I, Theorem 4.10]{Milne} therefore gives the comparison sequence
\[
\mathcal R_{\lambda,\delta,M}
\xrightarrow{\ \loc_\beta\ }\mathcal H_{\beta,M}
\xrightarrow{\ \alpha\ }\mathcal R_{\lambda,\delta,M}^{\vee},
\]
exact at the middle term, where $G^\vee=\Hom(G,\Q/\Z)$ and
\[
\alpha(z)(x)=\sum_{p\mid\beta}\langle z_p,\loc_p x\rangle_p.
\]
Its middle kernel is both $Z_M$ and $Z_M^\perp$. Thus
\begin{equation}\label{eq:relaxed-lagrangian}
Z_M=Z_M^\perp,\qquad \log_2|Z_M|=2Mb.
\end{equation}

Every subgroup of $\mathcal H_{\beta,M}\simeq R_M^{4b}$ has at most $4b$ generators. The two local normal forms give
$|\mathcal H_{\beta,M}^{\eta}|\le2^{2Mb+2b}$.
Let $D_\eta$ be the radical of the pairing restricted to this signed subgroup. Cross-sign pairings are killed by $2$, so twice an element of $D_\eta$ annihilates both signs. Since $2x=(x+cx)+(x-cx)$, a further multiplication by $2$ annihilates all of $\mathcal H_{\beta,M}$. Perfectness gives $4D_\eta=0$, and the generator bound yields $|D_\eta|\le2^{8b}$.

The image of $Z_M^\eta$ in the nondegenerate quotient
$\mathcal H_{\beta,M}^\eta/D_\eta$ is isotropic. Hence
\[
|Z_M^\eta|^2\le|\mathcal H_{\beta,M}^\eta|\,|D_\eta|,
\qquad \log_2|Z_M^\eta|\le Mb+5b.
\]
The quotient $Z_M/(Z_M^++Z_M^-)$ is killed by $2$ and has at most $4b$ generators. Its order is therefore at most $2^{4b}$. Using \eqref{eq:relaxed-lagrangian} gives
$\log_2|Z_M^+|+\log_2|Z_M^-|\ge2Mb-4b$, and consequently
$\log_2|Z_M^\eta|\ge Mb-9b$.
Applying $1+\eta c$ to a lift of an element of $Z_M^\eta$ gives
\[
2Z_M^\eta\subseteq\loc_\beta(\mathcal R_{\lambda,\delta,M}^\eta)
\subseteq Z_M^\eta.
\]
Since $|Z_M^\eta[2]|\le2^{4b}$, we obtain
\begin{equation}\label{eq:relaxed-size}
b(M-13)\le
\log_2|\loc_\beta(\mathcal R_{\lambda,\delta,M}^\eta)|
\le b(M+5).
\end{equation}

For $1\le m\le M$, coefficient inclusion identifies
\[
\mathcal R_{\lambda,\delta,m}
=\mathcal R_{\lambda,\delta,M}[2^m].
\]
Global injectivity follows from $E(K)[2^\infty]=0$. The usual local condition pulls back by the Kummer sequence, and the transverse condition pulls back under coefficient inclusion because $A_M$ is rational at every auxiliary prime. These identifications commute with conjugation and with localization; the latter coefficient maps at $\beta$ are injective. Write $2^e$ for the exponent of $\mathcal R_{\lambda,\delta,M}^\eta$, with $e=0$ if this group is zero. In that case its localization image is zero, and \eqref{eq:relaxed-size} gives $M\le13$. If $e\ge1$, coefficient inclusion identifies this group with $\mathcal R_{\lambda,\delta,e}^\eta$ and preserves the order of its localization image. The same conductor has additional precision $M+k-e\ge2$ at level $e$, so \eqref{eq:relaxed-size} at levels $M$ and $e$ gives
\[
b(M-13)\le b(e+5).
\]
Since $b\ge1$, we have $M-e\le18$ in both cases. An element of maximal order proves the lemma with $c_r=18$.
\end{proof}

With the local coordinates above, $c_{\rm loc}=1$ uniformly in
$M$, $k$, and $q$. The character-detection constant
$c_{\rm ch}=8$ is independent also of $H$, $C_\pm$, $\chi_\pm$,
and the finite excluded set. Lemma~\ref{lem:relaxed-large-order}
allows the uniform choice $c_r=18$, independent of the support bound
and of $M,k,\lambda,\delta$ and their prime factors.

Suppose now that the standard system satisfies strong nonvanishing, and let $f=\fstr$. Minimality gives a single $k_*\ge2$ such that
\begin{equation}\label{eq:lower-support-vanishing}
\kappa_{n,M}=0\quad
\text{if }\nu(n)<f,\quad n\in\mathcal N_{M+k_*},\quad M\ge1.
\end{equation}
Indeed, for each support size below $f$, negate the strong condition and then take the maximum of the finitely many resulting precision thresholds. Increasing precision only shrinks the conductor set. Lemma~\ref{lem:deep-derivative-local} consequently gives
\begin{equation}\label{eq:minimal-support-selmer}
\kappa_{n,M}\in\Sel_{2^M}(E/K)^{\epsilon_f},
\qquad \loc_p\kappa_{n,M}=0\quad(p\mid n)
\end{equation}
for every support-$f$ conductor in $\mathcal N_{M+k_*}$.

\begin{lemma}\label{lem:prime-replacement}
There is a constant $d_f$ with the following property. Suppose $\nu(\lambda)=f$, $\lambda\in\mathcal N_{M+k}$, $k\ge k_*$, and $\delta\mid\lambda$ is a proper divisor. One can choose a new prime $q$, as deep as desired, and $p\mid\lambda/\delta$ such that
\[
\ell(\psi_q^{\epsilon_f}(\kappa_{\lambda,M}))
\ge\ell(\kappa_{\lambda,M})-c_{\rm ch},\qquad
\ell(\kappa_{\lambda q/p,M})\ge\ell(\kappa_{\lambda,M})-d_f.
\]
The new conductor retains $\delta$ and contains $q$.
\end{lemma}
\begin{proof}
Choose $s\in\mathcal R_{\lambda,\delta,M}^{-\epsilon_f}$ of order at least $2^{M-c_f}$ by Lemma~\ref{lem:relaxed-large-order}. Lemma~\ref{lem:deep-character-detection} supplies a fresh $q$ detecting both $s$ and $\kappa_{\lambda,M}$, with uniform loss. Pair $s$ with $\kappa_{\lambda q,M}$ in \eqref{eq:global-reciprocity}. Outside $\lambda q$ both are finite, and at $\delta$ both are transverse. Thus
\begin{equation}\label{eq:replacement-reciprocity}
\sum_{p\mid\lambda/\delta}
\langle\loc_ps,\Phi_p(\loc_p\kappa_{\lambda q/p,M})\rangle_p
=-\langle\loc_qs,\Phi_q(\loc_q\kappa_{\lambda,M})\rangle_q.
\end{equation}
The detection inequalities give
\[
\ell(\loc_qs)\ge M-c_f-c_{\rm ch},
\qquad
\ell(\loc_q\kappa_{\lambda,M})
\ge\ell(\psi_q^{\epsilon_f}(\kappa_{\lambda,M}))
\ge\ell(\kappa_{\lambda,M})-c_{\rm ch}.
\]
Consequently the signed pairing estimate gives
\[
\ell\!\left(
\langle\loc_qs,\Phi_q(\loc_q\kappa_{\lambda,M})\rangle_q
\right)
\ge\ell(\kappa_{\lambda,M})
-c_f-2c_{\rm ch}-c_{\rm loc}
\]
whenever the right side is positive. Thus one may take
\[
d_f=c_f+2c_{\rm ch}+c_{\rm loc}.
\]
This constant is uniform in the coefficient depth, the additional
precision, the conductor, the retained subproduct, and the chosen
replacement primes.
On the left only the finite projection of $\loc_ps$ contributes. A sum in the cyclic group $\tfrac1{2^M}\Z/\Z$ has order at most the largest order of its summands. At least one term on the left therefore has order at least $2^{\ell(\kappa_{\lambda,M})-d_f}$. Bilinearity bounds its order by that of $\kappa_{\lambda q/p,M}$, proving the assertion. If $\ell(\kappa_{\lambda,M})\le d_f$, keep the detecting prime $q$ chosen above and take any $p\mid\lambda/\delta$.
\end{proof}

\subsection{Proof of the corank theorem}\label{app:kolyvagin-corank-proof}
We first justify the rational decomposition used in the statement.
The rational Tate space
$\mathscr V_K=(\varprojlim_M S_K[2^M])\otimes_{\Z_2}\Q_2$
has dimension $s_K$ and eigenspace dimensions $s^\pm$. Restriction and corestriction preserve the local conditions and satisfy
$\cor\res=[2]$ and $\res\cor=1+c$. Applying these identities also to $E^D$, whose twisting isomorphism changes $c$ to $-c$, gives \eqref{eq:rational-signs}.
The integral addition map from the two eigensubgroups has kernel and cokernel killed by $2$; after tensoring the Tate module with $\Q_2$ it becomes an isomorphism. Put $T_S=T_2(S_K^{\rm div})$ and $L_\eta=\ker(c-\eta:T_S\to T_S)$. This is a saturated $\Z_2$-submodule of rank $s^\eta$, so $L_\eta/2^M L_\eta$ embeds in $(S_K^{\rm div}[2^M])^\eta$ as a free $R_M$-module.

\begin{proof}[Proof of Theorem~\ref{thm:Kolyvagin}]
\emph{Upper bounds.}
More generally, suppose standard classes of sign $\epsilon$ and fixed support size $r$ have arbitrarily large order at arbitrarily deep additional precision. We first show $s^{-\epsilon}\le r$. If $s^{-\epsilon}>r$, restrict the $r$ scalar localization coordinates at the primes of a conductor $n$ to $L_{-\epsilon}/2^M L_{-\epsilon}$. Their matrix has $r$ rows and $s^{-\epsilon}>r$ columns over $R_M$, so Smith normal form gives a primitive vector $x_0$ in its kernel, with $\ell(x_0)=M$. Since each scalar coordinate has kernel killed by $2$, the class $x=2x_0$ satisfies
\[
x\in\Sel_{2^M}(E/K)^{-\epsilon},\qquad
\loc_px=0\ (p\mid n),\qquad \ell(x)\ge M-1.
\]
Choose $q$ detecting both $x$ and $\kappa_{n,M}$. Reciprocity applied to $x$ and $\kappa_{nq,M}$ leaves only the $q$ term: at $n$ the first localization is zero, and away from $nq$ both are finite. By \eqref{eq:deep-finite-transverse}, detection, and \eqref{eq:signed-pairing-estimate}, its logarithmic order is at least
\[
\ell(\kappa_{n,M})-1-2c_{\rm ch}-c_{\rm loc}
\]
whenever this is positive, contradicting the existence of classes of arbitrarily large order. Hence $s^{-\epsilon}\le r$.

Detecting $\kappa_{n,M}$ at a fresh $q$ also gives, by \eqref{eq:deep-finite-transverse}, classes $\kappa_{nq,M}$ of sign $-\epsilon$, support $r+1$, and arbitrarily large order, still with arbitrary additional precision. Applying the just-proved bound to this family yields $s^\epsilon\le r+1$. Strong nonvanishing supplies the high orders required here: for any $t$, choose a nonzero class at precision $k+t$ and lift its coefficients using \eqref{eq:coefficient-compatibility}; its lift at level $M+t$ has order at least $2^{t+1}$. Taking $r=f$ proves both upper bounds.

\emph{The lower bound in the principal sign.}
Fix $k\ge k_*$. At level $M$, let $2^h$ be the largest order of a support-$f$ class with conductor in $\mathcal N_{M+k}$, and choose a conductor attaining it. This maximum exists because these classes belong to the finite group $\Sel_{2^M}(E/K)$ by \eqref{eq:minimal-support-selmer}. Label its primes $p_1,\ldots,p_f$ as they are removed in the following construction.

Starting with $\lambda_1=p_1\cdots p_f$, apply Lemma~\ref{lem:prime-replacement} repeatedly, retaining all previously chosen new primes. This gives
\[
\lambda_i=p_i\cdots p_fq_1\cdots q_{i-1},\qquad
\eta_i=\kappa_{\lambda_i,M}\quad(1\le i\le f+1).
\]
At step $i\le f$, the prime $q_i$ detects $\eta_i$ and replaces $p_i$; after the last replacement choose $q_{f+1}$ detecting $\eta_{f+1}$. All new primes lie at depth at least $M+k$. Each replacement loses at most $d_f$ in logarithmic order, and each
chosen detecting coordinate loses at most $c_{\rm ch}$. Hence
\[
\ell(\eta_i)\ge h-(i-1)d_f,\qquad
\ell(\psi_{q_i}^{\epsilon_f}(\eta_i))
\ge h-(i-1)d_f-c_{\rm ch}.
\]
There are $f$ replacements and one final detection. Choosing
$C_f\ge fd_f+c_{\rm ch}$ therefore gives
\[
h-C_f\le\ell(\eta_i)\le h,\qquad
\ell(\psi_{q_i}^{\epsilon_f}(\eta_i))\ge h-C_f
\quad(1\le i\le f+1).
\]
For fixed $f$, the accumulated loss is independent of $M$, $k$, and all
auxiliary primes. The constant can be enlarged to include the bounded
integral eigenspace defects used in the upper-bound argument.
For $j<i$, the prime $q_j$ divides $\lambda_i$, so
$\loc_{q_j}\eta_i=0$ by \eqref{eq:minimal-support-selmer}. Therefore the matrix
\[
\bigl(\psi_{q_j}^{\epsilon_f}(\eta_i)\bigr)_{1\le i,j\le f+1}
\]
is triangular. Every entry is divisible by $2^{M-h}$, because every $\eta_i$ has order at most $2^h$. Divide the entire matrix by this common factor and regard it over $\Z/2^h\Z$. Its diagonal entries have valuations at most $C_f$. If $h>(f+1)C_f$, its determinant has valuation at most $(f+1)C_f$.

Apply the adjugate identity to any relation $\sum_i a_i\eta_i=0$. It implies that every $a_i$ is divisible by $2^{h-(f+1)C_f}$. Thus the subgroup generated by the $\eta_i$ has a quotient
\[
(\Z/2^{h-(f+1)C_f}\Z)^{f+1}.
\]
The maximal orders $h$ are unbounded as $M$ varies, by the coefficient-lifting argument above at the fixed precision $k$. A cofinitely generated $2$-primary group of corank less than $f+1$ cannot have such quotients of its finite subgroups: the $(f+1)$st elementary divisor is bounded by its finite part. Since the principal signed subgroup of $S_K$ has corank $s^{\epsilon_f}$, this proves $s^{\epsilon_f}\ge f+1$. When $f=0$, the same argument consists of one class and one detecting prime.

Combining the bounds proves \eqref{eq:Kolyvagin-signed}; restriction and corestriction in \eqref{eq:rational-signs} give \eqref{eq:Kolyvagin-corank}.
\end{proof}

\subsection{The three group structures and their integral comparison}\label{app:three-group-proof}
\begin{proof}[Proof of Theorem~\ref{thm:full-selmer-structure}]
Hypothesis \Surjectivity{3} gives full $2$-adic image by Lemma~\ref{lem:mod-eight}, and $\intr(E,K)=0$ gives $T_E=0$. Thus Theorem~\ref{thm:Kolyvagin} applies. Residual $S_3$ also gives
$C(F)[2^\infty]=0$ for all three pairs: on restriction to $G_K$, the residual image is $S_3$ or $A_3$, and neither fixes a nonzero vector.

Proposition~\ref{prop:usual-structure}, applied to each pair $(C,F)$, gives a finite abelian $2$-group $A_F$ such that
\[
\Sel_{2^\infty}(C/F)\simeq(\Q_2/\Z_2)^{s_F}\oplus A_F^{\oplus2},
\qquad
\Sha(C/F)[2^\infty]\simeq(\Q_2/\Z_2)^{s_F-r_F}\oplus A_F^{\oplus2}.
\]
Its finite-level assertion gives \eqref{eq:three-finite-structure}. In particular,
$d_2(C)=s_2(C)+2\dim_{\F_2}A_C[2]$ for either rational curve.

By Proposition~\ref{prop:nesting}, intrinsic index zero gives $a_0+a_1=2R+1$. Thus
\[
s_K=s_2(E)+s_2(E^D)=2R+1-2(t_{\rm pr}+t_{\rm op})
\]
is odd. Theorem~\ref{thm:Kolyvagin} says that the principal corank is $f+1$ and the other lies between $0$ and $f$. Oddness of their sum makes the other corank congruent to $f$ modulo $2$, so it is $f-2h$ with $0\le h\le\lfloor f/2\rfloor$. This proves \eqref{eq:principal-full-structure}--\eqref{eq:quadratic-full-structure}. The two rational identities at level $2$ give \eqref{eq:rational-finite-factor-counts}. Finally the divisible part of $\Sha$ vanishes exactly when $r_F=s_F$.
\end{proof}

Write $A_E,A_{E^D}$ for the groups $A_{\rm pr},A_{\rm op}$ assigned to the corresponding curves. The relation between the three finite groups respects the integral quadratic comparison. If $\jmath:E_K^D\to E_K$ is the twisting isomorphism, restriction and corestriction give
\[
\mathcal R(x,y)=\res(x)+\jmath_*\res(y),\qquad
\mathcal N(z)=(\cor(z),\cor(\jmath_*^{-1}z)).
\]
They preserve Selmer conditions and satisfy $\mathcal R\mathcal N=[2]$ and $\mathcal N\mathcal R=[2]$. For the latter identity the cross terms vanish: their restrictions are $(1+c)$ applied to a negative class, or $(1-c)$ applied to a positive class, and restriction is injective since $E(K)[2^\infty]=0$. The maps descend to the finite quotients
$A_E^{\oplus2}\oplus A_{E^D}^{\oplus2}$ and $A_K^{\oplus2}$ with the same identities. This is the integral comparison underlying the rational decomposition \eqref{eq:rational-signs}.

\subsection{The groups of a terminal auxiliary twist}\label{app:terminal-sha-proof}
For this last assertion, use the hypotheses of Proposition~\ref{prop:terminal-sha} and put $C=E^{d_T}$.
\begin{proof}[Proof of Proposition~\ref{prop:terminal-sha}]
The proofs of \eqref{eq:low-rank} and \eqref{eq:sha-rational-vanish} give
\[
\Sha(C/\Q)[2^\infty]=\Sha(C^D/\Q)[2^\infty]=0,
\qquad \rank C(\Q)+\rank C^D(\Q)=1.
\]
The rational eigenspace decomposition then gives $\rank C(K)=1$.

For $x\in\Sha(C/K)[2^\infty]$, the localization of $\cor_{K/\Q}x$ at a rational place $v$ is the sum of the local corestrictions over $w\mid v$. Each term vanishes because $x$ is locally trivial. Thus $\cor_{K/\Q}x\in\Sha(C/\Q)[2^\infty]=0$, and the restriction--corestriction identity gives
\[
(1+c)x=\res\cor(x)=0.
\]
Let $\jmath:C^D_K\to C_K$ be the twisting isomorphism. It and its local counterparts identify $\Sha(C^D/K)[2^\infty]$ with $\Sha(C/K)[2^\infty]$, with $c\jmath=-\jmath c$. Apply the same corestriction argument to $\jmath^{-1}x$ on $C^D$ and transport the result by $\jmath$. This gives $(1-c)x=0$. Adding the two identities yields $2x=0$.

The finite-level Kummer sequence shows that $\Sha(C/K)[2]$ is finite. Hence the entire $2$-primary group is finite and killed by $2$. Since $C(K)[2]=0$ and $\rank C(K)=1$, the same sequence gives
\[
\dim_{\F_2}\Sha(C/K)[2]=d_K(C)-1.
\]
The Cassels--Tate pairing on this finite group is perfect and alternating, so the dimension is nonnegative and even. Apply Proposition~\ref{prop:usual-structure} and \eqref{eq:finite-Sel-structure} to $C$. Here $s_K(C)=1$ and every $e_i=1$, giving the asserted finite-level structures.
\end{proof}

\section*{AI Declaration}
\addcontentsline{toc}{section}{AI Declaration}
The author used GPT-6 Astra to improve the English prose, identify relevant literature, assist with detailed calculation. The author has reviewed the contents of the manuscript and takes responsibility for the final text.


\end{document}